%% file: Entropy_new_version_no_comments.tex
\documentclass[12pt]{amsart}
\usepackage[top=4.2cm, bottom=4cm, left=2.4cm, right=2.4cm]{geometry}
\usepackage[utf8]{inputenc}
\usepackage[USenglish]{babel}
\usepackage[T1]{fontenc} 
\usepackage{mathrsfs}
\usepackage{mathtools}
\usepackage{amsmath}
\usepackage{amssymb,epsfig}
\usepackage{caption}
\usepackage{subcaption}
\usepackage{comment}
\usepackage{soul}
\usepackage[all]{xy}
\usepackage{bbm}
\usepackage{todonotes}
\usepackage{bbm}

\usepackage{amsthm}
\usepackage[absolute]{textpos}
\usepackage{mathtools,emptypage}

\usepackage[bookmarks=true]{hyperref}
\usepackage{xcolor}
\hypersetup{
	colorlinks,
	linkcolor={red!50!black},
	citecolor={blue!50!black},
	urlcolor={blue!80!black},
}

\usepackage{tipa} 
\usepackage{tikz}

\mathtoolsset{showonlyrefs}

\newtheorem{theorem}{Theorem}
\numberwithin{theorem}{section}

\newtheorem{lemma}[theorem]{Lemma}

\newtheorem*{claim*}{Claim}

\newtheorem{proposition}[theorem]{Proposition}

\newtheorem{remark}[theorem]{Remark}

\newtheorem{corollary}[theorem]{Corollary}

\newtheorem*{question*}{Question}

\newtheorem*{T1}{Theorem~\ref{theo_intro:freq_does_not_depend}}

\newtheorem*{T2}{Theorem~\ref{theo:intro_entropy_formula_branched_coverings}}

\newtheorem*{T3}{Theorem~\ref{theo_intro:asymptotic_formula}}

\newtheorem*{T4}{Theorem~\ref{theo_intro:asymptotic_BM_measures}}

\newtheorem*{C1}{Corollary~\ref{cor:GT_manifolds_estimate}}

\theoremstyle{remark}

\theoremstyle{definition}
\newtheorem{definition}[theorem]{Definition}

\newtheorem*{warning*}{Warning}
\newtheorem*{convention*}{Convention}
\newtheorem*{example*}{Example}
\newtheorem*{assumption*}{Assumption}

 \newcounter{mcomments}

 \newcounter{ccomments}

\newcounter{ncomments}

\newcommand{\R}{\mathbb{R}}

\newcommand{\Z}{\mathbb{Z}}

\newcommand{\Geod}{\textup{Geod}}
\newcommand{\LocGeod}{\textup{LocGeod}}
\newcommand{\Br}{\textup{Broken}_\textup{bc}(\Y,\Sigma)}

\renewcommand{\next}{\textup{\texttt{next}}}
\newcommand{\prev}{\textup{\texttt{prev}}}
\newcommand{\internal}{\textup{int}}
\newcommand{\external}{\textup{ext}}

\newcommand{\inj}{\textup{inj}^\perp(\Sigma)}
\newcommand{\brdeg}{\textup{br-deg}}
\newcommand{\freq}{\textup{\texttt{freq}}}

\renewcommand{\d}{{\mathrm d}}
\newcommand{\CATminus}{\mathrm{CAT(-1)}}

\newcommand{\N}{\mathbb{N}}
\newcommand{\sfd}{{\sf d}}                         

\newcommand{\restr}[1]{\lower3pt\hbox{$|_{#1}$}}
\newcommand{\limi}{\varliminf}
\newcommand{\lims}{\varlimsup}                      
\newcommand{\X}{{\rm X}}
\newcommand{\Y}{{\rm Y}}
\newcommand{\W}{{\rm W}}
\newcommand{\mm}{\mathfrak m}

\title[Entropy of branched coverings]{The entropy of Gromov-Thurston manifolds and branched coverings}

\author[Nicola Cavallucci]{Nicola Cavallucci}
\address{Département de mathématiques,
Université de Fribourg,
Ch. du musée 23
CH-1700 Fribourg,
Switzerland}
\email{n.cavallucci23@gmail.com} 

\author{Merlin Incerti-Medici}
\email{merlin.medici@pm.me}
\urladdr{https://www.merlinmedici.ch/}

\begin{document}

\begin{abstract}
    We develop a general theory for the dynamics of the geodesic flow of locally $\textup{CAT}(\kappa)$ branched coverings and we show that it does not depend on the specific covering but only on the base space, the branching set and the degree of the covering. We find an explicit formula for the entropy: it equals the topological pressure of the associated natural dynamical system on the space of broken geodesics with a natural geometric potential. We find the exact asymptotic of the entropy as the number of sheets of the branched covering goes to infinity, as well as asymptotic properties of the measures of maximal entropy.
\end{abstract}

\maketitle

\tableofcontents

\section{Introduction}

Gromov-Thurston manifolds were introduced in \cite{GromovThurston87} to give an example of a sequence of Riemannian manifolds with sectional curvature arbitrarily close to $-1$ that do not carry any hyperbolic metric. The construction starts with a closed, oriented, hyperbolic manifold $\Y$, together with an embedded, oriented, totally geodesic, codimension two submanifold $\Sigma$, which is the boundary of an embedded, oriented, totally geodesic, codimension one submanifold $N$. A pair $(\Y,\Sigma)$ as above is called a \emph{Gromov-Thurston pair}. Given $k\ge 2$, the degree $k$ Gromov-Thurston manifold is the cyclic branched covering of $\Y$ over $\Sigma$. Topologically, it is obtained by taking $k$ distinct copies of $\Y$, cutting them along $N$ in order to obtain two boundary components $N^\pm$ and gluing cyclically the positive boundary of a copy to the negative boundary of the next one. We refer to Section \ref{subsec:GT-manifolds} for more details. 

The cyclic cover carries a natural singular metric, corresponding to the length metric associated to the pull-back of the hyperbolic metric on $\Y \setminus \Sigma$. It is a locally CAT$(-1)$ metric. The resulting metric space will be denoted by $\textup{GT}^k(\Y,\Sigma)$. The Riemannian metrics of \cite{GromovThurston87} with sectional curvature close to $-1$ are obtained via a smoothing of the singular metric, under suitable assumptions on $\Y$ and $\Sigma$. Gromov-Thurston manifolds carry other interesting Riemannian metrics, notably Einstein ones, as proved in \cite{FinePremoselli2020,HamenstadtJackel2024}. However, in this paper we are interested in their canonical singular metric and on the properties of the corresponding geodesic flow. 

More generally, we are interested in the dynamical properties of the geodesic flow of the singular metric on branched coverings with curvature bounded above. More precisely, we consider the following setting.
\begin{assumption*}
    $(\Y,\sfd_\Y)$ is a compact, locally $\textup{CAT}(\kappa)$ space, $\Sigma\subseteq \Y$ is closed and locally convex and $\pi\colon \X \to \Y$ is a finite-degree branched covering with singular set $\Sigma$.
\end{assumption*}
The last condition means that $\pi$ is continuous and surjective with $\Sigma = \{y\in \Y\,:\, \#\pi^{-1}(y)=1\}$, that the restriction $\pi \colon \pi^{-1}(\Sigma) \to \Sigma$ is a homeomorphism and that the restriction $\pi\colon \X\setminus\pi^{-1}(\Sigma) \to \Y\setminus \Sigma$ is a covering. As explained in \cite{Allcock2000}, see also \cite{Gromov-HypGps} and Proposition \ref{prop:branched_covering_loc_CAT-1}, $\X$ carries a natural length metric $\sfd$ that is locally $\textup{CAT}(\kappa)$ as well. Moreover, if the covering $\pi$ has finite degree then $\X$ is compact. In the sequel, we will denote by $(\X,\Y,\Sigma,\pi)$ a compact, locally $\textup{CAT}(\kappa)$ branched covering as above. It has \emph{pure-degree} $k$ if the covering $\pi$ has degree $k$ on $\X\setminus \pi^{-1}(\Sigma)$. For more details we refer to Section \ref{sec:top_branched_coverings}. 

The most natural class of examples of branched coverings is provided by wedges. More precisely, given a compact, locally $\textup{CAT}(\kappa)$ space $\Y$, a closed, locally convex subset $\Sigma$ and a number $k\in \N$, the space $\vee^k(\Y,\Sigma)$ is obtained by gluing $k$ copies of $\Y$ along $\Sigma$. Both wedges and Gromov-Thurston manifolds $\textup{GT}^k(\Y,\Sigma)$ have pure-degree $k$. Wedges and Gromov-Thurston manifolds are clearly different; however our main results will show that, from a dynamical perspectives, they are essentially identical.

The geometries of the base space $\Y$ and the branched covering space $\X$ are closely related. The basic observation is that every local geodesic $\gamma$ of $\X$ can be canonically subdivided into a countable number of pieces $\gamma_j$ where each $\gamma_j$ is either entirely contained in $\Sigma$ or it is a local geodesic that meets $\Sigma$ only at the endpoints. The projection $\pi\circ \gamma_j$ of each piece is a local geodesic of $\Y$, due to the covering properties of $\pi$. However, the concatenation of two consecutive projected segments $\pi \circ \gamma_j$ and $\pi \circ \gamma_{j+1}$ does not necessarily yield a local geodesic of $\Y$. For this reason, in Section \ref{subsec:branched_coverings_definition} we introduce the space of $\Sigma$-broken local geodesic lines $\textup{Broken}(\Y,\Sigma)$. A $\Sigma$-broken local geodesic line in $\Y$ is simply a (possibly countable) concatenation of local geodesics of $\Y$ such that each piece is either entirely contained in $\Sigma$ or it meets $\Sigma$ only at the endpoints. The second pieces are the most interesting ones: in the sequel they will be called \emph{external}. 
Basic instructive examples, see for instance the ones in Section \ref{subsec:Examples}, illustrate that not every $\Sigma$-broken local geodesic line comes as the projection of some local geodesic line of $\X$. For this reason we introduce the space $\textup{Broken}(\Y,\Sigma;\X)$, which is simply the set of $\Sigma$-broken geodesics of $\Y$ that are projections of some local geodesic lines in $\X$.
Our first result says that this set does not depend on the branched covering: it is a property of the couple $(\Y,\Sigma)$.
\begin{theorem}
\label{theo_intro:broken_does_not_depend}
    Let $(\X,\Y,\Sigma,\pi)$, $(\X',\Y,\Sigma,\pi')$ be two compact, locally $\textup{CAT}(\kappa)$ branched coverings over the same couple $(\Y,\Sigma)$. Then $\textup{Broken}(\Y,\Sigma;\X) = \textup{Broken}(\Y,\Sigma;\X').$
\end{theorem}

This special set of $\Sigma$-broken local geodesic line will be denoted by $\Br$, where the subscript stands for "branched covering". Indeed, a purely geometric characterization of those $\Sigma$-broken local geodesic lines $\eta$ that belong to $\Br$ is given in Theorem \ref{theo:characterization_broken_geodesics}. The condition can be checked locally at every concatenation time $t$ between two consecutive pieces $\eta_j,\eta_{j+1}$:
\begin{itemize}
    \item[(i)] either the concatenation $\eta_j\star \eta_{j+1}$ is a local geodesic in $\Y$,
    \item[(ii)] or the concatenation point is the unique point $z$ on $\Sigma$ that minimizes the quantity $\sfd(\eta(t-s), z) + \sfd(\eta(t+s),z)$, for $s$ small enough.
\end{itemize}
In the first case we say that the intersection is \emph{non-branching} if $\eta_{j+1} \subseteq \Sigma$ and \emph{extremal} otherwise. An intersection as in (ii) is called \emph{non-extremal}. Non-extremal geodesics should be thought as a light beam that meets $\Sigma$ and is specularly reflected.

Now, every external local geodesic of $\Y$ has more than one lift, due to branching. When we have the concatenation of two local geodesics $\sigma \star \sigma'$ in $\Y$ we would like to count, given a lift $\tilde{\sigma}$ of $\sigma$ to $\X$, how many lifts $\tilde{\sigma}'$ of $\sigma'$ to $\X$ form a local geodesic when concatenated with $\tilde{\sigma}$. For that, we introduce the \emph{geodesic continuation frequency} functions $\freq^\X_m\colon \Br \to \N$, for $m\ge 1$. Informally, given $\eta \in \Br$, the function $\freq_m^\X$ counts the number of times $t$ in $(-1, 0]$ for which there are exactly $m$ different ways to continue a lift of $\eta\vert_{(-1,t]}$ as a lift. For branched coverings of pure-degree $k$, like $k$-wedges and Gromov-Thurston manifolds, these functions are easier to understand: $\freq_1^\X$, $\freq_{k-1}^\X$ and $\freq_k^\X$ count the number of non-branching, non-extremal and extremal intersections in $(-1,0]$, respectively. In the general case one has to consider that different local geodesics may have a different number of lifts. Our second result shows that these functions actually depend only on the degree of the pure-degree branched covering.

\begin{theorem}
\label{theo_intro:freq_does_not_depend}
    Let $(\X,\Y,\Sigma,\pi)$, $(\X',\Y,\Sigma,\pi')$ be two compact, pure-degree $k$, locally $\textup{CAT}(\kappa)$ branched coverings over the same couple $(\Y,\Sigma)$. Then $\freq_m^\X = \freq_m^{\X'}$ for every $m\in \N$. Moreover, $\freq_m^\X \equiv 0$ for $m\in \{2,\ldots, k-2\}$.
\end{theorem}

One of the main motivations for this paper came from finding a formula for the topological entropy of the geodesic flow of a branched covering. Given a metric space $(\X,\sfd)$, we denote by $\LocGeod(\X)$ the space of local geodesics $\gamma \colon \R \to \X$, together with the natural reparametrization flow $\Phi_t(\gamma)(\cdot) := \gamma(\cdot + t)$ for every $t\in \R$. The dynamical system $(\LocGeod(\X),\Phi_1)$ is called the geodesic flow of $\X$. If $\X$ is compact and locally $\textup{CAT}(\kappa)$ then $\LocGeod(\X)$ is compact, see Proposition \ref{prop:loc_geod_is_compact}. One of the most interesting invariants associated to the geodesic flow is its topological entropy, which is denoted by $h_\textup{top}(\LocGeod(\X),\Phi_1)$. When $\kappa \le -1$, it coincides with the critical exponent $h(\tilde{\X};\pi_1(\X))$ of the fundamental group of $\X$, see \cite{Ricks2021}. In this notation, $\tilde\X$ stands for the universal cover of $\X$. We refer to Section \ref{subsec:locally_CAT} for more details on the definitions.

As said before, we have a well-defined projection map $\Pi \colon \LocGeod(\X) \to \Br$ which is surjective (by definition), continuous and satisfies $\Pi\circ \Phi_1 = \Phi_1\circ \Pi$, where $\Phi_1$ is the time-one map of the reparametrization flow on $\Br$. We refer to Proposition \ref{prop:projection_geodesics_to_broken_geodesics} for more details.
Using Ledrappier-Walters' Theorem (see \cite{LedrappierWalters1977}) we express the topological entropy of the geodesic flow on $\X$ as a topological pressure on the dynamical system $(\Br,\Phi_1)$, with potential involving the functions $\freq_m^\X$.

\begin{theorem}
\label{theo:intro_entropy_formula_branched_coverings}
    Let $(\X,\Y,\Sigma,\pi)$ be a compact, locally $\textup{CAT}(\kappa)$ branched covering. Then
    $$h_\textup{top}(\LocGeod(\X),\Phi_1) = \mathscr{P}\left(\Br,\Phi_1,\sum_{m\in \N}\log(m)\freq_m^\X\right).$$
\end{theorem}

The quantity on the right hand side is a topological pressure: more precisely it is defined as
\begin{equation}
    \label{eq:top_pressure_intro}
    \sup\left\{h_\mu + \int\sum_{m\in \N}\log(m)\freq_m^\X\,:\, \mu \in \mathcal{M}_1(\Br,\Phi_1)\right\},
\end{equation}
where $\mathcal{M}_1(\Br,\Phi_1)$ denotes the set of $\Phi_1$-invariant probability measures on the set $\Br$ and where $h_\mu$ is the Kolmogorov-Sinai entropy of the measure $\mu$. 
We refer to Section \ref{subsec:dynamical_systems_entropy_pressure} for more details. 

Theorem \ref{theo:intro_entropy_formula_branched_coverings} gives an explicit description of the topological entropy of the geodesic flow of a branched covering in terms of a topological pressure on $\Br$. Combining Theorem \ref{theo:intro_entropy_formula_branched_coverings} with Theorem \ref{theo_intro:broken_does_not_depend} we get
\begin{theorem}
\label{theo:intro_entropy_does_not_depend}
    Every two compact, locally $\textup{CAT}(\kappa)$, branched coverings  $(\X,\Y,\Sigma,\pi)$, $(\X',\Y,\Sigma,\pi')$ of pure degree $k$ satisfy
    $$h_\textup{top}(\LocGeod(\X),\Phi_1) = h_\textup{top}(\LocGeod(\X'),\Phi_1).$$
\end{theorem}

In particular, the topological entropy of the geodesic flow of $\textup{GT}^k(\Y,\Sigma)$ coincides with the topological entropy of the geodesic flow of the $k$-wedge $\vee^k(\Y,\Sigma)$, for every Gromov-Thurston pair $(\Y,\Sigma)$.

Now, to every couple $(\Y,\Sigma)$, where $(\Y,\sfd_\Y)$ is a compact, locally $\textup{CAT}(\kappa)$ space and $\Sigma \subseteq \Y$ is a closed, locally convex subset we can attach the following invariants:
\begin{itemize}
    \item[(i)] the space $\Br$.
    \item[(ii)] The functions $\freq_{\textup{non-br}},\freq_{\textup{non-ext}}, \freq_{\external}\colon \Br \to \N$ that count the number of non-branching, non-extremal and extremal intersections in $(-1,0]$, respectively. By what we said above, for every pure-degree $k$ branched covering $(\X,\Y,\Sigma,\pi)$ it holds that $\freq_1^\X = \freq_{\textup{non-br}}$, $\freq_{k-1}^\X = \freq_\textup{non-ext}$ and $\freq_{k}^\X = \freq_\external$.
    \item[(iii)] For every $k\ge 2$, the number $h_\textup{top}(\Y,\Sigma,k)$ which is the topological entropy of the geodesic flow of any compact, pure-degree $k$ branched covering $(\X,\Y,\Sigma,\pi)$: it is well defined by Theorem \ref{theo:intro_entropy_does_not_depend}.
\end{itemize}

The next natural question consists in understanding the asymptotic behavior of the sequence of numbers $\{h_\textup{top}(\Y,\Sigma,k)\}$ as $k$ goes to infinity. The new perspective provided by Theorems \ref{theo:intro_entropy_formula_branched_coverings} and \ref{theo_intro:broken_does_not_depend} gives the answer. Indeed, $h_\textup{top}(\Y,\Sigma,k)$ can be found by solving "different optimization problems" in the same space $\Br$. Moreover, the different optimization problems are of the form
$$\mathscr{P}\left(\Br,\Phi_1,\log(k-1)\freq_{\textup{non-ext}} + \log(k)\freq_\external\right).$$
For $k$ going to infinity, two things happen. From one side, the potential functions 
$$\log(k-1)\freq_{\textup{non-ext}} + \log(k)\freq_\external$$ become approximately of the form $\log(k)\freq$, where $\freq = \freq_{\textup{non-ext}} + \freq_\external$ counts the number of intersections that give rise to branching. From the other side, the potential functions dominate the term we want to maximize in the right hand side of \eqref{eq:top_pressure_intro}, provided that two things happen: the Kolmogorov-Sinai entropies are uniformly bounded, which means $h_\textup{top}(\Br,\Phi_1) < \infty$, and $\int \freq \,\d\mu > 0$ for at least one $\mu \in \mathcal{M}_1(\Br,\Phi_1)$.

Under these two assumptions, taking the limit for $k\to +\infty$ corresponds to the "low-temperature case" in the classical thermodynamic formalism, and the limit value tends towards what is called the ergodic optimization of the function $\freq$, see Section \ref{subsec:ergodic_optimization}. A geometric argument allows to compute this value in terms of the normal injectivity radius of $\Sigma$. This number is defined as $\inj := \inf\{\ell(\sigma)\,:\, \sigma \in \Sigma^\circlearrowleft\}$, where
$$\Sigma^\circlearrowleft := \{\sigma\colon I \to \Y \textup{ local geodesic }:\, \partial I = \sigma^{-1}(\Sigma)\},$$
i.e. is the set of local geodesics with exactly the endpoints in $\Sigma$.
\begin{theorem}
    \label{theo_intro:asymptotic_formula}
    Let $(\Y,\sfd_\Y)$ be a compact, locally $\textup{CAT}(\kappa)$ space and let $\Sigma \subseteq \Y$ be closed and locally convex. Assume that $h_\textup{top}(\Br,\Phi_1) < \infty$ and that $\Sigma^\circlearrowleft \neq \emptyset$. Then for every $k\ge 2$ it holds that
    $$\frac{\log(k-1)}{\inj}\le h_\textup{top}(\Y,\Sigma,k) \le \frac{\log(k)}{\inj} + h_\textup{top}(\Br,\Phi_1),$$
    hence
    $$\lim_{k\to +\infty} \frac{h_\textup{top}(\Y,\Sigma,k)}{\log(k)} = \frac{1}{\inj}.$$
    If $\kappa=-1$, then $h_\textup{top}(\Br,\Phi_1) \le  h(\tilde{\Y};\pi_1(\Y)) + \log(2)/\inj$ (we recall that $h(\tilde{\Y};\pi_1(\Y))$ denotes the critical exponent of the group $\pi_1(\Y)$ acting on the universal cover $\tilde\Y$ of $\Y$).
\end{theorem}

The authors are not aware of any examples where $h_\textup{top}(\Br,\Phi_1) = \infty$. The condition $\Sigma^\circlearrowleft \neq \emptyset$ is discussed in detail in Section \ref{subsec:locally_convex_subsets} and, when $\kappa = 0$, it is equivalent to the fact that the inclusion of the $\pi_1(\Sigma)$ into $\pi_1(\Y)$ is not surjective: it is a non-triviality condition.

In Theorem \ref{theo_intro:asymptotic_formula}, the asymptotic is independent of the curvature upper bound $\kappa$, suggesting that, asymptotically, the whole entropy contribution is concentrated along the singular set $\Sigma$ and arises from the gluing procedure. Moreover, we emphasize once again that this result is independent of the particular type of branched covering, so it applies equally to the wedges $\vee^k(\Y,\Sigma)$ and to Gromov-Thurston manifolds $\textup{GT}^k(\Y,\Sigma)$. In this last case, using Margulis type estimates we can express the additive constant of Theorem \ref{theo_intro:asymptotic_formula} in terms of the dimension and the diameter.

\begin{corollary}
\label{cor:GT_manifolds_estimate}
    Let $n\in \N$ and $D\ge 0$. Then there exists $E(n,D)\ge 0$ such that for every Gromov-Thurston pair $(\Y,\Sigma)$ with $\textup{dim}(\Y) = n$ and $\textup{Diam}(\Y) \le D$ it holds that
    $$\left\vert h_\textup{top}(\LocGeod(\textup{GT}^k(\Y,\Sigma)), \Phi_1) - \frac{\log(k)}{\inj}\right\vert \le  E(n,D).$$
\end{corollary}

When the curvature bound satisfies $\kappa = -1$ and $\Y$ has finite topological dimension, there is a unique measure of maximal entropy for the dynamical system $(\LocGeod(\X), \Phi_1)$: it is called the Bowen-Margulis measure and it is denoted by $\mm_\textup{BM}^\X$. The Bowen-Margulis measures associated to two different normal branched coverings over the couple $(\Y,\Sigma)$ of pure degree $k$ project to the same measure $\mm_k$ on $(\Br,\Phi_1)$, see Theorem \ref{theo:wedges_maximal_entropy}. So, we have another sequence of invariants associated to a couple $(\Y,\Sigma)$ with $\Y$ compact and locally $\CATminus$ and $\Sigma \subseteq \Y$ closed and locally convex: the sequence of measures $\{\mm_k\} \subseteq \mathcal{M}_1(\Br,\Phi_1)$.

In Section \ref{subsec:dynamical_properties} we prove some dynamical properties of the measures $\mm_k$, and correspondingly of the Bowen-Margulis measures on every pure-degree branched covering. For instance we show that $\mm_k$-a.e. $\eta \in \Br$ has infinitely many non-extremal intersections, with reflection angle as close to zero as we want. Under minimal assumptions we show that $\mm_k$-a.e. $\eta \in \Br$ has arbitrarily long external subsegments. 

For the purpose of this introduction we focus only on the asymptotic properties of the sequence $\{\mm_k\}$. Here we denote by $\mm_\infty$ any of its sublimit with respect to the weak convergence. We denote by $\textup{Broken}_{\textup{bc}}^\perp$ the set of broken geodesics $\eta \in \Br$ with only external pieces, each of which of length exactly equal to $\inj$. In other words, $\eta$ is the concatenation of local geodesics leaving and coming back to $\Sigma$ in the shortest possible way.

\begin{theorem}
    \label{theo_intro:asymptotic_BM_measures}
    Let $(\Y,\sfd_\Y)$ be a compact, locally $\textup{CAT}(-1)$ space  and let $\Sigma \subseteq \Y$ be a closed, locally convex subset such that $\Sigma^\circlearrowleft \neq \emptyset$. Then $\mm_\infty$ is supported on $\textup{Broken}_{\textup{bc}}^\perp(\Y,\Sigma)$. 
\end{theorem}

In Section \ref{subsec:asymptotic_BM} we show that the dynamical system $(\textup{Broken}_{\textup{bc}}^\perp(\Y,\Sigma), \Phi_1)$ is dynamically simple, being essentially conjugated to a product of a subshift of finite type with the interval. We expect that the sequence $\{\mm_k\}$ is actually convergent and that $\mm_\infty$ is a measure of maximal entropy of $(\textup{Broken}_{\textup{bc}}^\perp(\Y,\Sigma), \Phi_1)$. We will discuss possible strategies for approaching this question and relative difficulties at the end of Section \ref{subsec:asymptotic_BM}.

\subsection{Comparison with the literature}
In \cite[Théorème 3.1.3]{Stocker19}, a non-optimal lower bound for $h_\textup{top}(\LocGeod(\textup{GT}^{2k}(\Y,\Sigma), \Phi_1))$ was obtained. Namely, he proved that for every $k\ge 1$ it holds that
\begin{equation}
    h_\textup{top}(\LocGeod(\textup{GT}^{2k}(\Y,\Sigma), \Phi_1)) \ge n-1 + C\log(k),
\end{equation}
where $C$ is a constant depending on $\textup{GT}^{2}(\Y,\Sigma)$.
Comparing with Theorem \ref{theo_intro:asymptotic_formula}, we notice that asymptotically, i.e. for $k$ very large, this lower bound is probably not optimal since the constant $C$ is quite involved and likely smaller than $\inj$. Our study has been influenced by some of the constructions in \cite{Stocker19}, but we performed a more careful analysis. There are no known results in the literature regarding any upper bound on the topological entropy of the geodesic flow of Gromov-Thurston manifolds.

Besides symmetric spaces and trees (\cite{Lim2008}), there are no explicit examples where the topological entropy of the geodesic flow of a locally $\textup{CAT}(0)$ space is computed. The closest case is \cite{LedrappierLim2010}, where for compact quotients of regular Euclidean or hyperbolic buildings it is expressed as a certain topological pressure over an apartment, having as potential a suitable geometric function. Our Theorem \ref{theo:intro_entropy_formula_branched_coverings} has a similar structure. Both are based on \cite{LedrappierWalters1977}.

\subsection{Structure of the paper}

The paper is organized as follows. In Section \ref{sec:preliminaries} we recall some preliminaries of metric geometry and ergodic theory. In Section \ref{subsec:locally_CAT}, we recall and prove some relevant facts about the geometry of $\mathrm{CAT}(\kappa)$ spaces, especially their geodesic flow and locally convex subsets. In Section \ref{sec:top_branched_coverings} we introduce branched coverings, the space of broken geodesics and the geodesic continuation frequency functions. Here we will also prove Theorems \ref{theo_intro:broken_does_not_depend} and \ref{theo_intro:freq_does_not_depend}. In Section \ref{sec:examples} we will provide examples of branched coverings, like wedges and Gromov-Thurston manifolds. Other construction will be briefly discussed at the end of the section. In Section \ref{sec:entropy_branched_covering} we compute the entropy of branched coverings proving Theorems \ref{theo:intro_entropy_formula_branched_coverings} and \ref{theo:intro_entropy_does_not_depend}.
In Section \ref{sec:asymptotic} we prove Theorem \ref{theo_intro:asymptotic_formula}. In Section \ref{sec:asymptotic_BM} we study the dynamical properties of the Bowen-Margulis measure and we prove Theorem \ref{theo_intro:asymptotic_BM_measures}.

\subsection{Acknowledgments}
The authors thank Corey Bregman for many interesting and helpful discussions in the early part of the project and Pierre Pansu for answering our questions about Gromov-Thurston manifolds. We also thank Gilles Courtois and Fr\'ed\'eric Paulin for independently pointing us towards Arnaud Stocker's PhD-thesis and Gérard Besson for providing us the manuscript.

\section{Preliminaries}
\label{sec:preliminaries}
Let $(\X, \sfd)$ be a metric space. Given a point $x \in \X$ and a real number $r\ge 0$, we denote by $B(x,r)$ the open ball of center $x$ and radius $r$. 

Let $B\subseteq \X$. A subset $S$ of $B$ is called {\em $r$-separated} if $\sfd(y,y') > r$ for all $y,y'\in S$, while it is called {\em $r$-dense} if for all $y \in B$ there exists $z\in S$ such that $\sfd(y, z) \le r$. We define $\textup{Pack}_\sfd(B,r) \in \N\cup \{\infty\}$ as the maximal cardinality of a $r$-separated subset of $B$ and $\textup{Cov}_\sfd(B,r) \in \N\cup \{\infty\}$ as the minimal cardinality of a $r$-dense subset. When the metric $\sfd$ is clear from the context, we just omit it. The two quantities above are classically related, more precisely:
	\begin{equation}
		\label{eq:Cov_Pack}
		\textup{Pack}_\sfd(B,2r) \le \textup{Cov}_\sfd(B,2r) \le \textup{Pack}_\sfd(B,r)
	\end{equation}
	for every $B\subseteq \X$ and every $r>0$.

\subsection{Curves}
\label{subsec:curves}

Given an interval $I\subseteq \R$, we denote by $\inf I, \min I, \sup I, \max I$ the infimum, minimum, supremum and maximum of $I$, respectively. We denote by $\textup{Int}(I)$ the interior of $I$.

Let $(\X,\sfd)$ be a metric space. A curve is a continuous map $\gamma\colon I \to \X$, where $I$ is an interval of $\R$. If $I$ contains its minimum then we define the \emph{starting point of $\gamma$} as $\alpha(\gamma) := \gamma(\min I)$. Similarly, if $I$ contains its maximum, we define the \emph{ending point of $\gamma$} as $\omega(\gamma) := \gamma(\max I)$. 
Given a curve $\gamma\colon I \to \X$ we define the curve $-\gamma\colon I\to \X$, $-\gamma(t) := \gamma(\sup I + \inf I - t)$, i.e. the curve $\gamma$ followed in the opposite direction. The length of a curve $\gamma$ is denoted by $\ell(\gamma)$.

\begin{definition}[Concatenation of curves]
\label{defin:concatenation_of_curves}    A subset $J\subseteq \Z$ is called an \emph{interval} if for every integers $j \le j'\le j''$ such that $j,j''\in J$ then $j' \in J$ as well. Let $J\subseteq \Z$ be an interval and let $\{\gamma_j\colon I_j \to \X\}_{j\in J}$ be a set of curves. If $\#J \ge 2$ we assume the following conditions.
\begin{itemize}
    \item[-] $I_j = [0,L_j]$ if $j\notin\{\inf J, \sup J\}$.
    \item[-] $I_{\inf J}$ is either as above or of the form $(-\infty,0]$. Similarly, $I_{\sup J}$ is either as above or of the form $[0,+\infty)$.
    \item[-] $\omega(\gamma_j) = \alpha(\gamma_{j+1})$ for every $j \in J \setminus \{\sup J\}$.
\end{itemize}
We set $T_0 = 0$, $T_j = \sum_{i=0}^{j-1} L_i$ if $j > 0$ and $T_j := -\sum_{i=j}^{-1} L_i$ if $j<0$.
We define the \emph{concatenation of the curves $\{\gamma_j\}_{j\in J}$} to be
$$\underset{j\in J}{\star}\gamma_j \colon [T_{\inf J}, T_{\sup J}] \to \X, \quad t\mapsto \gamma_j(t - T_j) \text{ for } t\in [T_j,T_{j+1}],$$
with the obvious extensions when $T_{\inf J} = -\infty$ or $T_{\sup J} = +\infty$.

\end{definition}


\subsection{Dynamical systems, entropy and pressure}
\label{subsec:dynamical_systems_entropy_pressure}

A dynamical system is a couple $(\W,T)$ where $\W$ is a topological space and $T\colon \W \to \W$ is a continuous map. For us, $\W$ will always be a compact, metrizable space. We denote by $\mathscr{B}$ the $\sigma$-algebra of Borel subsets of $\W$ and by $\mathcal{M}_1(\W,T)$ the set of probability measures on $(\W,\mathscr{B})$ that are $T$-invariant, i.e. such that $T_\#\mu = \mu$. A measure $\mu\in \mathcal{M}_1(\W,T)$ is ergodic if the $\mu$-measure of every Borel subset $A$ of $\W$ such that $T^{-1}A \subseteq A$ is either $0$ or $1$. The set of ergodic measures is denoted by $\mathcal{E}_1(\W,T)$. The measure-theoretic entropy of the dynamical system $(\W,T)$ is by definition:
	\begin{equation}
		\label{eq:defin_top_entropy_via_measures}
		h_\text{meas}(\W,T) := \sup_{\mu\in \mathcal{M}_1(\W,T)}h_\mu(\W,T),
	\end{equation}
	where $h_\mu(\W,T)$ is the Kolmogorov-Sinai entropy of the measure $\mu$. We refer to \cite[\S 4.4]{Walters2000} for the definition. When it is clear from the context, we will simply write $h_\mu$ in place of $h_\mu(\W,T)$. A measure realizing the supremum in \eqref{eq:defin_top_entropy_via_measures} is said to be \emph{of maximal entropy}.
    
	In \cite{Bow73}, R. Bowen introduced the \emph{topological entropy} of $(\W,T)$. It is defined as
	\begin{equation}
		\label{eq:defin-Bowen-entropy}
		h_\textup{top}(\W,T) =\lim_{r\to 0}\lims_{n \to +\infty}\frac{1}{n} \log \textup{Cov}_{\text{\texthtd}^n}(\W,r),
	\end{equation}
	where $\text{\texthtd}$ is any metric inducing the topology of $\W$ and where $\text{\texthtd}^n$ is the \emph{dynamical metric}
	\begin{equation}
		\label{eq:defin_dynamical_metric}
		\text{\texthtd}^n(y,y') := \max_{i=0,\ldots,n-1} \text{\texthtd}(T^iy,T^iy')
	\end{equation}
	and $\textup{Cov}_{\text{\texthtd}^n}(\W,r)$ denotes the minimal number of balls of radius $r$, with respect to the metric $\text{\texthtd}^n$, needed to cover $\Y$. This quantity does not depend on the choice of $\textup{\texthtd}$. For a subset $B\subseteq \Y$ which is not necessarily $T$-invariant we define 
\begin{equation}
    \label{eq:defin_top_entropy_fibers}
    h_\textup{top}(B,T) := \lim_{r\to 0} \lims_{n\to +\infty} \frac{1}{n}\log \textup{Cov}_{\text{\texthtd}^n}(B,r),
\end{equation}
where $\text{\texthtd}$ is any metric inducing the topology of $\W$.
The variational principle (\cite[Theorem 8.6]{Walters2000}) says the following.
	\begin{proposition}
		\label{prop:variational_principle}
		Let $(\W,T)$ be a compact, metrizable dynamical system. Then $h_\mu \le h_\textup{top}(\W,T)$ for every $\mu\in \mathcal{M}_1(\W,T)$. Moreover, $h_\textup{meas}(\W,T) = h_\textup{top}(\W,T)$. 
	\end{proposition}

    The map $\mathcal{M}_1(\W,T) \ni \mu \mapsto h_\mu \in [0,\infty]$ is not continuous in general, where on $\mathcal{M}_1(\W,T)$ we consider the topology of weak convergence. It is upper semicontinuous in case the compact dynamical system is $h$-expansive (see \cite{Bowen1972}), i.e. if there exists $\varepsilon > 0$ such that 
    $$\sup_{w\in \W} h_\textup{top}(\{w'\in \W\,:\, \textup{\texthtd}(T^nw,T^nw')\le \varepsilon \textup{ for every } n\in \N\} = 0.$$
    This property does not depend on the choice of the metric $\textup{\texthtd}$ inducing the topology of $\W$.
    \begin{proposition}[{\cite{Bowen1972}}]
    \label{prop:measure_entropy_upper_semicontinuous}
        Let $(\W,T)$ be a compact, metrizable, $h$-expansive dynamical system and let $\mu_n,\mu \in \mathcal{M}_1(\W,T)$ be such that $\mu_n \to \mu$. Then
        $$h_\mu \ge \lims_{n\to +\infty} h_{\mu_n}.$$
    \end{proposition}

    Let $(\W,T)$, $({\rm Z},S)$ be two dynamical systems. We say that $({\rm Z},S)$ is a factor of $(\W,T)$ if there exists a continuous, surjective map $\pi\colon \W \to {\rm Z}$ such that $\pi \circ T = S \circ \pi$. We say that $(\W,T)$ and $({\rm Z},S)$ are conjugated if there is a map $\pi$ as above which is a homeomorphism. The next result is a straightforward application of the definitions. 

    \begin{lemma}
\label{lemma:pushforward_is_invariant}
    Let $(\W,T), ({\rm Z},S)$ be two compact dynamical systems and let $\pi \colon \W \to {\rm Z}$ be a continuous, surjective map such that $\pi \circ T = S \circ \pi$. If $\mu \in \mathcal{M}_1(\W,T)$ then $\pi_\#\mu \in \mathcal{M}_1({\rm Z},S)$.
\end{lemma}

The main result of \cite{LedrappierWalters1977} will be fundamental for our theorems.

\begin{proposition}[{\cite{LedrappierWalters1977}}]
\label{prop:Ledrappier_entropy_fibers}
    Let $(\W,T), ({\rm Z},S)$ be two compact dynamical systems and let $\pi \colon \Y \to {\rm Z}$ be a continuous, surjective map such that $\pi \circ T = S \circ \pi$. Then, for every $\mu \in \mathcal{M}_1({\rm Z},S)$ we have that
    $$\sup_{\substack{\nu \in \mathcal{M}_1(\W,T) \\ \pi_\#\nu = \mu}} h_\nu(\W,T) = h_\mu({\rm Z},S)  + \int_{{\rm Z}}h_\textup{top}(\pi^{-1}(z),T)\,\d\mu(z).$$
\end{proposition}

Observe that part of the statement consists in saying that the supremum in the left-hand side is over a non-empty set.
The definition of $h_\textup{top}(\pi^{-1}(z),T)$ is
as in \eqref{eq:defin_top_entropy_fibers}.

The notion of topological entropy can be generalized to the one of topological pressure by adding a scalar potential. Let $(\W,T)$ be a dynamical system and let $f\colon \W \to [0,+\infty)$ be a measurable function. The \emph{topological pressure} of the dynamical system $(\W,T)$ with respect to $f$ is the quantity
\begin{equation}
    \label{eq:defin_top_pressure}
    \mathscr{P}(\W,T,f) := \sup_{\mu \in \mathcal{M}_1(\W,T)} \left( h_\mu + \int f \,\d\mu \right).
\end{equation}
When $f \equiv 0$ then we obtain the quantity defined in \eqref{eq:defin_top_entropy_via_measures}. 



\subsection{Ergodic optimization}
\label{subsec:ergodic_optimization}
Given a dynamical system $(\W,T)$, a function $f\colon \W \to \R$ and an integer $n\in \N$, the function $A_n(f) := \frac{1}{n}\sum_{i=0}^{n-1} (f\circ T^i)$ is called \emph{Birkhoff $n$-average} of $f$. We recall the statement of the Birkhoff's ergodic theorem. 

\begin{theorem}
\label{theo:Birkhoff}
    Let $(\W,T)$ be a dynamical system, let $\mu\in \mathcal{M}_1(\W,T)$ and let $f\in L^1(\mu)$. Then the limit 
    $$\lim_{n\to +\infty}A_n(f)(w) = \lim_{n\to +\infty}\frac{1}{n}\sum_{i=0}^{n-1} (f\circ T^i)(w) =: \bar{f}(w)$$
    exists for $\mu$-a.e. $w\in \W$. Moreover, $\bar{f}$ is $T$-invariant and $\int f \,\d \mu = \int \bar{f}\,\d\mu$. If $\mu$ is ergodic then $\bar{f}$ is $\mu$-a.e. constant.
\end{theorem}

Let $(\W,T)$ be a dynamical system and let $f\colon \Y \to \R$ be a bounded function, so that $f\in L^1(\mu)$ for every $\mu \in \mathcal{M}_1(\W,T)$. The ergodic optimization problem consists in finding the following supremum:
$$\textup{Erg-Opt}(f) := \sup_{\mu \in \mathcal{M}_1(\W,T)} \int f \,\d\mu = \sup_{\mu \in \mathcal{E}_1(\W,T)} \int f \,\d\mu.$$
The number $\textup{Erg-Opt}(f)$ is called the \emph{maximum ergodic average} of $f$. When the function $f$ is upper semicontinuous then there exists at least one measure realizing the supremum (\cite[Proposition 2.4]{Jenkinson2006}) in the definition of $\textup{Erg-Opt}(f)$, and 
$$\textup{Erg-Opt}(f) = \sup_{w\in \W} \lims_{n\to \infty} A_n(f)(w),$$
see \cite[Proposition 2.2]{Jenkinson2006}. Observe that the right-hand side is a purely topological object. The functions we will consider in this paper are typically not upper semicontinuous. In this generality, we recover at least two obvious inequalities. One of them is expressed in terms of periodic orbits. A point $w\in \W$ is said to be \emph{periodic} if there exists $n\in \N$ such that $T^n w = w$. The minimal $n\in \N$ with this property is called the \emph{period} of $w$ and it is denoted by $\textup{per}(w)$. Therefore we have $\bigcup_{n\in \N} T^n w = \{w,Tw,\ldots,T^{\textup{per}(w)-1}w\}$. To every periodic point $w$ we associate the measure 
$$\mu_w := \frac{1}{\textup{per}(w)}\sum_{i=0}^{\textup{per}(w)-1} \delta_{T^iw},$$
where $\delta_{T^iw}$ is the Dirac measure at $T^iw$. By definition, $\mu_w\in \mathcal{M}_1(\W,T)$.

\begin{proposition}
\label{prop:ergodic_optimization_bounds}
    Let $(\W,T)$ be a dynamical system and let $f\colon \W \to \R$ be bounded. Then
    $$\sup_{w \textup{ periodic}}\lims_{n\to \infty} A_n(f)(w)\le \textup{Erg-Opt}(f) \le \sup_{w\in \W} \lims_{n\to \infty} A_n(f)(w).$$
\end{proposition}
\begin{proof}
    For every $w \in \W$ periodic, we have 
    $$\int f \,\d\mu_w = \frac{1}{\textup{per}(w)}\sum_{i=0}^{\textup{per}(w)-1}f(T^i(w)).$$
    It is straightforward to check that $\frac{1}{\textup{per}(w)}\sum_{i=0}^{\textup{per}(w)-1}f(T^i(w)) = \lims_{n\to \infty} \frac{1}{n} \sum_{i=0}^{n-1} (f\circ T^i)(w)$. This gives the first inequality. On the other hand, let $\mu \in \mathcal{E}_1(\W,T)$ be fixed. Then, for $\mu$-a.e. $w\in\Y$ we have that
    $$\int f\,\d\mu = \lim_{n\to \infty} A_n(f)(w)$$
    by Theorem \ref{theo:Birkhoff}. This gives the second inequality. 
\end{proof}

\section{Locally CAT$(\kappa)$ spaces}
\label{subsec:locally_CAT}

Let $(\Y,\sfd_\Y)$ be a metric space. A \emph{geodesic} of $(\Y,\sfd_\Y)$ is a curve $\gamma \colon [a,b] \to \Y$ which is an isometric embedding. The metric space $(\Y,\sfd_\Y)$ is said to be geodesic if for every couple of points $y,y'\in \Y$ there exists a geodesic with starting point $y$ and ending point $y'$. Such a geodesic will be occasionally denoted by $[y,y']$. A \emph{local geodesic} is a curve $\gamma \colon I \to \Y$, which is a local isometric embedding. If it is defined on $I=\R$, it is called a \emph{local geodesic line}. A metric space is \emph{geodesically complete} if every local geodesic $\gamma\colon [a,b] \to \Y$ can be extended to a local geodesic $\tilde\gamma\colon [a-\varepsilon,b+\varepsilon] \to \Y$ for some $\varepsilon > 0$.

The only connected, simply connected, $2$-dimensional Riemannian manifold of constant sectional curvature $\kappa\in \R$ will be denoted by $\mathbb{M}_\kappa^2$. We also define $D_\kappa := \pi/\sqrt{\kappa}$.
A metric space $(\Y,\sfd_\Y)$ is said to be CAT$(\kappa)$ if it is geodesic and every geodesic triangle of perimeter $<2D_\kappa$ is thinner than its comparison triangle in $\mathbb{M}_\kappa^2$ (see for instance \cite{BridsonHaefliger} for the basics of CAT$(\kappa)$-geometry).
As a consequence, every CAT$(\kappa)$-space is $D_\kappa$-{\em uniquely geodesic}, i.e. for every two points $y,y'$ with $\sfd(y,y')<D_\kappa$ there exists a unique geodesic with endpoints $y$ and $y'$. 
A metric space $(\Y,\sfd_\Y)$ is said to be \emph{locally \textup{CAT}$(\kappa)$} if every point has a neighbourhood $U$ such that $(U,\sfd)$ is CAT$(\kappa)$. Every locally CAT$(\kappa)$ space $(\Y,\sfd_\Y)$ admits a universal cover $\tilde{\Y}$ equipped with a metric $\tilde\sfd_\Y$ such that $(\tilde{\Y},\tilde{\sfd}_\Y)$ is CAT$(\kappa)$. 
Moreover, the fundamental group $\pi_1(\Y)$ acts naturally and faithfully by deck transformations and by isometries on $(\tilde{\Y}, \tilde{\sfd}_\Y)$. Hence $\pi_1(\Y)$ can be identified with a discrete subgroup $\Gamma$ of the group of isometries $\textup{Isom}(\tilde{\Y})$ of $(\tilde{\Y},\tilde\sfd_\Y)$ with the property that the quotient metric space $\Gamma \backslash \tilde{\Y}$ equipped with the quotient metric is isometric to $(\Y,\sfd_\Y)$.
The \emph{critical exponent} of the group $\Gamma \cong \pi_1(\Y)$ is
$$h(\tilde{\Y};\Gamma) := \lims_{R \to +\infty} \frac{1}{R}\log(\#\Gamma o \cap B(o,R)),$$
where the right-hand side does not depend on the choice of the basepoint $o \in \tilde{\Y}$. The limit in the right-hand side exists if $\kappa \le -1$ by \cite{Roblin2002}, see also \cite{Cavallucci:ErgodicLimitSet} for the Gromov-hyperbolic case. 
The $\textup{CAT}$-radius of a locally $\textup{CAT}(\kappa)$ space $(\Y,\sfd_\Y)$ is defined as
$$\rho_\textup{CAT}(\Y) := \sup\{r> 0\,:\, B(y,r) \textup{ is CAT}(\kappa)\textup{ for every }y\in\Y\}.$$
\begin{proposition}
\label{prop:lower_bound_CAT_radius}
    If $(\Y,\sfd_\Y)$ is a compact, locally $\textup{CAT}(\kappa)$ then $\rho_\textup{CAT}(\Y) > 0$. Moreover, for every $D,H \ge 0$ there exists $\rho_0 = \rho_0(D,H)$ such that if $(\Y,\sfd_\Y)$ is a compact, locally $\textup{CAT}(-1)$ space with $\textup{Diam}(\Y) \le  D$ and $h(\tilde{\Y};\pi_1(\Y)) \le H$ then either $\rho_\textup{CAT}(\Y) \ge \rho_0$ or $\pi_1(\Y)$ is virtually cyclic.
\end{proposition}
\begin{proof}
    The first part of the statement is well-known. The second one follows by \cite[Theorem 1.12]{BCGS2017}, see also \cite[Corollary 1.3]{CavallucciSambusetti2024} for a more general statement.
\end{proof}

\subsection{Geodesic flow and Bowen-Margulis measure} \label{subsec:Geodesic.flow.Bowen.Margulis.measure}

Let $(\X,\sfd)$ be a compact metric space. The \emph{space of local geodesic lines} is $\LocGeod(\X) := \{\gamma \colon \R \to \X\,:\,\gamma \text{ is a local geodesic}\}$. It is endowed with the topology of uniform convergence on compact subsets of $\R$. By \cite[Lemma 3.3]{Cavallucci2024OtalPeigne}, the topology on $\LocGeod(\X)$ is metrizable via the metric
\begin{equation}
    \label{eq:defin_D_a}
    {\sf D}(\gamma,\gamma') := \sup_{s\in \R} \sfd(\gamma(s),\gamma'(s))e^{-\vert s \vert}.
\end{equation}

We recall that the space $\LocGeod(\X)$ is not necessarily compact, see \cite[Example 3.4]{Cavallucci2024OtalPeigne}. But it is compact if $(\X,\sfd)$ is compact and locally CAT$(\kappa)$. This happens because there exists $\varepsilon > 0$ such that every local geodesic of $\X$ of length less than $\varepsilon$ is a geodesic, by \cite[Proposition II.1.4]{BridsonHaefliger} and Proposition \ref{prop:lower_bound_CAT_radius}.
\begin{proposition}
\label{prop:loc_geod_is_compact}
    Let $(\X,\sfd)$ be a compact, locally \textup{CAT}$(\kappa)$ space. Then $\LocGeod(\X)$ is compact.
\end{proposition}

The \emph{geodesic flow} on $\LocGeod(\X)$ is defined by $\Phi\colon \R \times \LocGeod(\X) \to \LocGeod(\X)$, $\Phi(t,\gamma)(\cdot) := \gamma(\cdot + t)$. We denote by $\Phi_t\colon \LocGeod(\X) \to \LocGeod(\X)$ the map $\Phi(t,\cdot)$. It is indeed a flow, i.e. $\Phi_0 = \textup{id}$ and $\Phi_{t+s} = \Phi_t \circ \Phi_s$ for every $t,s\in \R$. In the sequel we will consider the dynamical system $(\LocGeod(\X), \Phi_1)$ and we will call it the geodesic flow of $\X$.
For locally $\textup{CAT}(-1)$ spaces it is related to the critical exponent of the fundamental group of $\X$.

\begin{proposition}
\label{prop:top_ent_geod_flow_BM_measure}
    Let $(\X,\sfd)$ be a compact, locally \textup{CAT}$(-1)$ space and let $(\tilde{\X},\tilde{\sfd})$ be its universal cover. Then
    \begin{equation}
        \label{eq:h_top=critical_exponent}
        h_\textup{top}(\LocGeod(\X),\Phi_1) = h(\tilde{\X}, \pi_1(\X))
    \end{equation}
    and the dynamical system $(\LocGeod(\X),\Phi_1)$ is $h$-expansive. 
    Moreover there exists an ergodic, $\Phi_t$-invariant for every $t\in \R$, fully supported probability measure, called the Bowen-Margulis measure and denoted by $\mm_{\textup{BM}} \in \mathcal{E}_1(\LocGeod(\X),\Phi_1)$, such that
    $$h_\textup{top}(\LocGeod(\X),\Phi_1) = h_{\mm_{\textup{BM}}}.$$
    For every other measure $\nu \in \mathcal{M}_1(\LocGeod(\X), \Phi_1)$ it holds  that $h_{\nu} < h_{\mm_{\textup{BM}}}$.
\end{proposition}
\begin{proof}
    The equality \eqref{eq:h_top=critical_exponent} follows by \cite[Theorem A and Corollary 2.20]{ Cavallucci2024OtalPeigne}. The existence of the ergodic, $\Phi_t$-invariant, fully supported measure $\mm_{\textup{BM}}$ has been proved in \cite{RoblinThesis}. The fact that it is a measure of maximal entropy has been proved for instance in \cite[Theorem 18]{Ricks2021}, whose proof does not use the geodesic completeness assumption. The $h$-expansivity of the geodesic flow is \cite[Lemma 20]{Ricks2021}. The last part of the statement is \cite[Theorem 28]{Ricks2021}, where the geodesic completeness assumption therein is required only to guarantee that $\LocGeod(\X)$ has finite topological dimension. 
    Indeed it is used only through \cite[Corollary 21]{Ricks2021}, which is a consequence of \cite[Theorem 3.5]{Bowen1972} which requires $\LocGeod(\X)$ to have finite topological dimension. We now prove that it is always the case for compact, locally $\CATminus$ spaces. Indeed, $\LocGeod(\X)$ is isometric to the quotient of the space of geodesic lines $\Geod(\tilde{\X})$ of the universal cover $\tilde{\X}$ via the natural discrete action induced by $\pi_1(\X)$, see \cite[Corollary 3.7]{Cavallucci2024OtalPeigne}. Hence the topological dimension of $\LocGeod(\X)$ coincides with the one of $\Geod(\tilde\X)$ by \cite[Theorem 1.12.7]{Engelking1995}.
    By the Hopf parametrization (see \cite[page 12]{RoblinThesis}), $\Geod(\tilde\X)$ is homeomorphic to $(\partial \tilde\X \times \partial \tilde\X \setminus \Delta) \times \R$, where $\partial \tilde\X$ is the boundary at infinity of $\tilde\X$. By \cite[Theorems 1.1.2 and 1.5.16]{Engelking1995} the conclusion is true if $\partial\tilde\X$ has finite topological dimension. A visual metric (see \cite[page 7]{RoblinThesis}) on $\partial\tilde\X$ has finite Hausdorff dimension by \cite[Theorem 6.1 and Proposition 5.7]{Cavallucci23}, hence finite topological dimension.
\end{proof}

\begin{remark}
\label{rem:h_top=h_crit}
    For compact, locally $\textup{CAT}(0)$ spaces with Gromov-hyperbolic fundamental groups, \eqref{eq:h_top=critical_exponent} holds by \cite[Theorem A and Corollary 2.20]{Cavallucci2024OtalPeigne}. In general, the uniqueness of the measure of maximal entropy is false, see \cite[Example 1.1]{Cavallucci2024OtalPeigne}.    
    Beyond hyperbolicity, \eqref{eq:h_top=critical_exponent} is true
    for compact, geodesically complete, locally $\textup{CAT}(0)$ spaces  by \cite[Theorem A]{Ricks2021}. The uniqueness of the measure of maximal entropy holds for compact, geodesically complete, locally $\textup{CAT}(0)$ spaces with rank-one axis by \cite[Theorem B]{Ricks2021}.
\end{remark}

\subsection{Locally convex subsets}
\label{subsec:locally_convex_subsets}

A subset $\Sigma$ of a locally CAT$(\kappa)$ space $(\Y,\sfd_\Y)$ is said to be \emph{locally convex} if for every $y\in \Sigma$ there exists $0<r<D_\kappa$ such that $(B(y,r),\sfd)$ is CAT$(\kappa)$ and for every $y',y'' \in \Sigma \cap B(y,r)$, the only geodesic joining $y'$ to $y''$ is entirely contained in $\Sigma$. A closed, locally convex subset of a $\textup{CAT}(\kappa)$ space is convex (see for instance \cite[Theorem 1.1]{RamosCuevas2013}). Therefore the property above holds for $r = \rho_\textup{CAT}(\Y)$.

\begin{lemma}
    \label{lemma:uniform_distance_geodesics_to_singular_set}
    Let $(\Y,\sfd_\Y)$ be a compact, locally $\textup{CAT}(\kappa)$ space and let $\Sigma \subseteq \Y$ be closed and locally convex. Then every local geodesic $\sigma \colon I \to \Y$ such that $\alpha(\sigma), \omega(\sigma) \in \Sigma$ and $\sup_{t\in I} \sfd_\Y(\sigma(t),\Sigma) \le \rho_\textup{CAT}(\Y)/2$ satisfies $\sigma(I) \subseteq \Sigma$.
\end{lemma}
\begin{proof}
    Let $\rho := \rho_\textup{CAT}(\Y)$ and $\sigma\colon I \to \Y$ as in the statement.
    Define $h\colon I \to [0,+\infty)$, $h(t) = \sfd_\Y(\sigma(t),\Sigma)$. Fix $t\in I$ and consider the $\textup{CAT}(\kappa)$ ball $B(\sigma(t),\rho)$. For every $s \in I$ such that $\vert t-s\vert < \rho/4$, we have the equality $\sfd_\Y(\sigma(s),\Sigma)= \sfd_\Y(\sigma(s), \Sigma \cap B(\sigma(t),\rho))$, by triangle inequality. Moreover, the function $s\mapsto \sfd_\Y(\sigma(s),\Sigma \cap B(\sigma(t),\rho))$ is convex, see \cite[Exercise II.2.6]{BridsonHaefliger}. This implies that $h$ is a locally convex function. A locally convex function on an interval is a convex function, hence $h$ is a non-negative convex function that assumes value zero on $\partial I$. This implies that $h\equiv 0$, hence $\sigma(t) \in \Sigma$ for every $t\in I$.
\end{proof}

    We define 
    \begin{equation}
        \label{eq:defin_coming_back_geodesics}
        \Sigma^\circlearrowleft := \{ \sigma\colon I \to \Y \textup{ non-trivial local geodesic with } \{\alpha(\sigma),\omega(\sigma)\} = \sigma^{-1}(\Sigma)\}.
    \end{equation}
    Here, by non-trivial we mean that $I$ is not a point. The \emph{normal injectivity radius} of $\Sigma$ is defined as
\begin{equation}
    \label{eq:defin_normal_inj_radius}
    \inj := \inf\{\ell(\gamma)\,:\, \gamma \in \Sigma^\circlearrowleft \}.
\end{equation}
We collect some properties of the normal injectivity radius. The notation $\angle_{x}(y,z)$ stands for the Alexandrov angle, see \cite[Definition I.1.12 and Chapter II.3]{BridsonHaefliger}.

\begin{proposition}
\label{prop:normal_inj_radius_minimizing_curve}
    Let $(\Y,\sfd_\Y)$ be a compact, locally $\textup{CAT}(\kappa)$ space and let $\Sigma \subseteq \Y$ be closed and locally convex. Then the following properties hold.
    \begin{itemize}
        \item[(i)] $\inj \ge \rho_\textup{CAT}(\Y)$ and if it is finite then the infimum in \eqref{eq:defin_normal_inj_radius} is indeed a minimum.
        \item[(ii)] Let $\sigma \colon [0,\ell ] \to \Y$ be a minimizing element for $\inj$. Then $\angle_{\sigma(0)}(\sigma(t), z) \ge \pi/2$ whenever $t<\rho_\textup{CAT}(\Y)$ and $z\in \Sigma \cap B(\sigma(0),\rho_\textup{CAT}(\Y))$ and $\angle_{\sigma(\ell)}(\sigma(\ell - t), z) \ge \pi/2$ for every $t<\rho_\textup{CAT}(\Y)$ and $z\in \Sigma \cap B(\sigma(\ell),\rho_\textup{CAT}(\Y))$.
        \item[(iii)] If $\kappa = 0$ then the following conditions are equivalent:
            \begin{itemize}
                \item[(a)] $\Sigma^\circlearrowleft = \emptyset$,
                \item[(b)] $\inj = \infty$,
                \item[(c)] $\Sigma$ is a deformation retract of $\Y$,
                \item[(d)] the map $i_*\colon \pi_1(\Sigma) \to \pi_1(\Y)$ induced by the inclusion $i\colon \Sigma \to \Y$ is surjective.
            \end{itemize}
        \item[(iv)] If $\kappa=0$, $\Y$ is geodesically complete, and $\Sigma \neq \Y$, then $\Sigma^\circlearrowleft \neq \emptyset$.
    \end{itemize}
\end{proposition}
\begin{proof}
    We start by showing (i). Let $\rho:=\rho_\textup{CAT}(\Y)$. By Lemma \ref{lemma:uniform_distance_geodesics_to_singular_set}, $\inj\ge \rho$. Moreover, $\inj < \infty$ if and only if $\Sigma^\circlearrowleft \neq \emptyset$.
    Let $\sigma_n \colon [0,\ell_n] \to \Y$ be a sequence of non-trivial local geodesics with $\{\alpha(\sigma_n),\omega(\sigma_n)\} = \sigma_n^{-1}(\Sigma)$ and such that $\ell(\sigma_n) = \ell_n$ converges to $\inj$ as $n$ goes to infinity. By Ascoli-Arzelà's Theorem we can extract a limit curve $\sigma_\infty\colon [0,\ell_\infty]\to \Y$ which is a local geodesic by Proposition \ref{prop:loc_geod_is_compact} that satisfies $\alpha(\sigma_\infty),\omega(\sigma_\infty) \in \Sigma$, because $\Sigma$ is closed, and $\ell_\infty = \inj$. Moreover, $\ell_\infty = \ell(\sigma_\infty)$ because $\sigma_\infty$ is a local geodesic. By Lemma \ref{lemma:uniform_distance_geodesics_to_singular_set}, for every $n$ we can find $t_n \in [\rho/2,\ell_n-\rho/2]$ such that $\sfd_\Y(\sigma_n(t_n),\Sigma)\ge \rho/2$. Therefore there exists $t_\infty \in [\rho/2,\ell_\infty - \rho/2]$ such that $\sigma_\infty(t_\infty) \notin \Sigma$. Let $$t_* := \sup\{t>0\,:\, \sigma\restr{[0,t)} \subseteq \Sigma\} < \ell_\infty.$$ 
    Then, we can find $s_* > t_*$ such that $\sigma\restr{[t_*,s_*]}$ is one of the competitors for the infimum in \eqref{eq:defin_normal_inj_radius}. This gives $\inj \le s_*-t_* = \ell_\infty$, forcing $t_*=0$ and $s_* = \ell_\infty$, implying that $\{\alpha(\sigma_\infty), \omega(\sigma_\infty)\} = \sigma_\infty^{-1}(\Sigma)$. This shows that $\sigma_\infty$ realizes the minimum in \eqref{eq:defin_normal_inj_radius}, proving (i).

    The condition in (ii) follows directly by \cite[Exercise II.2.6]{BridsonHaefliger}.

    From now on we assume $\kappa = 0$. Let $p\colon\tilde\Y\to\Y$ be the universal cover of $\Y$, which is $\textup{CAT}(0)$, and $\tilde\Sigma = p^{-1}(\Sigma)$. Each connected component of $\tilde \Sigma$ is a convex subset of $\tilde \Y$. We have already seen the equivalence between (a) and (b). We observe that $\Sigma^\circlearrowleft = \emptyset$ if and only if $\tilde\Sigma$ is connected. Indeed, if $\tilde\Sigma$ is connected then the lift $\tilde\sigma$ of $\sigma \in \Sigma^\circlearrowleft$ is a geodesic segment in $\tilde\Y$ connecting two points of $\tilde\Sigma$. By connectedness and convexity, the image of $\tilde\sigma$ is included in $\tilde \Sigma$, hence the image of $\sigma$ is included in $\Sigma$, contradicting the definition of $\Sigma^\circlearrowleft$. If $\tilde\Sigma$ is disconnected, then by cocompactness of $\tilde\Y$ we can find two connected components of $\tilde\Sigma$ and a geodesic segment $\tilde\sigma$ joining the two and realizing their distance. Then $p\circ \tilde\sigma \in \Sigma^\circlearrowleft$, which is then non-empty. Now, if $\tilde\Sigma$ is connected, the projection map $\textup{pr}_{\tilde\Sigma} \colon \tilde\Y \to \tilde\Sigma$ is $1$-Lipschitz and $\pi_1(\Y)$-invariant. Moreover, the map $H\colon \tilde\Y \times [0,1] \to \tilde\Y$, $(y,t)\mapsto [y,\textup{pr}_{\tilde\Sigma}(y)](t)$ is a $\pi_1(\Y)$-equivariant homotopy equivalence, that descends to a deformation retraction of $\Y$ onto $\Sigma$. This shows that (a) implies (c). If $\Sigma$ is a deformation rectraction of $\Y$ then $i_*\colon\pi_1(\Sigma) \to \pi_1(\Y)$ is an isomorphism, hence (c) implies (d). Finally, if (d) holds then $\tilde\Sigma$ is connected, showing (a). 

    If we further assume that $\Y$ is geodesically complete and $\Sigma^\circlearrowleft=\emptyset$, then $\tilde{\Sigma}$ is a connected, $\pi_1(\Y)$-invariant, convex subset of $\tilde\Y$. By geodesic completeness, $\tilde{\Sigma} = \tilde\Y$ as follows by \cite[Lemma 3.13]{CapraceMono2009}, hence $\Sigma = \Y$. This proves (iv).  
\end{proof}

\begin{remark}
    Item (iii) is not true for $\kappa = 1$: a small ball $\Sigma$ into a round sphere $\Y$ satisfies $\Sigma^\circlearrowleft \neq \emptyset$ but $\Sigma$ is not a deformation retract of $\Y$. In many explicit examples one can check directly if $\Sigma^\circlearrowleft$ is empty or not.
\end{remark}

We end this section with a construction that will be useful later. 

\begin{lemma}
    \label{lemma:unique_projection_CAT}
    Let $(\Y,\sfd_\Y)$ be a compact $\textup{CAT}(\kappa)$ space of diameter $<D_\kappa/2$ and let $\Sigma \subseteq \Y$ be closed and convex. Let $y,y'\in \Y$ be two points such that $\#[y,y'] \cap \Sigma \le 1$. Then, there exists a unique point $z\in \Sigma$ which realizes the quantity $\inf_{z\in\Sigma}\sfd_\Y(y,z)+\sfd_\Y(z,y')$.
\end{lemma}
\begin{proof}
    The function $\sfd_\Y(y,\cdot)$ is convex when restricted to every geodesic, by \cite[Exercise II.2.6]{BridsonHaefliger}, and strictly convex when restricted to a geodesic that do not contain $y$. Similarly for $y'$. Therefore, the function $f\colon \Sigma \to \R$, $z\mapsto \sfd_\Y(y,z)+\sfd_\Y(z,y')$ is convex and therefore it admits a minimum. If the minimum is not unique, then we find a non-trivial geodesic $[z,z']\subseteq \Sigma$ realizing the minimum. The condition on the strict convexity of $\sfd_\Y(y,\cdot)$ implies that $y,z,z'$ belong to the same geodesic. Similarly, $y',z,z'$ belong to the same geodesic. Therefore the geodesic $[y,y']$ contains $[z,z']$, which is impossible by assumption.
\end{proof}

The unique point $z\in \Sigma$ above is called the \emph{$\Sigma$-midpoint} between $y$ and $y'$ and it is denoted by $\texttt{mid}_\Sigma(y,y')$.

\subsection{Examples}
\label{subsec:Examples}
The theory we are going to develop can be applied to every couple $(\Y,\Sigma)$ where $(\Y,\sfd_\Y)$ is a compact, locally $\textup{CAT}(\kappa)$ space and $\Sigma \subseteq \Y$ is closed and locally convex. This is a very wild generality, already when $\kappa = -1$. 
To familiarize the reader with the topic, we offer some concrete examples that will help them test all the concepts we are about to introduce and that already give interesting outputs.

\begin{itemize}
    \item[(1)] $\Y = \mathbb{S}^1$ and $\Sigma = \{y\}$ is a point.
    \item[(2)] $\Y = [0,1]$ and $\Sigma = \{0,1\}$.
    \item[(3)] $\Y$ is the graph with two vertices $x,y$, an edge $E$ connecting them, one loop around $x$ and one around $y$. Here $\Sigma$ can be taken to be either a point on $E$ (not necessarily the midpoint) or the whole $E$. Similarly, one can take $\Y$ to be any finite graph, even with weighted lengths on the edges, and $\Sigma$ is a union of vertices and edges.
    \item[(4)] $\Y$ is a compact hyperbolic surface with totally geodesic boundary $\Sigma$.
    \item[(5)] $\Y$ is a closed hyperbolic surface and $\Sigma$ is a periodic geodesic.
    \item[(6)] $\Y$ is a closed hyperbolic surface and $\Sigma$ is the disjoint union of a closed geodesic and a point.
\end{itemize}

\section{Branched coverings and broken geodesics}
\label{sec:top_branched_coverings}

This section is devoted to the study of branched coverings with upper curvature bounds.

\subsection{Branched coverings and broken geodesics}
\label{subsec:branched_coverings_definition}

A continuous, surjective map $\pi\colon \X\to\Y$ between topological spaces is a branched covering if the following two conditions are satisfied:
\begin{itemize}
    \item[(i)] $\pi\restr{\pi^{-1}(\Sigma)}\colon \pi^{-1}(\Sigma) \to \Sigma$ is a homeomorphism, where $\Sigma := \{y\in \Y\,:\, \#\pi^{-1}(y) = 1\}$;
    \item[(ii)] $\pi\restr{\X\setminus\pi^{-1}(\Sigma)}\colon \X \setminus\pi^{-1}(\Sigma) \to \Y\setminus\Sigma$ is a covering.
\end{itemize}
The set $\Sigma$ is called the \emph{singular set} and, with an abuse of notation, we identify $\pi^{-1}(\Sigma)$ with $\Sigma$.
If $\Y$ is locally connected, the map $\Y\setminus \Sigma \to \N\cup\{\infty\}$, $y\mapsto \#\pi^{-1}(y)$ is locally constant and therefore $\Sigma$ is closed. We say that the branched covering has \emph{finite degree} if $\sup_{y\in \Y} \#\pi^{-1}(y) < \infty$ and we denote the degree by $\textup{deg}(\pi)$. We say it has \emph{pure degree $k\in \N$} if $\#\pi^{-1}(y) = k$ for every $y\in \Y\setminus \Sigma$.
In particular, if $\Y\setminus \Sigma$ is connected and locally connected, then the branched covering has pure degree. By definition, $\#\pi^{-1}(y) \ge 2$ for every $y\in \Y\setminus \Sigma$. We will always assume that $\Y\setminus\Sigma \neq \emptyset$.

Suppose that $\Y$ is equipped with a length distance $\sfd_\Y$. We define
\begin{equation}
\label{eq:defin_metric_covering}
    \begin{aligned}
        \sfd(x,x') := &\inf \Big( \left\{ \sfd_\Y(\pi(x),\pi(x'')) + \sfd_\Y(\pi(x''),\pi(x'))\,:\, x'' \in \pi^{-1}(\Sigma) \right\} \\
        &\qquad \qquad \cup \left\{\ell(\gamma)\,:\, \gamma \text{ is a curve in } \Y\setminus\Sigma \text{ joining } \pi(x) \text{ to } \pi(x')\right\} \Big),
    \end{aligned}
\end{equation}
for $x,x'\in \X$. This defines a distance on $\X$, whose properties are collected in the next lemma.
\begin{lemma}
\label{lemma:basic_properties_covering}
    Let $(\Y,\sfd_\Y)$ be a length space and let $\pi\colon \X\to \Y$ be a branched covering with singular set $\Sigma$. Then $(\X,\sfd)$ is a length space, $\pi\colon (\X,\sfd) \to (\Y,\sfd_\Y)$ is $1$-Lipschitz, $\pi\restr{\X\setminus\pi^{-1}(\Sigma)}$ is a local isometry and $\pi\restr{\pi^{-1}(\Sigma)}$ is an isometry. Moreover, $\pi$ is a branched covering when $\X$ is equipped with the topology induced by $\sfd$. Finally, if $\Y$ is compact and $\pi$ has finite degree, then $\X$ is compact.
\end{lemma}
\begin{proof}
    A length space is locally connected, hence $\Sigma$ is closed.
    Let $x,x' \in \X$. Every curve $\gamma$ in $\Y\setminus\Sigma$ joining $\pi(x)$ to $\pi(x')$ satisfies $\ell(\gamma) \ge \sfd_\Y(\pi(x),\pi(x'))$. By triangular inequality and \eqref{eq:defin_metric_covering} we deduce that $\sfd(x,x')\ge \sfd_\Y(\pi(x),\pi(x'))$. This shows that $\pi\colon (\X,\sfd) \to (\Y,\sfd_\Y)$ is $1$-Lipschitz. Let us fix $y\in \Y\setminus \Sigma$ and a small ball $B=B(y,r)$ with the property that $r< \sfd(y,\Sigma)/2$ and that, denoting by $V$ one connected component of $\pi^{-1}(B)$, the map $\pi\restr{V}\colon V\to B$ is a homeomorphism. For every $x,x' \in V$ and every $\varepsilon > 0$ small enough we can find a curve $\gamma_\varepsilon \in \Y\setminus \Sigma$ joining $\pi(x)$ to $\pi(x')$ such that $\ell(\gamma_\varepsilon) \le \sfd_\Y(\pi(x),\pi(x')) + \varepsilon$. By \eqref{eq:defin_metric_covering}, $\sfd(x,x')\le \ell(\gamma_\varepsilon)$. This shows that $\pi$ is a local isometry on $\Y\setminus \Sigma$ and that $\pi \restr{\X\setminus \pi^{-1}(\Sigma)}$ is a covering with respect to the topology induced by $\sfd$. If $x,x' \in \pi^{-1}(\Sigma)$ then $\sfd(x,x') \le \sfd(\pi(x),\pi(x'))$ by \eqref{eq:defin_metric_covering}, showing that $\pi\restr{\pi^{-1}(\Sigma)}$ is an isometry, in particular a homeomorphism. Hence $\pi$ is a branched covering with respect to the topology induced by $\sfd$ on $\X$. The last statement  is obvious.
\end{proof}

\begin{definition}
    A compact geodesic branched covering $(\X,\Y,\Sigma,\pi)$ is the datum of a compact, geodesic space $(\Y,\sfd_\Y)$, a closed subset $\Sigma \subseteq \Y$, and a finite degree branched covering $\pi\colon \X \to \Y$ with singular set $\Sigma$. We will always equip $\X$ with the metric $\sfd$ of \eqref{eq:defin_metric_covering}.
\end{definition}

We introduce the space of $\Sigma$-broken local geodesics. They will play a key role in the description of the local geodesics of a branched covering.

\begin{definition}
\label{defin:broken_geodesics}
    Let $(\X,\sfd)$ be a metric space and let $\Sigma \subseteq \X$.  Let $J\subseteq \Z$ be an interval and let $\{\eta_j\}_{j\in J}$ be a family of local geodesics $\eta_j\colon I_j \to \X$ such that
    \begin{itemize}
        \item[(i)] they satisfy the conditions of Definition \ref{defin:concatenation_of_curves}, in particular $\omega(\eta_j) = \alpha(\eta_{j+1})$ for every $j\in J \setminus \{\sup J\}$;
        \item[(ii)] each $\eta_j$ is of one of the following types:
            \begin{itemize}
            \item either $\eta_j \subseteq \Sigma$, and we say that $\eta_j$ is \emph{internal},
            \item or $\eta_j^{-1}(\Sigma) = \partial I_j$, and we say that $\eta_j$ is \emph{external}.
        \end{itemize}
    \end{itemize}
    The associated $\Sigma$-broken local geodesic is the curve $\eta := \underset{j\in J}{\star}\eta_j$ defined in Definition \ref{defin:concatenation_of_curves}. The set of all $\Sigma$-broken local geodesics whose domain of definition is $\R$, together with their reparametrizations by time shifts, is denoted by $\textup{Broken}(\X,\Sigma)$. Its elements are called \emph{$\Sigma$-broken local geodesic lines}.
\end{definition}

\begin{remark}
    \label{rem:metric_on_broken_geodesics}
    The space $\textup{Broken}(\X,\Sigma)$ is metrizable. For instance, the same function ${\sf D}$ of \eqref{eq:defin_D_a} defines a distance on $\textup{Broken}(\X,\Sigma)$ inducing the compact-open topology. Moreover, the usual reparametrization flow $\Phi_t(\eta)(\cdot) := \eta(\cdot + t)$ for every $t\in \R$ defines a flow on $\textup{Broken}(\Y,\Sigma)$. 
\end{remark}

The next lemma describes every local geodesic of a compact, geodesic branched covering as a $\Sigma$-broken local geodesic.

\begin{lemma}
    \label{lemma:concatenation_of_geodesics_in_X}
    Let $(\X,\Y,\Sigma,\pi)$ be a compact, geodesic branched covering. Every local geodesic $\gamma$ of $\X$ can be seen in a unique way as a $\Sigma$-broken local geodesic $\gamma = \underset{j\in J}{\star}\gamma_j$ such that if $\gamma_j$ is internal then $\gamma_{j-1}$ and $\gamma_{j+1}$ are external. Moreover, $\pi\circ \gamma_j$ is a local geodesic in $\Y$ for every $j\in J$.
\end{lemma}
\begin{proof}
    Given a local geodesic $\gamma \colon I \to \X$, the set $\gamma^{-1}(\X\setminus\Sigma)$ is open, hence it is a countable union of open subintervals of $I$. Every local geodesic segment corresponding to each of such intervals is external. The remaining corresponding local geodesic segments are internal. The uniqueness of such a decomposition is clear.
    If $\gamma_j$ is internal then $\pi\circ \gamma_j$ is a local geodesic because $\pi$ is an isometry on $\Sigma$ by Lemma \ref{lemma:basic_properties_covering}. If $\gamma_j$ is external then $\pi\circ\gamma_j$ is a local geodesic because $\pi$ is a local isometry on $\X\setminus \Sigma$ by Lemma \ref{lemma:basic_properties_covering}.
\end{proof}

The next result describes the number of possible lifts of the pieces appearing in Lemma \ref{lemma:concatenation_of_geodesics_in_X}. It represents the first step towards the computation of the entropy of a branched covering.
We recall that given a curve $\gamma \colon I \to \Y$ and a point $x\in \pi^{-1}(\alpha(\gamma))$, a lift of $\gamma$ starting at $x$ is a curve $\tilde{\gamma}\colon I \to \X$ such that $\pi\circ\tilde{\gamma} = \gamma$ and $\alpha(\tilde{\gamma}) = x$.

\begin{lemma}
\label{lemma:lifts_geodesics_coverings}
    Let $(\X,\Y,\Sigma,\pi)$ be a compact, geodesic branched covering. Let $\gamma \colon I \to \Y$ be a local geodesic and let $x\in \pi^{-1}(\alpha(\gamma))$. Then there exists a lift $\tilde{\gamma}$ of $\gamma$ starting at $x$, and every lift is a local geodesic. Moreover, the following properties hold.
    \begin{itemize}
        \item[(i)] If $\gamma \subseteq \Sigma$ then $\tilde{\gamma}$ is unique and $\tilde{\gamma}\subseteq \Sigma$.
        \item[(ii)] If $\gamma \cap \Sigma\subseteq \{\omega(\gamma)\}$ then $\tilde{\gamma}$ is unique and $\tilde{\gamma} \cap \Sigma \subseteq \{\omega(\tilde{\gamma})\}$.
        \item[(iii)] If $\alpha(\gamma) = \gamma\cap\Sigma$ then  there exist exactly $k:=\#\pi^{-1}(\omega(\gamma))$ lifts $\tilde{\gamma}_1,\ldots, \tilde{\gamma}_k$.
    \end{itemize}
\end{lemma}

\begin{proof}
    Let us first deal with local geodesics in cases (i), (ii) and (iii). The existence and the uniqueness in case (i) is trivial. In case (ii), for every $t<\sup I$ we have a unique lift of $\gamma\restr{I\setminus(t,\sup I]}$ as in the statement, because of the usual covering properties. By compactness and continuity we have a unique lift of $\gamma$. In both cases, the lift is a local geodesic by Lemma \ref{lemma:basic_properties_covering}.    
    In case (iii), let $\{x_1,\ldots,x_k\} = \pi^{-1}(\omega(\gamma))$ and consider the curve $-\gamma$. Apply (ii) to find the unique lifts $-\tilde{\gamma}_i$ in $\X$ of $-\gamma$ starting at $x_i$. The endpoints of all these curves are $x = \alpha(\gamma) \in \Sigma$. The curves $-(-\tilde{\gamma}_i)$, projects to $\gamma$ and starts at $x$. Any other curve with this property, once reversed, has to coincide with one of the $-\tilde{\gamma}_i$'s, by the uniqueness in (ii). By construction, the lift is a local geodesic by Lemma \ref{lemma:basic_properties_covering}.

    For the general case we proceed as follows.
    By Lemma \ref{lemma:concatenation_of_geodesics_in_X} applied to the trivial covering $\textup{id}\colon \Y \to \Y$, the local geodesic $\gamma$ can be decomposed, uniquely, as $\gamma = \underset{j\in J}{\star} \gamma_j$, where each $\gamma_j$ is a local geodesic which is either internal or external. If $\gamma_j$ is internal then it falls in case (i). If $\gamma_j$ is external there are two possibilities.
    In the first case, $\omega(\gamma_j) = \gamma_j\cap \Sigma$. Hence we are in case (ii). 
    In the second case we have that $\alpha(\gamma_j) \in \gamma_j\cap \Sigma$ and $\alpha(t)\notin \Sigma$ for every $t\in \textup{Int}(I_j)$. 
    Fix such a $t$: the restriction of $\gamma_j$ up to $t$ is a curve satisfying (iii), hence a lift exists. The restriction on the remaining interval is a curve satisfying (ii), so a lift exists again. Combining the lifts, we get a lift of $\gamma_j$. Actually, we get exactly $k:=\#\pi^{-1}(\gamma_j(t))$ lifts, for $t\in \textup{Int}(I_j)$.

    By construction, every lift $\tilde{\gamma}\colon I \to \X$ of $\gamma$ is a concatenation of local geodesics, hence it is $1$-Lipschitz. Therefore, to show that it is a local geodesic it is enough to show that $\sfd(\tilde{\gamma}(t),\tilde{\gamma}(s)) \ge \vert t - s \vert$ for $t,s$ sufficiently close. Indeed,
    $$\sfd(\tilde{\gamma}(t),\tilde{\gamma}(s)) \ge \sfd_\Y(\gamma(t),\gamma(s)) = \vert t - s \vert$$
    if $t,s$ are close enough because $\gamma$ is a local geodesic. Here, we used that $\pi$ is $1$-Lipschitz, by Lemma \ref{lemma:basic_properties_covering}.
\end{proof}

\begin{corollary}
\label{cor:covering_is_submetry}
    Let $(\X,\Y,\Sigma,\pi)$ be a compact, geodesic branched covering. Then  $\pi$ is a submetry, i.e. $\pi(B(x,r)) = B(\pi(x),r)$ for every $x\in \X$ and every $r>0$. Moreover, lifts of geodesics are geodesics. Finally, $\sfd(x,z) = \sfd_\Y(\pi(x),z)$ for every $x\in\X$ and $z\in \Sigma$.
\end{corollary}
\begin{proof}
    By Lemma \ref{lemma:basic_properties_covering} $\pi$ is $1$-Lipschitz, so $\pi(B(x,r)) \subseteq B(\pi(x),r)$. On the other hand, for every $y\in B(\pi(x),r)$, we can lift a geodesic $\gamma$ joining $\pi(x)$ to $y$ to a local geodesic $\tilde\gamma$ starting at $x$, by Lemma \ref{lemma:lifts_geodesics_coverings}. The point $x'=\omega(\tilde\gamma)$ belongs to $B(x,r)$ because of Lemma \ref{lemma:basic_properties_covering} and satisfies $\pi(x')=y$. This shows the first part of the claim. The same proof gives that $\sfd_\Y(\pi(x),y)\le \sfd(x,x') \le \ell(\tilde{\gamma}) = \ell(\gamma) = \sfd_\Y(\pi(x),y)$, forcing the equalities and showing that $\tilde{\gamma}$ is a geodesic. Finally, given $x\in \X$ and $z\in \Sigma$, we consider a geodesic $[z,\pi(x)]$. The construction of lifts in Lemma \ref{lemma:lifts_geodesics_coverings} gives a curve $\gamma$ joining $z$ to $x$. This is possible since in the decomposition of $\gamma$ provided by Lemma \ref{lemma:concatenation_of_geodesics_in_X}, there exists at least a piece of type (iii) in Lemma \ref{lemma:lifts_geodesics_coverings}. By construction, $\sfd(z,x)\le \ell(\gamma) = \sfd_\Y(z,\pi(x))$, implying $\sfd(z,x) = \sfd_\Y(z,\pi(x))$ by Lemma \ref{lemma:basic_properties_covering}.
\end{proof}

\subsection{Locally CAT$(\kappa)$ branched coverings}

If $(\Y,\sfd_\Y)$ is $\textup{CAT}(\kappa)$ and $\Sigma$ is $D_\kappa$-convex, then the total space of a branched covering over $\Sigma$ is $\textup{CAT}(\kappa)$ by \cite[Theorem 2.1]{Allcock2000}. The following local version is a direct consequence. We remark that this generalization is stated, without proof, in \cite[Section 4.4]{Gromov-HypGps}. 

\begin{proposition}
\label{prop:branched_covering_loc_CAT-1}
    Let $(\Y,\sfd_\Y)$ be a compact, locally $\textup{CAT}(\kappa)$, geodesic space and let $\Sigma \subseteq \Y$ be closed and locally convex. If $\pi\colon \X \to \Y$ is a branched covering with singular set $\Sigma \subseteq \Y$ then $(\X,\sfd)$ is locally $\textup{CAT}(\kappa)$ and $\Sigma$ is locally convex in $\X$. Moreover,
    for every $\textup{CAT}(\kappa)$ ball $B(y,r) \subseteq \Y$ with $r<D_\kappa$ and every $x\in \pi^{-1}(y)$, the ball $B(x,r)$ is $\textup{CAT}(\kappa)$. 
\end{proposition}
\begin{proof}
    Let $B(y,r) \subseteq \Y$ with $r<D_\kappa$ be $\textup{CAT}(\kappa)$, let $x\in \pi^{-1}(y)$ and consider the ball $B(x,r)$. Then $\pi\restr{B(x,r)} \colon B(x,r) \to B(y,r)$ is a branched covering. Hence, $B(x,r)$ is $\textup{CAT}(\kappa)$ by \cite[Theorem 2.1]{Allcock2000}. This shows that $(\X,\sfd)$ is locally $\textup{CAT}(\kappa)$. The fact that $\Sigma$ is locally convex follows by Lemma \ref{lemma:basic_properties_covering}.
\end{proof}

\begin{definition}
    A compact locally $\textup{CAT}(\kappa)$ branched covering $(\X,\Y,\Sigma,\pi)$ is the datum of a compact, locally $\textup{CAT}(\kappa)$, geodesic space $(\Y,\sfd_\Y)$, of a closed, locally convex subset $\Sigma \subseteq \Y$, and a branched covering $\pi\colon \X \to \Y$ with singular set $\Sigma$. We will always equip $\X$ with the metric $\sfd$ of \eqref{eq:defin_metric_covering}.
\end{definition}

The following is an easy consequence of the positivity of the normal injective radius of $\Sigma$.
\begin{lemma}
\label{lemma:gamma^-1_is_discrete}
    Let $(\X,\Y,\Sigma,\pi)$ be a compact, locally $\textup{CAT}(\kappa)$ branched covering. Then $\partial\gamma^{-1}(\Sigma)$ is closed and discrete for every $\gamma \in \LocGeod(\X)$.
\end{lemma}
\begin{proof}
    The set is clearly closed. If $\Sigma^\circlearrowleft = \emptyset$ then $\gamma$ cannot have external pieces in the decomposition provided by Lemma \ref{lemma:concatenation_of_geodesics_in_X}, hence $\gamma \subseteq \Sigma$ and $\partial\gamma^{-1}(\Sigma) =\emptyset$. If $\Sigma^\circlearrowleft \neq \emptyset$ then every external piece in the unique decomposition of $\gamma$ has length at least $\inj \ge \rho_\textup{CAT}(\Y)$ by Proposition \ref{prop:normal_inj_radius_minimizing_curve}. Since every internal piece has two adjacent external pieces by Lemma \ref{lemma:concatenation_of_geodesics_in_X}, we get that $\partial\gamma^{-1}(\Sigma)$ is discrete.
\end{proof}

We describe the local behavior of local geodesics in locally $\textup{CAT}(\kappa)$ branched coverings.
\begin{lemma}
\label{lemma:geodesics_cross_Sigma}
    Let $(\X,\Y,\Sigma,\pi)$ be a compact, locally $\textup{CAT}(\kappa)$ branched covering. Let $z\in \Sigma$ and $B=B(z,\min\{\rho_\textup{CAT}(\Y),D_\kappa\}) \subseteq \Y$. 
    Let $y,y'\in B$ and let $x\in \pi^{-1}(y)$. 
    Then one of the following five cases occurs.
    \begin{itemize}
        \item[(1)] $y,y'\in \Sigma$. Then $\pi^{-1}(y) = \{x\}$, $\pi^{-1}(y') = \{x'\}$ and $[x,x']\subseteq \Sigma$.
        \item[(2)] $y\notin \Sigma$, $y'\in \Sigma$. In this case, $\pi^{-1}(y') = \{x'\}$ and the geodesic $[x,x']$ is the unique lift of $[y,y']$ starting at $x$.
        \item[(3)] $y\in \Sigma$, $y'\notin \Sigma$. In this case, for every $x'\in\pi^{-1}(y')$, the geodesic $[x,x']$ is the lift of $[y,y']$ starting at $x$ and ending at $x'$.
        \item[(4)] $y,y'\notin \Sigma$ and there exists $p\in [y,y']\cap \Sigma$. Then $p=\textup{\texttt{mid}}_\Sigma(y,y')$ and for every $x' \in \pi^{-1}(y')$ the geodesic $[x,x']$ is the concatenation of the unique lift $[x,p]$ of $[y,p]$ as in (2) and the unique lift $[p,x']$ of $[p,y']$ as in (3).
        \item[(5)] $y,y'\notin \Sigma$ and $[y,y']\cap \Sigma = \emptyset$.
        Then
        \begin{itemize}
            \item[(i)] there exists a unique $x'_* \in \pi^{-1}(y')$ such that $[x,x'_*]$ is the lift of $[y,y']$;
            \item[(ii)] for every $x'\in \pi^{-1}(y')\setminus\{x'_*\}$, the geodesic $[x,x']$ equals the concatenation of the unique lift $[x,p]$ of $[y,p]$ as in (2) and the unique lift $[p, x']$ of $[p, y']$ as in (3), where $p=\textup{\texttt{mid}}_\Sigma(y,y')$.
        \end{itemize}
    \end{itemize}
    Moreover, for every two distinct $x_1,x_2\in \pi^{-1}(y)$ it holds that $\sfd(x_1,x_2) = 2\sfd_\Y(y,\Sigma)$.
\end{lemma}
\begin{proof}
    Geodesics between points of $B$ are unique. It is clear that the geodesic $[y,y']$ satisfies at least one of the five conditions. 

    \textbf{Case (1).} If $y,y'\in \Sigma$ then $[y,y'] \subseteq \Sigma$. In this case, $\pi^{-1}(y) = \{x\}$, $\pi^{-1}(y') = \{x'\}$ by definition of branched covering. Moreover $[x,x']$ is exactly $\pi^{-1}([y,y'])$ because of Corollary \ref{cor:covering_is_submetry}.

    \textbf{Case (2).} If $y\notin \Sigma$ and $y'\in \Sigma$ then $\pi^{-1}(y') = \{x'\}$. The existence of a unique lift of $[y,y']$ to a local geodesic connecting $x$ to $x'$ is provided by Lemma \ref{lemma:lifts_geodesics_coverings}.(ii). This lift is a geodesic by Corollary \ref{cor:covering_is_submetry}.

    \textbf{Case (3).} If $y\in \Sigma$ and $y'\notin \Sigma$, for every $x'\in \pi^{-1}(y')$ there exists a lift of $[y,y']$ to a local geodesic connecting $x$ to $x'$, by Lemma \ref{lemma:lifts_geodesics_coverings}.(iii). This lift is a geodesic by Corollary \ref{cor:covering_is_submetry}. 

    \textbf{Case (4).} If $y,y'\notin \Sigma$ and there exists $p \in [y,y']\cap \Sigma$, then $[y,y']$ is the concatenation of the geodesics $[y,p]$ and $[p,y']$. For every $x'\in \pi^{-1}(y')$ we can find lifts $[x,p]$, $[p,x']$ applying Cases (2) and (3). The concatenation of $[x,p]$ and $[p,x']$ is a curve joining $x$ to $x'$ whose length is at most $\sfd_\Y(y,p) + \sfd_\Y(p,y') = \sfd_\Y(y,y') \le \sfd(x,x')$ because of Lemma \ref{lemma:basic_properties_covering}. Then it is a geodesic. Moreover, $p = \textup{\texttt{mid}}_\Sigma(y,y')$ because $\sfd_\Y(y,p) + \sfd_\Y(p,y') = \sfd_\Y(y,y') \le \sfd_\Y(y,q) + \sfd_\Y(q,y')$ for every $q\in \Sigma$, so Lemma \ref{lemma:unique_projection_CAT} applies.

    \textbf{Case (5).} If $y,y'\notin \Sigma$ and $[y,y']\cap \Sigma = \emptyset$. Lemma \ref{lemma:lifts_geodesics_coverings}.(ii) implies that there exists a unique lift of $[y,y']$ starting at $x$, and Corollary \ref{cor:covering_is_submetry} says that such a lift is a geodesic. Call the endpoint of this lift $x_*' \in \pi^{-1}(y')$. Let now $x' \in \pi^{-1}(y')\setminus\{x'_*\}$ and call $p:=\textup{\texttt{mid}}_\Sigma(y,y')$. The geodesic $[x,x']$ has to intersect $\Sigma$, otherwise it would project to a local geodesic joining $y$ to $y'$ by Lemma \ref{lemma:basic_properties_covering}, hence to $[y,y']$. This would imply $x'=x'_*$. Let $q \in [x,x']\cap \Sigma$ be fixed. Then we have 
    \begin{equation}
        \label{eq:lifting}
        \sfd(x,x') = \sfd(x,q) + \sfd(q,x') \ge \sfd_\Y(y,q) + \sfd_\Y(q,y') \ge \sfd_\Y(y,p) + \sfd_\Y(p,y'),
    \end{equation}
    because of Lemma \ref{lemma:basic_properties_covering} and Lemma \ref{lemma:unique_projection_CAT}. Moreover, we have equalities if and only if $q=p$. Now, consider the geodesic segments $[y,p]$ and $[p,y']$. Because of Cases (2) and (3) there exist unique lifts $[x,p]$ and $[p,x']$. The curve obtained by concatenating these two segments joins $x$ to $x'$ and has length $\sfd_\Y(y,p) + \sfd_\Y(p,y') \ge \sfd(x,x')$. Therefore we have equalities in \eqref{eq:lifting}, $q=p$ and the concatenation of $[x,p]$ and $[p,x']$ is the geodesic $[x,x']$.

    Finally, the last statement is a special case of (5). Indeed, if we take $y=y'$ then either $y \in \Sigma$ and there is nothing to prove or $y\notin \Sigma$ and $[y,y]\cap \Sigma = \emptyset$. For every two distinct $x_1,x_2\in \pi^{-1}(y)$ we apply Case (5) and we get that $\sfd(x_1,x_2) = 2\sfd(y,\texttt{mid}_\Sigma(y,y))$. By definition of the $\Sigma$-midpoint, $\texttt{mid}_\Sigma(y,y)$ is simply the projection of $y$ on $\Sigma$, hence $\sfd(y,\texttt{mid}_\Sigma(y,y)) = \sfd(y,\Sigma)$, which concludes the proof.
\end{proof}

An immediate consequence is the following estimate.
\begin{corollary}
\label{cor:uniform_distance_segments}
    Let $(\X,\Y,\Sigma,\pi)$ be a compact, locally $\textup{CAT}(\kappa)$ branched covering and let $\sigma \colon I \to \Y$ be a local geodesic such that $\sigma \cap \Sigma = \{\alpha(\sigma),\omega(\sigma)\}$. Then there exists $t \in I$ such that for every two distinct lifts $\tilde{\sigma},\tilde{\sigma}'\colon I \to \X$ of $\sigma$ we have that $\sfd(\tilde\sigma(t),\tilde\sigma'(t)) \ge \min\{\rho_\textup{CAT}(\Y),D_\kappa\}$.
\end{corollary}
\begin{proof}
    Call $\rho := \min\{\rho_\textup{CAT}(\Y),D_\kappa\}$. By Lemma \ref{lemma:uniform_distance_geodesics_to_singular_set} we may find $t\in I$ such that $\sfd_\Y(\sigma(t),\Sigma) = \rho/2$. Let $z\in \Sigma$ be a point such that $\sigma(t) \in B(z,\rho)$. Every two distinct lifts $x_1,x_2$ of $\sigma(t)$ satisfy $\sfd(x_1,x_2) = 2\sfd_\Y(\sigma(t),\Sigma) = \rho$ by Lemma \ref{lemma:geodesics_cross_Sigma}. Finally, Lemma \ref{lemma:lifts_geodesics_coverings}, and its proof, shows that there are exactly $\#\pi^{-1}(\sigma(t))$ lifts of $\sigma$, each of them passing through one of the lifts of $\sigma(t)$ at time $t$. This concludes the proof.
\end{proof}

While lifts of local geodesic lines are local geodesic lines, in general the projection of local geodesic lines are not local geodesic lines, but only $\Sigma$-broken local geodesic lines. 

\begin{proposition}
\label{prop:projection_geodesics_to_broken_geodesics}
    Let $(\X,\Y,\Sigma,\pi)$ be a compact, locally $\textup{CAT}(\kappa)$ branched covering. Then the map 
    $$\Pi \colon \LocGeod(\X) \to \textup{Broken}(\Y,\Sigma), \quad \gamma \mapsto \pi\circ \gamma$$ is continuous and satisfies
    $$\Phi_t \circ \Pi = \Pi \circ \Phi_t$$
    for every $t\in \R$. The image of $\Pi$ is denoted by $\textup{Broken}(\Y,\Sigma;\X) \subseteq \textup{Broken}(\Y,\Sigma)$. It is compact and $\Phi_t$-invariant. Moreover, $\textup{Broken}(\Y,\Sigma;\X) \supseteq \LocGeod(\Y)$.
\end{proposition}

\begin{proof}
    The map is well defined because of Lemma \ref{lemma:concatenation_of_geodesics_in_X} and it is clearly continuous. The conjugation of the flows is trivial. Finally, $\textup{Broken}(\Y,\Sigma;\X)$ is compact because it is a continuous image of a compact space, and it is clearly $\Phi_t$-invariant. The last containment is Lemma \ref{lemma:lifts_geodesics_coverings}.
\end{proof}

We are ready to prove the following more precise version of Theorem \ref{theo_intro:broken_does_not_depend}.

\begin{theorem}
\label{theo:characterization_broken_geodesics}
    Let $(\X,\Y,\Sigma,\pi)$ be a compact, locally $\textup{CAT}(\kappa)$ branched covering. A broken geodesic $\eta = \underset{j\in J}{\star} \eta_j \in \textup{Broken}(\Y,\Sigma)$ belongs to $\textup{Broken}(\Y,\Sigma;\X)$ if and only if for every $j\in J$ one of the following two mutually exclusive conditions holds.
    \begin{itemize}
        \item[(a)] $\eta_j \star \eta_{j+1}$ is a local geodesic in $\Y$,
        \item[(b)] $\eta_j,\eta_{j+1}$ are external, $\eta_j\star \eta_{j+1}$ is not a local geodesic in $\Y$, they concatenate at time $t\in \R$ and 
        \begin{equation}
            \label{eq:condition_broken_is_liftable_to_a_geodesic}
            \alpha(\eta_{j+1}) = \omega(\eta_j) = \textup{\texttt{mid}}_\Sigma(\eta(t-s),\eta(t+s))
        \end{equation}
        for every $s < \min\{\rho_\textup{CAT}(\Y),D_\kappa\}$.
    \end{itemize}
    In particular, if $(\X,\Y,\Sigma,\pi)$, $(\X',\Y,\Sigma,\pi')$ are two compact, locally $\textup{CAT}(\kappa)$ branched coverings over the same couple $(\Y,\Sigma)$ then $\textup{Broken}(\Y,\Sigma;\X) = \textup{Broken}(\Y,\Sigma;\X').$
\end{theorem}

\begin{proof}
    Let $\gamma \in \LocGeod(\X)$ and let $\gamma = \underset{j\in J}{\star} \gamma_j$ be the unique decomposition as $\Sigma$-broken local geodesic line provided by Lemma \ref{lemma:concatenation_of_geodesics_in_X}. Then $\Pi(\gamma) = \eta = \underset{j\in J}{\star} \eta_j$, where each $\eta_j$ is a local geodesic of $\Y$ by Lemma \ref{lemma:concatenation_of_geodesics_in_X}. Let $t$ be the concatenation time between $\eta_j$ and $\eta_{j+1}$, call $z := \eta(t)$, fix $s < \rho := \min\{\rho_\textup{CAT}(\Y),D_\kappa\}$ and call $y_j := \eta(t-s)$, $y_{j+1} := \eta(t+s)$. Then we are in the position to apply Lemma \ref{lemma:geodesics_cross_Sigma}. If we are in cases (1), (2), (3), (4) of Lemma \ref{lemma:geodesics_cross_Sigma} then $\gamma\restr{[t-s,t+s]}$ is the unique lift of the geodesic $[y_j,y_{j+1}]$ joining $\gamma(t-s)$ to $\gamma(t+s)$. But $\gamma\restr{[t-s,t+s]}$ is also the lift of $\eta\restr{[t-s,t+s]}$, then $\eta\restr{[t-s,t+s]} = [y_j,y_{j+1}]$ and $\eta_{j}\star \eta_{j+1}$ is a local geodesic in $\Y$. In Case (5) we notice that $y_{j+1}$ cannot be $x'_*$ from (i). Indeed, the geodesic $\gamma\restr{[t-s,t+s]}$ intersects $\Sigma$, so does its projection. By the description in (ii) we get that $\eta\restr{[t-s,t]} = [y_j,\textup{\texttt{mid}}_\Sigma(y_j,y_{j+1})]$ and $\eta\restr{[t,t+s]} = [\textup{\texttt{mid}}_\Sigma(y_j,y_{j+1}),y_{j+1}]$. We have shown that every $\eta \in \textup{Broken}(\Y,\Sigma;\X)$ satisfies conditions (a) or (b) for every $j\in J$.

    Vice versa, let $\eta = \underset{j\in J}{\star} \eta_j \in \textup{Broken}(\Y,\Sigma)$ be a $\Sigma$-broken geodesic line satisfying conditions (a) and (b) for every $j\in J$. We fix $j_0 \in J$ and we define $\gamma_{j_0}$ as one of the lifts of $\eta_{j_0}$ provided by Lemma \ref{lemma:lifts_geodesics_coverings}. We now consider $\eta_{j_0+1}$. If $\eta_{j_0}\star \eta_{j_0+1}$ satisfies (a), i.e. it is a local geodesic in $\Y$, then $\gamma_{j_0}\star \gamma_{j_0+1}$ is a local geodesic of $\X$ for any choice of the lift $\gamma_{j_0+1}$. This follows by Corollary \ref{cor:covering_is_submetry}. If $\eta_{j_0}\star \eta_{j_0+1}$ satisfies (b) then we call $t$ the concatenation time, $z:=\eta(t)$, $s<\rho := \min\{\rho_\textup{CAT}(\Y),D_\kappa\}$, $y_{j_0} = \eta(t-s)$, $y_{j_0+1} = \eta(t+s)$ and $x_{j_0}$ the point on $\gamma_{j_0}$ that lifts $y_{j_0}$. We are in position to apply Lemma \ref{lemma:geodesics_cross_Sigma}. If the geodesic $[y_{j_0},y_{j_0 + 1}]$ intersects $\Sigma$, then we are in Case (4) of Lemma \ref{lemma:geodesics_cross_Sigma} and $[y_{j_0},y_{j_0+1}] = \eta \restr{[t-s,t+s]}$ violating the fact that $\eta_{j_0}\star \eta_{j_0+1}$ is not a local geodesic. Therefore, $[y_{j_0},y_{j_0 + 1}] \cap \Sigma = \emptyset$. Therefore Case (5) occurs. Since $\eta_{j_0+1}$ is external then $\#\pi^{-1}(y_{j_0}+1) \ge 2$. Therefore we can choose a lift $x_{j_0+1}$ of $y_{j_0 +1}$ for which Case (5).(ii) holds. Lemma \ref{lemma:lifts_geodesics_coverings} gives the existence of a unique lift $\gamma_{j_0+1}$ of $\eta_{j_0+1}$ starting at $\omega(\gamma_{j_0})$ and passing through $x_{j_0+1}$. The condition satisfied by $\eta_{j_0}$ and $\eta_{j_0+1}$ together with Lemma \ref{lemma:geodesics_cross_Sigma} imply that the concatenation $\gamma_{j_0}\star \gamma_{j_0+1}$ is a local geodesic in $\X$. We can continue in the same way for all $j\in J$ in order to find a local geodesic $\gamma \in \LocGeod(\X)$ such that $\Pi(\gamma) = \eta$, i.e. $\eta \in \textup{Broken}(\Y,\Sigma;\X)$.

    Conditions (a) and (b) are independent of the branched covering. This gives the last part of the thesis. 
\end{proof}

The subset of $\Sigma$-broken local geodesic lines satisfying the conditions of Theorem \ref{theo:characterization_broken_geodesics} will be denoted by $\Br$. They are the $\Sigma$-broken local geodesics of $\Y$ that admit a lift to a local geodesic in any non-trivial branched covering over $\Y$ with singular set $\Sigma$. 
Theorem \ref{theo:characterization_broken_geodesics} justifies the next definition.

\begin{definition}
Let $(\Y,\sfd_\Y)$ be a compact, locally $\textup{CAT}(\kappa)$ space and let $\Sigma \subseteq \Y$ be closed and locally convex.
Given $\eta \in \Br$ and $t\in \partial \eta^{-1}(\Sigma)$ we say that $\eta(t)$ is a concatenation point (between $\eta_j$ and $\eta_{j+1}$) of:
\begin{itemize}
    \item[a)] \emph{non-branching type} if $\eta_{j+1}$ is internal;
    \item[b)] \emph{extremal type} if $\eta_j \star \eta_{j+1}$ is a local geodesic and $\eta_{j+1}$ is external;
    \item[c)] \emph{non-extremal type} if $\eta_j \star \eta_{j+1}$ is not a local geodesic.
\end{itemize}
The \emph{concatenation angle} at $t$ is the quantity
$$\angle(\eta,t) := \angle_{\eta(t)}(\eta(t-s),\eta(t+s))$$
where $0<s<\min\{\rho_\textup{CAT}(\Y), D_\kappa\}$.
\end{definition}

\begin{corollary}
\label{cor:non-extremal_iff_angle<pi}
    Let $(\Y,\sfd_\Y)$ be a compact, locally $\textup{CAT}(\kappa)$ space and let $\Sigma \subseteq \Y$ be closed and locally convex.
    Let $\eta \in \Br$ and $t\in \partial\eta^{-1}(\Sigma)$ be the concatenation time between $\eta_j$ and $\eta_{j+1}$. Then: 
    \begin{itemize}
        \item[(i)] $\eta(t)$ is of non-branching type if and only if $\eta_j \star \eta_{j+1}$ satisfies Condition (a) of Theorem \ref{theo:characterization_broken_geodesics} and $\eta_{j+1}$ is internal;
        \item[(ii)] $\eta(t)$ is of extremal type if and only if $\eta_j \star \eta_{j+1}$ satisfies Condition (a) of Theorem \ref{theo:characterization_broken_geodesics} and $\eta_{j+1}$ is external;
        \item[(iii)] $\eta(t)$ is of non-extremal type if and only if $\eta_j \star \eta_{j+1}$ satisfies Condition (b) of Theorem \ref{theo:characterization_broken_geodesics}.
    \end{itemize}
    Moreover, $\angle(\eta,t)$ is well-defined and $\eta(t)$ is non-extremal if and only if $\angle(\eta,t)<\pi$.
\end{corollary}
\begin{proof}
    The proof of (i), (ii) and (iii) is a direct consequence of Theorem \ref{theo:characterization_broken_geodesics} and the definitions. The angle is well-defined since it corresponds to the angle at $\eta(t)$ between the geodesics $\eta_{j+1}$ and $-\eta_j$ in the $\textup{CAT}(\kappa)$ space $B(\eta(t),\min\{\rho_\textup{CAT}(\Y),D_\kappa\})$. This angle is $\pi$ if and only if $\eta_j\star \eta_{j+1}$ is a local geodesic.
\end{proof}

We introduce another relevant quantity, depending on the branched covering.

\begin{definition}
Let $(\X,\Y,\Sigma,\pi)$ be a compact, locally $\textup{CAT}(\kappa)$ branched covering, let $\eta \in \Br$ and let $t\in \partial \eta^{-1}(\Sigma)$. The \emph{branching degree} of the concatenation point $\eta(t)$ is $\brdeg^\X(\eta,t):=\#\pi^{-1}(\eta(t+s))$ for $s>0$ small enough. 
\end{definition}

\begin{corollary}
    Let $(\X,\Y,\Sigma,\pi)$ be a compact, locally $\textup{CAT}(\kappa)$ branched covering, let $\eta \in \Br$ and let $t\in \partial \eta^{-1}(\Sigma)$. Then $\brdeg^\X(\eta,t)$ is well-defined and $\brdeg^\X(\eta,t) = 1$ if and only if $\eta(t)$ is of non-branching type.
\end{corollary}
\begin{proof}
    Call $\eta_j$, $\eta_{j+1}$ the local geodesic concatenating at time $t$. If $\eta_{j+1}$ is internal, i.e. $\eta(t)$ is of non-branching type, then $\#\pi^{-1}(\eta(t+s)) = 1$ for every $s$ small enough. In the other case, for every $s$ sufficiently small $\eta(t+s) \in \Y\setminus \Sigma$. By connectedness,  $\#\pi^{-1}(\eta(t+s))$ is constant and it is at least $2$.
\end{proof}

\begin{remark}
    The quantity $\brdeg^\X(\eta,t)$ depends on the branched covering $(\X,\Y,\Sigma,\pi)$. 
\end{remark}

The computation of the entropy of the geodesic flow of a branched covering will reduce to the computation of the number of different local geodesic lines that project to the same $\Sigma$-broken geodesic line of $\Y$. With this in mind we introduce the following definition.

\begin{definition}
\label{defin:gcn}
    Let $(\X,\Y,\Sigma,\pi)$ be a compact, locally $\textup{CAT}(\kappa)$ branched covering. Let $\gamma \in \LocGeod(\X)$ and let $\gamma = \underset{j\in J}{\star}\gamma_j$ be the unique decomposition of Lemma \ref{lemma:concatenation_of_geodesics_in_X}. Let $t\in \partial \gamma^{-1}(\Sigma)$ be a concatenation time, say between $\gamma_{j}$ and $\gamma_{j+1}$. The \emph{geodesic continuation number} at $t$ is the quantity
    $$\textup{gcn}(\gamma,t) := \#\{\tilde{\gamma}_{j+1}\text{ lift of }\pi\circ\gamma_{j+1} \,:\, \gamma_{j}\star\tilde{\gamma}_{j+1} \text{ is a local geodesic}\}.$$
\end{definition}

This quantity can be computed directly on $\Br$. The proof is the same as the proof of Theorem \ref{theo:characterization_broken_geodesics}.

\begin{proposition}
\label{prop:gcn_depends_only_on_projection}
    Let $(\X,\Y,\Sigma,\pi)$ be a compact, locally $\textup{CAT}(\kappa)$ branched covering, let $\gamma = \underset{j\in J}{\star}\gamma_j$ be the unique decomposition of Lemma \ref{lemma:concatenation_of_geodesics_in_X} and let $t\in \partial \gamma^{-1}(\Sigma)$. Let $\eta := \Pi(\gamma) \in \Br$, where $\Pi$ is the map of Proposition \ref{prop:projection_geodesics_to_broken_geodesics}. Then $\textup{gcn}(\gamma,t)$ equals
    \begin{itemize}
        \item[(a)] $\brdeg^\X(\eta,t)$ if and only if $\eta(t)$ is non-branching or extremal;
        \item[(b)] $\brdeg^\X(\eta,t)-1$ if and only if $\eta(t)$ is non-extremal.
    \end{itemize}
    In particular, if $\Pi(\gamma)=\Pi(\gamma')$ then $\textup{gcn}(\gamma,t) = \textup{gcn}(\gamma',t)$, hence $\textup{gcn}^\X\colon \Br \to \N$ is well-defined.
\end{proposition}

We finally introduce the geodesic continuation frequency functions.

\begin{definition}
\label{defin:frequency_functions}
    Let $(\X,\Y,\Sigma,\pi)$ be a compact, locally $\textup{CAT}(\kappa)$ branched covering and let $m\in \N$. The \emph{$m$-geodesic continuation frequency} function is
    $$\freq_m^\X\colon \Br \to \N,\quad \eta \mapsto \#\{t\in (-1,0] \cap  \partial\eta^{-1}(\Sigma)\,:\,\textup{gcn}^\X(\eta,t)=m\}.$$
    Notice that $\textup{gcn}^\X(\eta,t)$ is well-defined because of Proposition \ref{prop:gcn_depends_only_on_projection}.
\end{definition}

\begin{remark}
    The choice of the interval $(-1,0]$ in the definition of $\freq_m^\X$ is somehow arbitrary and it has been made for simplifying the technical computations. Later, we will have to consider the Birkhoff averages $A_n(\freq_m^\X)$ computed with respect to the flow $\Phi_{-1}$. With our choice, $$A_n(\freq_m^\X)(\eta)=\frac{\#\{t\in (-n,0] \cap  \partial\eta^{-1}(\Sigma)\,:\,\textup{gcn}^\X(\eta,t)=m\}}{n}.$$ 
\end{remark}

The frequency functions are Borel.

\begin{lemma}
\label{lemma:frequency_function_Borel}
    Let $(\X,\Y,\Sigma,\pi)$ be a compact, locally $\textup{CAT}(\kappa)$ branched covering and let $m \in \N$. Then the function $\freq_m^\X$ is well-defined and Borel. Moreover, $0\le \freq_m^\X \le \lceil1/\inj\rceil$.
\end{lemma}

\begin{proof}
    The function $\freq_m^\X$ is well defined because $\partial \eta^{-1}(\Sigma)$ is discrete and closed for every $\eta \in \Br$ by Lemma \ref{lemma:gamma^-1_is_discrete}, hence $\partial \eta^{-1}(\Sigma) \cap (-1,0]$ is finite. Moreover, the distance between two consecutive times in $\partial \eta^{-1}(\Sigma)$ with same geodesic continuation number is at least $\inj$, giving the bound $\freq_m^\X \le \lceil 1/\inj \rceil$.

    To show that $\freq_m^\X$ is Borel measurable, we consider the product space $\Br \times \mathbb{R}$. The set of all boundary concatenation times is 
    $$S := \{(\eta, t) \in \Br \times \mathbb{R} \,:\, t \in \partial\eta^{-1}(\Sigma)\}.$$ 
    It can be rewritten as
    \begin{equation*}
        S = \{(\eta, t) \,:\, \eta(t)\in \Sigma\} \cap \bigcap_{k=1}^\infty \bigcup_{q \in \mathbb{Q}} \{(\eta, t) \,:\, |t - q| < 1/k \text{ and } \sfd(\eta(q),\Sigma) > 0\}.
    \end{equation*}
    Hence, the set $S \subseteq \Br  \times \R$ is Borel. By Proposition \ref{prop:gcn_depends_only_on_projection}, $\textup{gcn}_\X(\eta, t)$ is completely determined by the branching degree $\text{br-deg}_\X(\eta, t)$ and the concatenation angle $\angle(\eta, t)$. By definition, $\text{br-deg}_\X(\eta, t) = \#\pi^{-1}(\eta(t+s))$ for sufficiently small $s > 0$, hence it is Borel measurable. By Corollary \ref{cor:non-extremal_iff_angle<pi}, an intersection is non-extremal if and only if $\angle(\eta, t) < \pi$. By the upper semicontinuity of the Alexandrov angle, the condition $\angle(\eta, t) < \pi$ defines a Borel set. Therefore, the set $S_m = \{(\eta, t) \in S \,:\, \textup{gcn}_\X(\eta, t) = m\}$ is a Borel subset of $\Br \times \mathbb{R}$.
    
    Finally, we evaluate $\freq_m^\X(\eta) = \#\{t \in (-1, 0] \,:\, (\eta, t) \in S_m\}$. The distance between any two consecutive intersection times is at least $\inj$. We partition the interval $(-1, 0]$ into $K = \lceil 1/\inj \rceil$ half-open intervals $I_j$, each of length strictly less than $\inj$. Because the length of each $I_j$ is strictly less than the minimal distance between intersections, the fiber $S_m \cap (\{\eta\} \times I_j)$ contains at most one point for every $\eta$.
    This allows us to rewrite the frequency function as:
    \begin{equation*}
        \freq_m^\X(\eta) = \sum_{j=1}^K \mathbbm{1}_{\pi_{\Br}(S_m \cap (\Br \times I_j))}(\eta),
    \end{equation*}
    where $\pi_{\Br}$ is the projection onto $\Br$. By the Lusin-Novikov theorem on Borel sets with countable sections, the projection $\pi_{\Br}(S_m \cap (\Br \times I_j))$ is a Borel set in $\Br$. Thus, $\freq_m^\X$ is Borel.
\end{proof}

We can prove Theorem \ref{theo_intro:freq_does_not_depend}.

\begin{T1}
    Let $(\X,\Y,\Sigma,\pi)$, $(\X',\Y,\Sigma,\pi')$ be two compact, pure-degree $k$, locally $\textup{CAT}(\kappa)$ branched coverings over the same couple $(\Y,\Sigma)$. Then $\freq_m^\X = \freq_m^{\X'}$ for every $m\in \N$. Moreover, $\freq_m^\X \equiv 0$ for $m\in \{2,\ldots, k-2\}$.
\end{T1}

\begin{proof}
    Proposition \ref{prop:gcn_depends_only_on_projection} implies that for every $\eta \in \Br$ and every $t\in \partial\eta^{-1}(\Sigma)$ we have $\textup{gcn}^\X(\eta,t) =  1$ if $\eta(t)$ is non-branching, $\textup{gcn}^\X(\eta,t) =  k$ if $\eta(t)$ is extremal and $\textup{gcn}^\X(\eta,t) =  k-1$ if $\eta(t)$ is non-extremal. The same for $\textup{gcn}^{\X'}(\eta,t)$. Being internal, extremal or non-extremal is a property of the broken geodesic $\eta$ which does not depend on the specific branched covering. This shows that $\freq_m^\X (\eta) = \freq_m^{\X'}(\eta)$ and that $\freq_m^\X\equiv 0$ for $m\in \{2,\ldots,k-2\}$.
\end{proof}

For technical reasons we need the following additional definition.
\begin{definition}
\label{defin:g_m^T,R}
    Let $(\X,\Y,\Sigma,\pi)$ be a compact, locally $\textup{CAT}(\kappa)$ branched covering. Let $m\in \N$ and $T,R\ge 0$. We define
    $$\freq_m^{\X,T,R}\colon \Br \to [0,+\infty), \quad \eta \mapsto \#\{t\in (-T,R] \cap  \partial\eta^{-1}(\Sigma)\,:\,\textup{gcn}^\X(\eta,t)=m\}.$$
\end{definition}

The next proposition links the geodesic continuation frequency functions to the computation of the topological entropy of the geodesic flow on a branched covering. 

\begin{proposition}
\label{prop:counting_preimages_of_broken_geodesics}
    Let $(\X,\Y,\Sigma,\pi)$ be a compact, locally $\textup{CAT}(\kappa)$ branched covering, let $\Pi\colon \LocGeod(\X) \to \Br$ be the map of Proposition \ref{prop:projection_geodesics_to_broken_geodesics}, let $m\in \N$, let $\eta \in \Br$ and let $T,R\ge 0$. Then
    \begin{equation}
        \label{eq:number_of_preimages}
        \prod_{m=1}^\infty m^{\freq_m^{\X,T,R}(\eta)} \le \#\left\{ \gamma\restr{(-T,R]}\,:\, \gamma \in \Pi^{-1}(\eta)\right\} \le  \textup{deg}(\pi)\prod_{m=1}^\infty m^{\freq_m^{\X,T,R}(\eta)}.
    \end{equation}
\end{proposition}
\begin{proof}
    Let us enumerate the elements of $(-T,R] \cap \partial\eta^{-1}(\Sigma)$ as $\{t_1,\ldots, t_M\}$ with $t_j<t_{j+1}$. By definition, the local geodesic segment $\eta_j = \eta \restr{[t_j,t_{j+1}]}$ has exactly $m_j = \textup{gcn}^\X(\eta,t_j)$ possible lifts that form a local geodesic when concatenated with a lift of the previous local geodesic segment, for every $j$. If we do not consider the lifts of the segment ending in $t_1$ we get the lower bound in \eqref{eq:number_of_preimages}. If we consider also them, that are at most $\textup{deg}(\pi)$, we obtain the upper bound.
\end{proof}

\section{Examples}
\label{sec:examples}

In this section, we introduce the two main sources of examples of negatively branched coverings that we will study in the sequel: wedges and Gromov-Thurston manifolds. 

\subsection{Wedges}
\label{subsec:wedges}
 
We start with the prototypical example of a branched covering: wedges. Let $(\Y,\sfd_\Y)$ be a compact, locally $\textup{CAT}(\kappa)$ space, let $\Sigma \subseteq \Y$ be a closed, locally convex subset and let $k\ge 1$. Let $\vee^k(\Y,\Sigma)$ be the space obtained by gluing $k$ isometric copies of $\Y$ through the identification of $\Sigma$, with the map $\pi_k\colon \vee^k(\Y,\Sigma) \to \Y$ which is the identity on each copy. 
It is called the $k$-wedge of $\Y$ along $\Sigma$. Then $\vee^k(\Y,\Sigma)$ is an example of a branched covering of pure degree $k$ as introduced in Section \ref{sec:top_branched_coverings}. This construction implies that for every compact, locally $\textup{CAT}(\kappa)$ space and every closed and locally convex $\Sigma \subseteq \Y$, the set $\Br$ is not empty.
The reader may test the wedge and the corresponding space $\Br$ for the couples $(\Y,\Sigma)$ of Section \ref{subsec:Examples}.

\subsection{Gromov-Thurston manifolds}
\label{subsec:GT-manifolds}

Gromov-Thurston manifolds are interesting, non-trivial examples of negatively curved branched coverings. As for wedges, they come in sequences of spaces over a fixed pair $(\Y,\Sigma)$.
First of all we recall the definition of cyclic branched covering. Let $\Y$ be a compact, connected, oriented, $n$-dimensional, smooth manifold and let $\Sigma \subseteq \Y$ be an oriented, embedded, codimension $2$ submanifold such that $\Sigma = \partial N$, where $N\subseteq \Y$ is an oriented, embedded, codimension $1$ oriented submanifold. Given $k \ge 1$, the $k$-cyclic covering of $\Y$ with singularity $\Sigma$ is the space obtained as follows. We cut $\Y$ along the interior of $N$: the resulting space has two boundaries $N^+, N^-$ glued along the common boundary $\Sigma$. We prepare $k$ copies $\X_1,\ldots,\X_k$ of such cut manifolds and we glue them identifying $N_i^+$ with $N_{i+1}^-$, cyclically modulo $k$. The resulting space is denoted by $\X$: it is a topological $n$-manifold. There is a natural copy of $\Sigma$ in $\X$ and a natural map $\pi\colon \X \to \Y$ which sends an element of $\X_i$ in the corresponding point in $\Y$. Outside of $\Sigma$, the map is a degree $k$ covering, hence $\pi$ is a branched covering of pure-degree $k$.

Gromov-Thurston manifolds are particular examples of cyclic branched coverings. The setting is the following: $\Y$ is taken to be a closed hyperbolic $n$-manifold and $\Sigma$ is taken to be a totally geodesic, codimension $2$ submanifold satisfying the conditions above. A pair $(\Y,\Sigma)$ satisfying these conditions is called a \emph{Gromov-Thurston pair}. Explicit examples of such pairs $(\Y,\Sigma)$ have been constructed in \cite{GromovThurston87}. The associated $k$-cyclic branched covering is denoted by $\textup{GT}^k(\Y,\Sigma)$. Under these assumptions, $\Y$ is obviously a locally $\CATminus$ space and $\Sigma$ is locally convex, hence $\textup{GT}^k(\Y,\Sigma)$ is one of the examples of pure-degree $k$ branched coverings of Section \ref{subsec:branched_coverings_definition}. A local representation of a Gromov-Thurston manifold with $k=8$ in dimension $3$ is given in Figure \ref{fig:model.space}.

\begin{figure}[h!]
	\centering
    \def\svgwidth{5in}
    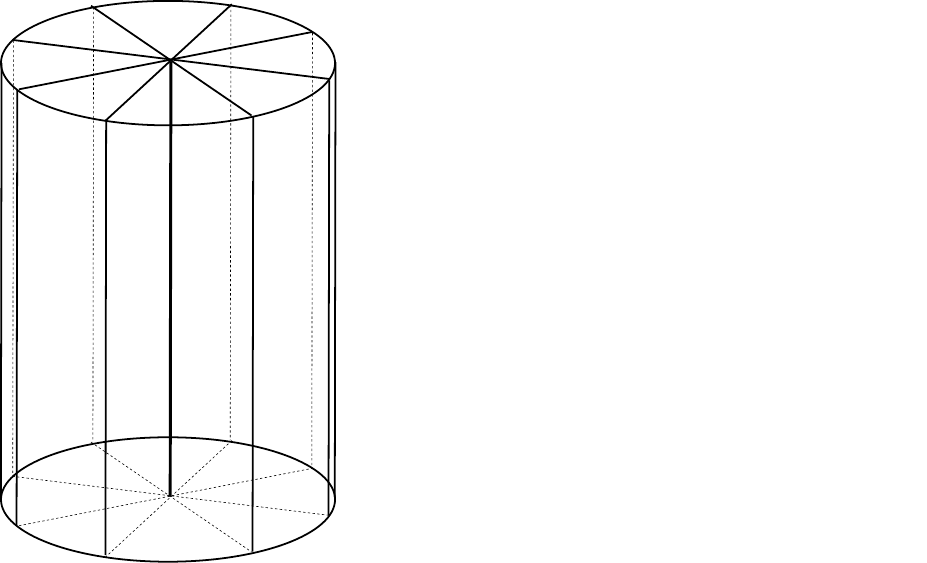
    \caption{Left: a picture of the local geometry of $\textup{GT}^8(\Y,\Sigma)$. The two dashed curves are each an extension of the local geodesic $\gamma$. Right: the region around the singularity set from a top-view. Any two local geodesic extensions of $\gamma$ can be distinguished in this top-view by the direction in which they leave the singularity set. We emphasize that every `slice' has an angle of $2\pi$.}
    \label{fig:model.space}
\end{figure}

\begin{remark}
    There are many other examples of branched coverings of smooth manifolds. First of all, the class of cyclic coverings contains other examples besides Gromov-Thurston manifolds. It is well known that for every $k\in \N$, if $\Y$ is a smooth, oriented, closed manifold of dimension $n\ge 3$ and $\Sigma \subseteq \Y$ is a closed, oriented, embedded $n-2$-submanifold, then the $k$-cycling branched cover of $\Y$ along $\Sigma$ exists if and only if $H_{n-2}([\Sigma]) = 0 \in H_{n-2}(\Y,\Z/k\Z)$, see for instance \cite[Proposition 2.3]{Hamenstadt2025}. As soon as $\Y$ as above carries a locally $\textup{CAT}(\kappa)$ metric such that $\Sigma$ is locally convex, we obtain branched coverings as in Section \ref{subsec:branched_coverings_definition}. Many interesting examples, like the ones in \cite{MostowSiu1980, Deraux2005},  are not of this type, strictly speaking. However, some results about the locally $\textup{CAT}(\kappa)$ properties of branched coverings over multiple, orthogonal, submanifolds are deduced in \cite{CharneyDavis1993, Allcock2000}.

In literature, a map $\pi\colon \X \to \Y$ between compact topological $n$-manifolds, possibly with boundary, such that $\pi^{-1}(\partial\Y)\subseteq \partial\X$, is called a branched covering if it is open and it has discrete fibers. Then, the set $B := \{x\in \X\,:\,\pi \text{ is not a local homeomorphism at }x\}$ and $\Sigma:=\pi(B)$ are sets of codimension at least $2$, see \cite{Chernavskii63, Vaisala67}. Moreover, the restriction of $\pi$ to $\X\setminus B$ is a covering map. If, additionally, $\#\pi^{-1}(y)=1$ for every $y\in \Sigma$ and $\Y$ is equipped with a locally $\textup{CAT}(\kappa)$ metric with respect to which $\Sigma$ is locally convex, then $\pi$ is a branched covering in the sense of Section \ref{subsec:branched_coverings_definition}. For more general results about branched coverings between topological manifolds we refer to \cite{Alexander1920, Fox1957}. For examples involving homology manifolds, see \cite{HeinonenRickman2002}. For applications in low-dimensional topology we suggest to read \cite{Hilden1974, Montesinos1974, BersteinEdmonds1979, Piergallini1995}.
\end{remark}

\section{Entropy of a branched covering}
\label{sec:entropy_branched_covering}

In this section we will relate the entropy of a locally $\textup{CAT}(\kappa)$ branched covering to the geometric properties of the base space, proving Theorem \ref{theo:intro_entropy_formula_branched_coverings}. The functions $\freq_m^\X$ are the one of Definition \ref{defin:frequency_functions}.

\begin{T2}
    Let $(\X,\Y,\Sigma,\pi)$ be a compact, locally $\textup{CAT}(\kappa)$ branched covering. Then
    \begin{equation}
        \label{eq:h_top=pressure}
        h_\textup{top}(\LocGeod(\X),\Phi_1) = \mathscr{P}\left(\Br,\Phi_1,\sum_{m\ge 1}\log(m)\freq_m^\X\right).
    \end{equation}
\end{T2}

\begin{proof}
    Let $\Pi\colon \LocGeod(\X) \to \Br$ be the map of Proposition \ref{prop:projection_geodesics_to_broken_geodesics}. We start by proving the next claim.

    \textbf{Claim.} For every $\eta \in \Br$ it holds that
    \begin{equation}
        \label{eq:h_top_<=_iota}
        h_\textup{top}(\Pi^{-1}(\eta),\Phi_1) = \sum_{m\ge 1}\log(m)\overline{\freq_m^\X}(\eta),
    \end{equation}
    where $\overline{\freq_m^\X} = \lims_{n\to +\infty}\frac{1}{n}\sum_{i=0}^{n-1} \freq_m^\X \circ \Phi_{-i}$.

    Recalling the definition in \eqref{eq:defin_top_entropy_fibers} we have
    \begin{equation}
        h_\textup{top}(\Pi^{-1}(\eta),\Phi_1) = \lim_{r\to 0} \lims_{n\to +\infty} \frac{1}{n} \log \textup{Cov}_{{\sf D}^n}(\Pi^{-1}(\eta),r),
    \end{equation}
    where ${\sf D}$ is the distance on $\LocGeod(\X)$ defined in \eqref{eq:defin_D_a}. We fix $r>0$ and we choose $R_r > 0$ such that $\sup_{\vert s \vert\ge R_r} 2\vert s \vert e^{-\vert s \vert} \le r/2$. For every $n\in \N$ we find a set $\mathcal{A}_r^n :=\{\gamma_1,\ldots,\gamma_N\} \in \Pi^{-1}(\eta)$ such that $\gamma_i \restr{(-n-R_r,R_r]} \neq \gamma_j \restr{(-n-R_r,R_r]}$ for $i\neq j$ and such that for every $\gamma \in \Pi^{-1}(\eta)$ there exists $i$ such that $\gamma \restr{(-n-R_r,R_r]} = \gamma_i \restr{(-n-R_r,R_r]}$. Such a set exists and has cardinality at most $\textup{deg}(\pi)\prod_{m \ge 1}^\infty m^{\freq_m^{\X,n + R_r,R_r}(\eta)} < \infty$ by Proposition \ref{prop:counting_preimages_of_broken_geodesics}. We show that $\mathcal{A}_r^n$ is $r$-dense in $\Pi^{-1}(\eta)$ with respect to ${\sf D}^n$. For every $\gamma \in \Pi^{-1}(\eta)$ let $\gamma_i \in \mathcal{A}_r^n$ be such that $\gamma \restr{(-n-R_r,R_r]} = \gamma_i \restr{(-n-R_r,R_r]}$. This means that the local geodesic lines $\Phi_h(\gamma)$ and $\Phi_h(\gamma_i)$ coincide on the interval $(-R_r,R_r]$ for every $h\in \{0,\ldots,n\}$. Therefore,
    $${\sf D}^n(\gamma,\gamma_i) = \sup_{0\le h \le n} {\sf D}(\Phi_h(\gamma), \Phi_h(\gamma_i)) \le 2\sup_{\vert s \vert \ge R_r} 2\vert s \vert e^{-\vert s \vert} \le r,$$
    where for the first inequality we used the triangular inequality. Therefore, we have that
    \begin{equation}
        \begin{aligned}
            h_\textup{top}(\Pi^{-1}(\eta),\Phi_1) \le \lim_{r\to 0}\lims_{n\to +\infty}\frac{1}{n}\log\#\mathcal{A}_r^n &\le \lim_{r\to 0}\lims_{n\to +\infty}\frac{1}{n}\log \left(\textup{deg}(\pi)\prod_{m \ge 1} m^{\freq_m^{\X,n+R_r,R_r}(\eta)}\right) \\
            &= \lim_{r\to 0}\lims_{n\to +\infty}\sum_{m\ge 1}\log(m) \frac{\freq_m^{\X,n+R_r,R_r}(\eta)}{n} \\
            &=\lim_{r\to 0}\sum_{m\ge 1}\log(m) \lims_{n\to +\infty}\frac{\freq_m^{\X,n+R_r,R_r}(\eta)}{n}.
        \end{aligned}
    \end{equation}
    Here, recall that all infinite products and sums are finite because the branched covering has finite degree. The very Definition \ref{defin:g_m^T,R} of $\freq_m^{\X,n+R_r,R_r}$ gives that $$\lims_{n\to +\infty}\frac{\freq_m^{\X,n+R_r,R_r}(\eta)}{n} = \lims_{n\to +\infty}\frac{\freq_m^{\X,n,0}(\eta)}{n} = \lims_{n\to +\infty}\frac{1}{n}\sum_{i=0}^{n-1}(\freq_m^\X\circ\Phi_{-i})(\eta) = \overline{\freq_m^\X}(\eta).$$
    Taking the limit for $r$ going to zero we obtain the first inequality.    
    
    We now show the other inequality. Let $\{t_1<t_2<\ldots < t_M\} := (-n,0]\cap \partial \eta^{-1}(\Sigma)$. Consider two distinct lifts $\gamma_j,\gamma_j'$ of the geodesic segment $\eta_j\colon [t_j,t_{j+1}]\to \Y$ for $1\le j< M$. Corollary \ref{cor:uniform_distance_segments} implies that there exists $t\in [t_j,t_{j+1}]$ such that $\sfd(\gamma_j(t),\gamma_j'(t)) \ge \rho_\textup{CAT}(\Y)$. This shows that if $\gamma, \gamma' \in \Pi^{-1}(\eta)$ are such that $\sfd(\gamma(s),\gamma'(s)) > 0$ for some $s\in [t_1,t_M]$ then there exists $t\in [t_1,t_M]$ such that $\sfd(\gamma(t),\gamma'(t)) \ge \rho_\textup{CAT}(\Y)$. The set $\{\gamma\restr{[t_1,t_M]}\,:\,\gamma \in \Pi^{-1}(\eta)\}$ has cardinality at least $\prod_{m\ge 1} m^{\freq_m^{\X,n,0}(\eta)}$ by Proposition \ref{prop:counting_preimages_of_broken_geodesics}. Let $\mathcal{A}^n = \{\gamma_1,\ldots,\gamma_N\}$ be a maximal set such that $\gamma_i \restr{[t_1,t_M]} \neq \gamma_j \restr{[t_1,t_M]}$. The discussion above says that $\mathcal{A}^n$ is $\rho_\textup{CAT}(\Y)$-separated with respect to ${\sf D}^n$. Therefore,
    
    \begin{equation}
        \begin{aligned}
            h_\textup{top}(\Pi^{-1}(\eta),\Phi_1) \ge \lims_{n\to +\infty}\frac{1}{n}\log\#\mathcal{A}^n &\ge \lims_{n\to +\infty}\frac{1}{n}\log \left(\prod_{m \ge 1} m^{\freq_m^{\X,n,0}(\eta)}\right) \\
            &= \lims_{n\to +\infty}\sum_{m\ge 1}\log(m) \frac{\freq_m^{\X,n,0}(\eta)}{n} \\
            &= \sum_{m\ge 1}\log(m) \overline{\freq_m^\X}(\eta),
        \end{aligned}
    \end{equation}
    where in the last equality we argued as before. This concludes the proof of the claim. 

    Recall that $\int \overline{\freq_m^\X}\,\d\mu = \int \freq_m^\X\,\d\mu$ for every $\mu \in\mathcal{M}_1(\Br,\Phi_1)$ because of Theorem \ref{theo:Birkhoff} and the fact that every measure which is $\Phi_1$-invariant is also $\Phi_{-1}$-invariant.
    Let now $\nu \in \mathcal{M}_1(\LocGeod(\X),\Phi_1)$ and consider $\Pi_\#\nu =: \mu  \in \mathcal{M}_1(\Br,\Phi_1)$ because of Lemma \ref{lemma:pushforward_is_invariant} and Proposition \ref{prop:projection_geodesics_to_broken_geodesics}. Proposition \ref{prop:Ledrappier_entropy_fibers} implies that
    \begin{equation*}
        \begin{aligned}
            h_\nu \le h_\mu + \int h_\textup{top}(\Pi^{-1}(\eta),\Phi_1)\,\d\mu = h_\mu + \int \sum_{m\ge 1}\log(m)\freq_m^\X\,\d\mu
        \end{aligned}
    \end{equation*}
    By the arbitrariness of $\nu$ and because of Proposition \ref{prop:variational_principle} we deduce that 
    $$h_\textup{top}(\LocGeod(\X),\Phi_1) \le \mathscr{P}\left(\Br,\Phi_1,\sum_{m\ge 1}\log(m)\freq_m^\X\right).$$

    For the other inequality we fix $\mu \in \mathcal{M}_1(\Br,\Phi_1)$ and $\varepsilon > 0$. By Propositions \ref{prop:Ledrappier_entropy_fibers} and \ref{prop:projection_geodesics_to_broken_geodesics} we can find $\nu_\varepsilon \in \mathcal{M}_1(\LocGeod(\X),\Phi_1)$ such that $\Pi_\#\nu_\varepsilon = \mu$ and 
    $$h_{\nu_\varepsilon}\ge h_\mu + \int h_\textup{top}(\Pi^{-1}(\eta),\Phi_1)\,\d\mu - \varepsilon.$$
    Rewriting the right hand side using \eqref{eq:h_top_<=_iota} we have
    \begin{equation}
    \label{eq:h_mu_nu}
        h_\mu + \int \sum_{m \ge 1} \log(m)\freq_m^\X\,\d\mu \le h_{\nu_\varepsilon} + \varepsilon.
    \end{equation}
    By Proposition \ref{prop:variational_principle} we obtain that 
    $$h_\mu + \int \sum_{m\ge 1} \log(m)\freq_m^\X\,\d\mu \le h_\textup{top}(\LocGeod(\X),\Phi_1) + \varepsilon.$$
    The arbitrariness of $\varepsilon$ and $\mu$ gives that
    $$\mathscr{P}\left(\Br,\Phi_1,\sum_{m\ge 1}\log(m)\freq_m^\X\right) \le h_\textup{top}(\LocGeod(\X),\Phi_1).$$
\end{proof}

We specialize the result to the locally $\CATminus$ case. Here, the map $\Pi$ is the one of Proposition \ref{prop:projection_geodesics_to_broken_geodesics}, the measure $\mm_{\textup{BM}}$ is the Bowen-Margulis measure of Proposition \ref{prop:top_ent_geod_flow_BM_measure} and the numbers $\overline{\freq_m^\X}(\mm_\textup{BM})$ are the ones of Theorem \ref{theo:Birkhoff} applied to the functions $\freq_m^\X$ and the dynamical system $(\LocGeod(\X),\Phi_{-1})$.

\begin{corollary}
\label{cor:m_BM_unique_measure_max_pressure}
    Let $(\X,\Y,\Sigma,\pi)$ be a compact, locally $\textup{CAT}(-1)$ branched covering. Then
    $$h_\textup{top}(\LocGeod(\X),\Phi_1) = h_{\Pi_\#\mm_{\textup{BM}}} + \sum_{m\ge 1}\log(m)\overline{\freq_m^\X}(\mm_\textup{BM})$$
    and for every other measure $\mu \in \mathcal{M}_1(\Br,\Phi_1)$ it holds that $$h_{\mu} + \int\sum_{m\ge 1}\log(m)\freq_m^\X\,\d\mu < h_{\Pi_\#\mm_{\textup{BM}}} + \sum_{m\ge 1}\log(m)\overline{\freq_m^\X}(\mm_\textup{BM}).$$
\end{corollary}
\begin{proof}
    By Proposition \ref{prop:top_ent_geod_flow_BM_measure}, it holds that $h_{\mm_\textup{BM}} = h_\textup{top}(\LocGeod(\X),\Phi_1)$. The proof of Theorem \ref{theo:intro_entropy_formula_branched_coverings} implies that
    \begin{equation}
        \begin{aligned}
            h_\textup{top}(\LocGeod(\X),\Phi_1) &= h_{\Pi_\#\mm_{\textup{BM}}} + \int \sum_{m \ge 1} \log(m)\freq_m^\X\,\d\Pi_\#\mm_{\textup{BM}} \\
            &= h_{\Pi_\#\mm_{\textup{BM}}} + \sum_{m \ge 1} \log(m)\overline{\freq_m^\X}(\mm_\textup{BM}),
        \end{aligned}
    \end{equation}
    where we used that $\mm_\textup{BM}$ is ergodic, by Proposition \ref{prop:top_ent_geod_flow_BM_measure}, and Theorem \ref{theo:Birkhoff}.

    We consider a measure $\mu\in \mathcal{M}_1(\Br,\Phi_1)$ such that
    $$h_{\mu} + \int\sum_{m\ge 1}\log(m)\freq_m^\X\,\d\mu < h_{\Pi_\#\mm_{\textup{BM}}} + \sum_{m\ge 1}\log(m)\overline{\freq_m^\X}(\mm_\textup{BM}),$$
    i.e. realizing the supremum in the topological pressure on the right-hand side of Theorem \ref{theo:intro_entropy_formula_branched_coverings}. For every $\varepsilon > 0$ let $\nu_\varepsilon$ be measures as in \eqref{eq:h_mu_nu}. Up to a subsequence, we can assume that $\nu_\varepsilon$ converges weakly to $\nu \in \mathcal{M}_1(\LocGeod(\X), \Phi_1)$, so $\Pi_\#\nu = \mu$ by continuity of $\Pi$. Moreover, the geodesic flow $(\LocGeod(\X), \Phi_1)$ is $h$-expansive by Proposition \ref{prop:top_ent_geod_flow_BM_measure}, so 
    $$h_\nu \ge \lims_{\varepsilon \to 0}  \,h_{\nu_\varepsilon} \ge h_\textup{top}(\LocGeod(\X), \Phi_1)$$
    by Proposition \ref{prop:measure_entropy_upper_semicontinuous}.
    Proposition \ref{prop:top_ent_geod_flow_BM_measure} implies that $\nu = \mm_\textup{BM}$ and so that $\mu = \Pi_\#\mm_\textup{BM}$.
\end{proof}

We are ready to show the following more precise version of Theorem \ref{theo:intro_entropy_does_not_depend}.

\begin{theorem}
\label{theo:wedges_maximal_entropy}
    Let $(\X,\Y,\Sigma,\pi)$ be a compact, locally $\textup{CAT}(\kappa)$ branched covering. Then
    $$h_\textup{top}(\LocGeod(\Y),\Phi_1) \le h_\textup{top}(\LocGeod(\X),\Phi_1).$$
    If $(\X',\Y,\Sigma,\pi')$ is another compact, locally $\textup{CAT}(\kappa)$, pure-degree $k$ branched covering, with $k=\textup{deg}(\pi)$, then
    $$h_\textup{top}(\LocGeod(\X),\Phi_1) \le h_\textup{top}(\LocGeod(\X'),\Phi_1).$$
    In particular, if both branched coverings have the same pure degree then their geodesic flow have same topological entropy. Moreover, if  $\kappa = -1$ then $\Pi_\#\mm_{\textup{BM}}^\X = \Pi_\#'\mm_{\textup{BM}}^{\X'}$, where $\Pi,\Pi'$ are the maps of Proposition \ref{prop:projection_geodesics_to_broken_geodesics} relative to $\X$ and $\X'$ and where $\mm_{\textup{BM}}^\X$, $\mm_{\textup{BM}}^{\X'}$ are the Bowen-Margulis measures relative to $\X$ and $\X'$.
\end{theorem}
\begin{proof}
    Let $\Pi\colon \LocGeod(\X) \to \Br$ be the map of Proposition \ref{prop:projection_geodesics_to_broken_geodesics}. Proposition \ref{prop:Ledrappier_entropy_fibers} implies that
    \begin{equation}
        \label{eq:h_broken<h_branched}
        h_\textup{top}(\Br,\Phi_1) \le h_\textup{top}(\LocGeod(\X),\Phi_1).
    \end{equation}
    Moreover, $\LocGeod(\Y)\subseteq \Br$ by Proposition \ref{prop:projection_geodesics_to_broken_geodesics} and it is $\Phi_1$-invariant. Therefore, $h_\textup{top}(\LocGeod(\Y),\Phi_1)\le h_\textup{top}(\Br,\Phi_1)$. Combining the two inequalities, we obtain that
    $$h_\textup{top}(\LocGeod(\Y),\Phi_1) \le h_\textup{top}(\LocGeod(\X),\Phi_1).$$
    
    Let $(\X',\Y,\Sigma,\pi')$ be a compact, locally $\textup{CAT}(\kappa)$, pure-degree $k$ branched covering, with $k=\textup{deg}(\pi)$. Proposition \ref{prop:gcn_depends_only_on_projection} implies that $\sum_{m\ge 1} \log(m)\freq_m^\X \le \sum_{m\ge 1} \log(m)\freq_m^{\X'}$. Theorem \ref{theo:intro_entropy_formula_branched_coverings} gives that
    $$h_\textup{top}(\LocGeod(\X), \Phi_1) \le h_\textup{top}(\LocGeod(\X'),\Phi_1).$$
    In particular, if both $(\X,\Y,\Sigma,\pi)$ and $(\X',\Y,\Sigma,\pi')$ have pure degree $k$, then 
    $$h_\textup{top}(\LocGeod(\X),\Phi_1) = h_\textup{top}(\LocGeod(\X'),\Phi_1).$$

    Finally, assume that $\kappa = -1$ and that $(\X,\Y,\Sigma,\pi)$ and $(\X',\Y,\Sigma,\pi')$ have pure degree $k$. The equality $h_\textup{top}(\LocGeod(\X),\Phi_1) = h_\textup{top}(\LocGeod(\X'),\Phi_1)$, together with Corollary \ref{cor:m_BM_unique_measure_max_pressure}, gives that
    $$h_{\Pi_\#\mm_{\textup{BM}}^\X} + \sum_{m\ge 1}\log(m)\overline{\freq_m^\X}(\mm_\textup{BM}^\X) = h_{\Pi_\#\mm_{\textup{BM}}^{\X'}} + \sum_{m\ge 1}\log(m)\overline{\freq_m^\X}(\mm_\textup{BM}^{\X'}).$$
    The uniqueness part of Corollary \ref{cor:m_BM_unique_measure_max_pressure} implies that $\Pi_\#\mm_{\textup{BM}}^{\X} = \Pi_\#\mm_{\textup{BM}}^{\X'}$.
    
\end{proof}

\begin{remark}
\label{rmk:entropy_if_Sigma_trivial}

    If $\Sigma^\circlearrowleft = \emptyset$ then every $\gamma \in \LocGeod(\X)$ intersects $\Sigma$ in at most one point. So, $A_n(\freq_m^\X)(\eta) \to 0$ as $n\to \infty$ for every $\eta \in \Br$, and Proposition \ref{prop:ergodic_optimization_bounds} gives $\textup{Erg-Opt}(\freq_m^\X) = 0$. This means that $\int \freq_m^\X\,\d\nu = 0$ for every $\nu \in \mathcal{M}_1(\Br,\Phi_1)$. Theorem \ref{theo:intro_entropy_formula_branched_coverings} gives
    $$h_\textup{top}(\LocGeod(\X),\Phi_1) = h_\textup{top}(\LocGeod(\Y),\Phi_1).$$   
    Example (2) of Section \ref{subsec:Examples} shows that it may happen that $\Sigma^\circlearrowleft \neq \emptyset$ but $h_\textup{top}(\LocGeod(\X),\Phi_1) = h_\textup{top}(\LocGeod(\Y),\Phi_1)$: the $2$-wedge $\vee^2(\Y,\Sigma)$ is simply $\mathbb{S}^1$ whose geodesic flow has topological entropy zero. We expect that this phenomenon does not happen for $k\ge 3$. We conjecture that if $\Sigma^\circlearrowleft \neq \emptyset$ and $k\ge 3$, then $h_\textup{top}(\LocGeod(\vee^k(\Y,\Sigma),\Phi_1) > h_\textup{top}(\LocGeod(\Y),\Phi_1)$. We will not investigate this matter further in this paper.
\end{remark}

\section{Asymptotic behavior of the entropy}
\label{sec:asymptotic}

We consider a compact, locally $\textup{CAT}(\kappa)$ space $(\Y,\sfd_\Y)$ and a closed, locally convex subset $\Sigma \subseteq \Y$. For every $k\in \N$, we denote by $\freq_m^k\colon \Br \to \N$ the functions of Theorem \ref{theo_intro:freq_does_not_depend}, that do not depend on the particular choice of the pure-degree $k$ branched covering $(\X,\Y,\Sigma,\pi)$. Actually, they do not even depend on $k$ in the sense we now explain.

\begin{definition}
    Let $(\Y,\sfd_\Y)$ be a compact, locally $\textup{CAT}(\kappa)$ space and $\Sigma \subseteq \Y$ be closed and locally convex. We define the functions
    \begin{equation*}
        \begin{aligned}
            \freq_{\textup{non-br}}\colon \Br \to \N, \quad &\eta \mapsto \#\{t\in (-1,0]\cap \partial\eta^{-1}(\Sigma)\,:\, \eta(t) \textup{ is non-branching}\},\\
            \freq_{\textup{non-ext}}\colon \Br \to \N, \quad &\eta \mapsto \#\{t\in (-1,0]\cap \partial\eta^{-1}(\Sigma)\,:\, \eta(t) \textup{ is non-extremal}\},\\
            \freq_{\textup{ext}}\colon \Br \to \N, \quad &\eta \mapsto \#\{t\in (-1,0]\cap \partial\eta^{-1}(\Sigma)\,:\, \eta(t) \textup{ is extremal}\}.
        \end{aligned}
    \end{equation*}
\end{definition}

The proof of Theorem \ref{theo_intro:freq_does_not_depend} shows the next property.

\begin{corollary}
\label{cor:freq_k=freq_ext}
    Let $(\X,\Y,\Sigma,\pi)$ be a compact, locally $\textup{CAT}(\kappa)$ branched covering of pure degree $k$. Then $\freq_1^\X = \freq_{\textup{non-br}}$, $\freq_{k-1}^\X = \freq_\textup{non-ext}$, $\freq_k^\X = \freq_\external$.
\end{corollary}

The last frequency function we define is the \emph{branching frequency function} 
$$\freq \colon \Br \to \N,\quad \freq = \freq_\textup{non-ext} + \freq_\textup{ext}.$$

\begin{proposition}
\label{prop:erg_opt_h}
    Let $(\Y,\sfd_\Y)$ be a compact, locally $\textup{CAT}(\kappa)$ space and $\Sigma \subseteq \Y$ be a closed, locally convex subset. Then $\Sigma^\circlearrowleft \neq \emptyset$ if and only if $\textup{Erg-Opt}(\freq) > 0$. Moreover, 
    \begin{equation}
        \label{eq:erg_opt_wedges}
        \textup{Erg-Opt}(\freq) = \frac{1}{\inj}.
    \end{equation}
\end{proposition}
The equality in \eqref{eq:erg_opt_wedges} makes sense and is true also in the case $\Sigma^\circlearrowleft = \emptyset$.

\begin{proof}
    If $\Sigma^\circlearrowleft = \emptyset$ then every $\eta \in \Br$ intersects $\Sigma$ in at most one point. As noticed in Remark \ref{rmk:entropy_if_Sigma_trivial}, this implies that $\int \freq\,\d\nu = 0$ for every $\nu \in \mathcal{M}_1(\Br,\Phi_1)$, hence $\textup{Erg-Opt}(\freq)=0$. On the other hand, if $\Sigma^\circlearrowleft \neq \emptyset$ we take $\sigma \colon [0,\ell] \to \Y$ to be a non-trivial local geodesic realizing the infimum in the definition \eqref{eq:defin_normal_inj_radius} of $\inj$, see Proposition \ref{prop:normal_inj_radius_minimizing_curve}. We define $\eta \in \textup{Broken}(\Y,\Sigma)$ as the bi-infinite alternate concatenation of $\sigma$ and $-\sigma$. The angle condition of Proposition \ref{prop:normal_inj_radius_minimizing_curve} satisfied by $\sigma$, together with the characterization of Theorem \ref{theo:characterization_broken_geodesics}, gives that $\eta \in \Br$. We apply Proposition \ref{prop:ergodic_optimization_bounds} to the periodic broken geodesic $\eta$ to get
    $$\textup{Erg-Opt}(\freq) \ge \lims_{n\to +\infty} A_n(\freq)(\eta) = \frac{1}{\inj} > 0.$$
    Let $\eta = \underset{j\in J}{\star}\eta_j \in \Br$ be arbitrary. If $\eta_j$ is external then $\ell(\eta_j) \ge \inj$, by definition of normal injectivity radius. Hence, the distance between two concatenation times $t,s$ counted by $\freq$ is at least $\inj$. Therefore, $\freq(\eta) \le \lceil1/\inj\rceil$ for every $\eta \in \Br$. In the same way, and using Proposition \ref{prop:ergodic_optimization_bounds}, we deduce that
    $$\textup{Erg-Opt}(\freq) \le \lims_{n\to +\infty}\frac{1}{n}\left\lceil\frac{n}{\inj}\right\rceil = \frac{1}{\inj},$$
    concluding the proof of \eqref{eq:erg_opt_wedges}.
\end{proof}

We denote by $h_\textup{top}(\Y,\Sigma,k) := h_\textup{top}(\LocGeod(\X),\Phi_1)$, where $(\X,\Y,\Sigma,\pi)$ is any branched covering of pure-degree $k$: it does not depend on the choice of the branched covering by Theorem \ref{theo:intro_entropy_does_not_depend}. We prove Theorem \ref{theo_intro:asymptotic_formula}.

\begin{T3}
    Let $(\Y,\sfd_\Y)$ be a compact, locally $\textup{CAT}(\kappa)$ space and let $\Sigma \subseteq \Y$ be closed and locally convex. Assume that $h_\textup{top}(\Br,\Phi_1) < \infty$ and that $\Sigma^\circlearrowleft \neq \emptyset$. Then for every $k\ge 2$ it holds that
    $$\frac{\log(k-1)}{\inj}\le h_\textup{top}(\Y,\Sigma,k) \le \frac{\log(k)}{\inj} + h_\textup{top}(\Br,\Phi_1),$$
    hence
    $$\lim_{k\to +\infty} \frac{h_\textup{top}(\Y,\Sigma,k)}{\log(k)} = \frac{1}{\inj}.$$
    If $\kappa=-1$, then $h_\textup{top}(\Br,\Phi_1) \le h(\tilde{\Y};\pi_1(\Y)) + \log(2)/\inj$.
\end{T3}

\begin{proof}
    Let $k\ge 2$ be fixed. We know that $\textup{Erg-Opt}(\freq) = 1/\inj > 0$ by Proposition \ref{prop:erg_opt_h}. For every $\varepsilon > 0$ let $\mu_\varepsilon \in \mathcal{M}_1(\Br,\Phi_1)$ be such that $\int \freq\,\d\mu_\varepsilon \ge \frac{1}{\inj} - \varepsilon$. By using $\mu_\varepsilon$ as one of the candidates for the computation of the topological pressure 
    $$\mathscr{P}(\Br, \Phi_1, \log(k-1)\freq_{\textup{non-ext}} + \log(k)\freq_\external)$$
    and by recalling Theorems \ref{theo:intro_entropy_formula_branched_coverings} and \ref{theo_intro:freq_does_not_depend} and Corollary \ref{cor:freq_k=freq_ext}, we obtain
    $$\log(k-1)\left(\frac{1}{\inj} - \varepsilon\right)\le h_\textup{top}(\Y,\Sigma,k).$$
    By taking the limit for $\varepsilon$ going to zero we have that 
    $$\frac{\log(k-1)}{\inj}\le h_\textup{top}(\Y,\Sigma,k).$$
    
    On the other hand, for every $\mu \in \mathcal{M}_1(\Br,\Phi_1)$ we have
    \begin{equation*}
        h_\mu + \int \log(k-1)\freq_{\textup{non-ext}} + \log(k)\freq_\external\,\d\mu \le h_\mu + \int \log(k)\freq\,\d\mu \le h_\mu + \frac{\log(k)}{\inj},
    \end{equation*}
    by Proposition \ref{prop:erg_opt_h}.
    The supremum over $\mu \in \mathcal{M}_1(\Br,\Phi_1)$ of the left-hand side equals $h_\textup{top}(\Y,\Sigma,k)$, by Theorems \ref{theo:intro_entropy_formula_branched_coverings} and \ref{theo_intro:freq_does_not_depend}. Therefore, 
    $$h_\textup{top}(\Y,\Sigma,k) \le \sup_{\mu \in \mathcal{M}_1(\Br,\Phi_1)} h_\mu + \frac{\log(k)}{\inj} = h_\textup{top}(\Br,\Phi_1) + \frac{\log(k)}{\inj},$$
    where the last equality is again Proposition \ref{prop:variational_principle}. This gives the first part of the thesis. 
    Since we are assuming $h_\textup{top}(\Br,\Phi_1) < \infty$ we obtain that
    $$\lim_{k\to +\infty} \frac{h_\textup{top}(\Y,\Sigma,k)}{\log(k)} = \frac{1}{\inj}.$$

    Let us suppose now that $\kappa = -1$ and let us consider the space $\X := \vee^2(\Y,\Sigma)$. Propositions \ref{prop:projection_geodesics_to_broken_geodesics} and \ref{prop:Ledrappier_entropy_fibers} imply that 
    $$h_\textup{top}(\Br,\Phi_1) \le h_\textup{top}(\LocGeod(\X), \Phi_1) = h(\tilde\X;\pi_1(\X)).$$
    The last equality follows by \eqref{eq:h_top=critical_exponent}. The proof will be completed by showing that $h(\tilde\X;\pi_1(\X)) \le h(\tilde\Y;\pi_1(\Y)) + \log(2)/\inj$.
    For that, we need a description of the universal cover of $\X$. Let $p\colon \tilde\Y \to \Y$ be the universal cover of $\Y$. Let us write $\Sigma = \bigsqcup_{i=1}^M \Sigma_i$, where each $\Sigma_i$ is a connected component of $\Sigma$. There are finitely many of them since $\sfd_\Y(\Sigma_i,\Sigma_j)\ge \rho_\textup{CAT}(\Y)$. We denote by $\mathcal{C}$ the set of all connected components of $p^{-1}(\Sigma)$. We define a sequence of spaces $\{\tilde \X_i\}$ as follows. $\tilde \X_0 = \tilde \Y$ and we say that this copy of $\tilde{Y}$ is at level $0$. The space $\tilde \X_1$ is obtained in the following way: for every $C\in\mathcal{C}$ we choose a copy $\tilde\Y_C$ of $\tilde{\Y}$ and we glue it to $\tilde{\X}_0$ along $C$. The new copies of $\tilde{\Y}$ are said to be at level $1$. The space $\tilde{\X}_1$ is $\CATminus$ because of \cite[Theorem II.11.3]{BridsonHaefliger}. We now repeat the construction: to every connected component $C$ in one of the level $1$ copies of $\tilde{\Y}$ we glue another copy of $\tilde{\Y}$ by identifying the spaces along the connected component. The new copies of $\tilde{\Y}$ are said of level $2$. Fix a basepoint $x_0 \in \tilde\X_0$ and $R>0$. We notice that if $i$ is big enough then the balls of radius $R$ around $x_0$ in the spaces $\tilde \X_i$ are isometric. Therefore, the sequence $\tilde \X_i$ converges in the pointed Gromov-Hausdorff sense to a limit space $\tilde \X_\infty$, which is $\CATminus$ by \cite[Corollary II.3.10]{BridsonHaefliger}. Moreover, by construction it is locally isometric to $\X$, therefore it is the universal cover of $\X$. By \cite[Proposition 5.7]{Cavallucci23},
    $$h(\tilde{\X}_\infty;\pi_1(\X)) = \lim_{R\to +\infty} \frac{1}{R}\log \textup{Pack}(B(x_0,R),r),$$
    for $r = \inj /4$. The construction gives the explicit bound 
    $$\textup{Pack}(B(x_0,R),r) \le 2^\frac{R}{\inj}\cdot \textup{Pack}(B_{\tilde\Y}(x_0,R+D),r),$$
    where $B_{\tilde{\Y}}(x_0,R+D)$ is the ball in $\tilde \Y$ and $D$ is the diameter of $\Y$. Then
    $$h(\tilde{\X}_\infty;\pi_1(\X)) \le  \lim_{R\to +\infty} \frac{1}{R}\log 2^\frac{R}{\inj}\cdot \textup{Pack}(B_{\tilde\Y}(x_0,R+D),r) = \frac{\log(2)}{\inj} + h(\tilde\Y; \pi_1(\Y)),$$
    where we used again \cite[Proposition 5.7]{Cavallucci23}.
\end{proof}

\begin{remark}
    The authors do not know any example where $h_\textup{top}(\Br,\Phi_1) = \infty$. The bound $h_\textup{top}(\Br,\Phi_1) \le h(\tilde{\Y};\pi_1(\Y)) + \log(2)/\inj$ is still valid, with the same proof, when $\kappa = 0$ and $\Y$ has Gromov hyperbolic fundamental group or when $\kappa = 0$ and $\Y$ is geodesically complete, by Remark \ref{rem:h_top=h_crit}. In general, it is plausible that $h_\textup{top}(\Br,\Phi_1) < \infty$ as soon as $h_\textup{top}(\LocGeod(\Y),\Phi_1) < \infty$.
\end{remark}

As a consequence we prove the Corollary \ref{cor:GT_manifolds_estimate}.

\begin{C1}
    Let $n\in \N$ and $D\ge 0$. Then there exists $E(n,D)\ge 0$ such that for every Gromov-Thurston pair $(\Y,\Sigma)$ with $\textup{dim}(\Y) = n$ and $\textup{Diam}(\Y) \le D$ it holds that
    $$\left\vert h_\textup{top}(\LocGeod(\textup{GT}^k(\Y,\Sigma)), \Phi_1) - \frac{\log(k)}{\inj}\right\vert \le  E(n,D).$$
\end{C1}

\begin{proof}
    By Theorem \ref{theo_intro:asymptotic_formula}, the thesis is true with $E = h(\tilde{\Y};\pi_1(\Y)) + \log(2)/\inj$. Since $\tilde{\Y}$ is isometric to the hyperbolic space $\mathbb{H}^n$ we have that $h(\tilde{\Y};\pi_1(\Y)) = n-1$. Therefore the thesis is true if we find a uniform bound $s_0>0$, depending only on $n$ and $D$, such that $\inj \ge s_0$. Proposition \ref{prop:normal_inj_radius_minimizing_curve} implies that $\inj \ge \rho_\textup{CAT}(\Y)$. The uniform bound depending only on $n$ and $D$ on the $\textup{CAT}$-radius of $\Y$ is provided by Proposition \ref{prop:lower_bound_CAT_radius}.
\end{proof}

\begin{remark}
    Given a compact, locally $\CATminus$ space $(\Y,\sfd_\Y)$ and a closed, locally convex subset $\Sigma$, the proof of Corollary \ref{cor:GT_manifolds_estimate} says that one can bound uniformly in $k$ the quantity 
    $$\left\vert h_\textup{top}(\Y,\Sigma,k), \Phi_1) - \frac{\log(k)}{\inj}\right\vert$$
    by knowing an upper bound on the critical exponent of $\pi_1(\Y)$ and a lower bound on the \textup{CAT}-radius of $\Y$. For instance, a uniform packing property on the universal cover $\tilde{\Y}$ gives a uniform upper bound on the critical exponent of $\pi_1(\Y)$ by \cite[Proposition 5.7 and Lemma 5.6]{Cavallucci23} and a uniform diameter bound gives the estimate on the \textup{CAT}-radius by Proposition \ref{prop:lower_bound_CAT_radius}.
\end{remark}

\section{Dynamical and asymptotic properties of the Bowen-Margulis measures}
\label{sec:asymptotic_BM}

Let $(\Y,\sfd_\Y)$ be a compact, locally $\CATminus$ space and let $\Sigma \subseteq\Y$ be a closed, locally convex subset. For every $k\in \N$, we define $\mm_k := \Pi_\#\mm_{\textup{BM}}^{\X}$, where $\mm_{\textup{BM}}^{\X}$ is the Bowen-Margulis measure of any branched covering $(\X,\Y,\Sigma,\pi)$ of pure-degree $k$ and $\Pi$ is the map of Proposition \ref{prop:projection_geodesics_to_broken_geodesics}. Theorem \ref{theo:wedges_maximal_entropy} implies that $\mm_k$ does not depend on the choice of the branched covering. In this last section we will study the properties of the measures $\mm_k$ associated to the couple $(\Y,\Sigma)$.

\subsection{Dynamical properties}
\label{subsec:dynamical_properties}
Our first finding is that $\mm_k$-a.e. $\eta$ crosses $\Sigma$ with a small concatenation angle infinitely many times. 

\begin{proposition}
\label{prop:BM_supported_infinite_intersections}
    Let $(\Y,\sfd_\Y)$ be a compact, locally $\CATminus$ space and let $\Sigma \subseteq\Y$ be a closed, locally convex subset such that $\Sigma^\circlearrowleft \neq\emptyset$. For every $\vartheta > 0$ consider the set
    $$\textup{NE}_\vartheta:=\{\eta \in \Br\,:\, \#\{t\in\partial\eta^{-1}(\Sigma)\,:\, \angle(\eta,t)<\vartheta\} = \infty\}.$$
    Then $\mm_k(\textup{NE}_\vartheta) = 1$ for every $k \ge 2$. In particular, $\mm_k$-a.e. $\eta \in \Br$ has infinitely many non-extremal intersections.
\end{proposition}
\begin{proof}
    For every $m\in \N$ we consider the set
    $$\textup{NE}_\vartheta(m) :=\{\eta \in \Br\,:\, \#\{t\in\partial\eta^{-1}(\Sigma)\,:\, \angle(\eta,t)<\vartheta\} > m\}.$$
    It is $\Phi_1$-invariant. For every $\sigma \in \Sigma^\circlearrowleft$ realizing the infimum in \eqref{eq:defin_normal_inj_radius} we can consider the $\Sigma$-broken local geodesic $\eta$ defined by the alternated concatenation of $\sigma$ and $-\sigma$, which belongs to $\Br$ as shown in the proof of Proposition \ref{prop:erg_opt_h}. Moreover $\eta \in \textup{NE}_\vartheta(m)$ by construction, for every $\vartheta > 0$. Hence $\textup{NE}_\vartheta(m) \ne \emptyset$. We show that $\textup{NE}_\vartheta(m)$ is open by proving that the set
    $$\textup{NE}_\vartheta(m)^c = \{\eta \in \Br\,:\, \#\{t\in\partial\eta^{-1}(\Sigma)\,:\, \angle(\eta,t)<\vartheta\} \le m\}$$
    is closed. Let $\{\eta_j\} \to \{\eta_\infty\}$ in $\Br$, where each $\eta_j$ has at most $m$ concatenation times with angle smaller than $\vartheta$. Suppose that $\eta_\infty$ has $m+1$ concatenation times, say $t_\infty^1,\ldots,t_\infty^{m+1} \in \R$, with angle smaller than $\vartheta$. Set $z^i := \eta_\infty(t_\infty^i)$. Take  $0<\varepsilon < \textup{min}_{i \ne j} \sfd_\Y(z^i,z^j)$. For every $i$ let us take a $\CATminus$ ball $B(z^i,2r)$ for some small $0<r<\varepsilon/2$. By taking $j$ big enough we can assume that $\eta_j([t_\infty^i - r, t_\infty^i+ r])\subseteq B(z^i,2r)$.     
    If they are local geodesics of $\Y$, hence geodesic segments in $B(z^i,2r)$, then this sequence of geodesic segments converges to $\eta_\infty([t_\infty^i -r, t_\infty^i +r])$, which is therefore a geodesic segment, contradicting the fact that $\eta_\infty$ has a non-extremal intersection at $t_\infty^i$. 
    This means that we can find times $t_j^i$, close to $t_\infty^i$, such that $\eta_j(t^i_j) \in \Sigma$ and the intersection is non-extremal, hence $t^i_j \in \partial\eta_j^{-1}(\Sigma)$. Moreover, by the upper semicontinuity of the angles, see \cite[Proposition II.3.3]{BridsonHaefliger}, $\angle(\eta_j,t^i_j) < \vartheta$ for every $j$ big enough and every $i\in \{1,\ldots,m+1\}$.
    By the choice of $r$ and $\varepsilon$, the points $t_j^i$ are all distinct, contradicting the fact that $\eta_j \in \textup{NE}_\vartheta(m)^c$. Hence $\textup{NE}_\vartheta(m)$ is open. 
    
    Let $(\X,\Y,\Sigma,\pi)$ be any branched covering and let $\Pi$ be the map of Proposition \ref{prop:projection_geodesics_to_broken_geodesics}. Then $\Pi^{-1}(\textup{NE}_\vartheta(m))$ is open in $\LocGeod(\X)$. Proposition \ref{prop:top_ent_geod_flow_BM_measure} implies that $\mm_{\textup{BM}}^k(\Pi^{-1}(\textup{NE}_\vartheta(m)))=1$, i.e. $\mm_k(\textup{NE}_\vartheta(m)) = 1$.    
    Therefore, $\mm_k(\bigcap_{N\in \N} \textup{NE}_\vartheta(m)) = 1$, which is the thesis.
\end{proof}

There are examples where every element of $\Br$ has only non-extremal intersections, e.g. Example (2) in Section \ref{subsec:Examples}. In other cases, there are elements of $\Br$ with extremal intersections. 

Given $\eta \in \Br$ we denote by $\{t_j(\eta)\}$ the elements of the set $\partial\eta^{-1}(\Sigma) \cap [0,+\infty)$ in increasing order. The next result estimates the average number of intersections with $\Sigma$.
\begin{proposition}
    Let $(\Y,\sfd_\Y)$ be a compact, locally $\CATminus$ space, let $\Sigma \subseteq\Y$ be a closed, locally convex subset and let $k\ge 2$. Then for $\mm_k$-a.e. $\eta \in \Br$ it holds that 
    $$\limi_{n\to +\infty}\frac{\#\partial\eta^{-1}(\Sigma)\cap[0,n-1]}{n} \ge \int\freq\,\d\mm_k$$
    and that
    $$\lims_{n\to +\infty} \frac{t_n(\eta)}{n} \le \frac{1}{\int \freq \,\d\mm_k}.$$
\end{proposition}
\begin{proof}
    By definition, $\#\partial\eta^{-1}(\Sigma)\cap[0,n-1] \ge \sum_{i=0}^{n-1} \freq \circ \Phi_i(-\eta)$. Therefore,
    $$\limi_{n\to +\infty}\frac{\#\partial\eta^{-1}(\Sigma)\cap[0,n-1]}{n} \ge \lim_{n\to +\infty}\frac{1}{n}\sum_{i=0}^{n-1} \freq \circ \Phi_i(-\eta) = \int \freq\,\d\mm_k$$
    for $\mm_k$-a.e. $\eta$. The equality follows by Theorem \ref{theo:Birkhoff} and Proposition \ref{prop:top_ent_geod_flow_BM_measure}.
    This implies that 
    $$\lims_{n\to +\infty} \frac{t_n(\eta)}{n} = \lims_{n\to +\infty} \frac{n}{\#\{t_j(\eta)\} \cap [0,n-1]} \le \frac{1}{\int\freq \,\d\mm_k}.$$
\end{proof}

However, the distance between two consecutive intersection times is arbitrarily large for $\mm_k$-a.e. $\eta \in \Br$, under mild assumptions.

\begin{proposition}
    Let $(\Y,\sfd_\Y)$ be a compact, locally $\CATminus$ space, let $\Sigma \subseteq\Y$ be a closed, locally convex subset such that $\Sigma^\circlearrowleft \neq \emptyset$ and let $k\ge 2$. Assume that 
    $\sup\{\ell(\gamma)\,:\, \gamma \in \Sigma^\circlearrowleft\} = \infty.$
    Then 
    $$\lims_{j\to +\infty} t_{j+1}(\eta) - t_j(\eta) = \infty$$
    for $\mm_k$-a.e. $\eta \in \Br$.
\end{proposition}
\begin{proof}
    For every $M \in \N$ consider the set
    $$A_M := \left\{\eta \in \Br\,:\, \sup_{j\ge 0}\vert t_{j+1}(\eta) - t_j(\eta) \vert \le M\right\}.$$
    Each $A_M$ is ${\Phi_1}$-invariant. Indeed, $\{t_j(\Phi_1^{-1}(\eta))\} \subseteq \{t_j(\eta)\}$ for every $\eta \in \Br$. More precisely, there exists $m \in \N$ such that $t_j(\Phi_1^{-1}(\eta)) = t_{j+m}(\eta)$ for every $j$. Therefore, $\Phi_1^{-1}(\eta) \in A_M$ if $\eta \in A_M$, hence $\Phi_1^{-1}(A_M)\subseteq A_M$. If the thesis is not true, then $\mm_k(\bigcup_{M\in \N} A_M) = 1$. Since $\mm_k$ is ergodic,
    there exists $M_0 \in \N$ such that $\mm_k(A_{M_0}) = 1$. Let us fix a local geodesic segment $\sigma \in \Sigma^\circlearrowleft$ of $\Y$ such that $\ell(\sigma) \ge M_0+1$. Consider the $\Sigma$-broken geodesic line $\eta$ obtained by concatenating $\sigma$ with $-\sigma$ infinitely many times. Then $\eta \in \Br$ because of Theorem \ref{theo:characterization_broken_geodesics}. We claim that there exists a neighbourhood of $\eta$ such that every of its elements does not belong to $A_{M_0}$. If not, we could find a sequence $\eta_i$ of elements of $\Br$ converging to $\eta$ and such that $t_1(\eta_i) - t_0(\eta_i) \le M_0$. Moreover, the sequence of times $\{t_0(\eta_i)\}, \{t_1(\eta_i)\}$ are bounded. We may also suppose that $t_0(\eta_i)$ is a non-branching intersection, so that $t_1(\eta_i) - t_0(\eta_i) \ge \inj$. Up to a subsequence, we can suppose they converge to $t_0^*, t_1^*$, respectively. Hence, $\inj \le t_1^* - t_0^* \le M_0$. The points $\eta(t_0^*), \eta(t_1^*)$ belong to $\Sigma$. The points $t_0^*,t_1^*$ may not belong to $\partial \eta^{-1}(\Sigma)$. However, for every $i$ there exists $s_i \in [t_0(\eta_j) + \rho_\textup{CAT}(\Y)/2, t_1(\eta_j) - \rho_\textup{CAT}(\Y)/2]$ such that $\sfd_\Y(\eta_i(s_i), \Sigma) \ge \rho_\textup{CAT}(\Y)/2$, by Lemma \ref{lemma:uniform_distance_geodesics_to_singular_set}. This gives the existence of times $t_0^* \le s_0^* <s_1^* \le t_1^*$ with $s_0^*,s_1^* \in \partial \eta^{-1}(\Sigma)$. The distance between any two consecutive times in $\partial \eta^{-1}(\Sigma)$ is at least $M_0+1$, giving a contradiction. Then, we can find an open set  $U\subseteq \Br \setminus A_{M_0}$. By Proposition \ref{prop:top_ent_geod_flow_BM_measure}, $\mm_k(U) > 0$, which is impossible since we are assuming that $\mm_k(A_{M_0}) = 1$. This concludes the proof.
\end{proof}

We remark that the condition $\sup\{\ell(\gamma)\,:\, \gamma \in \Sigma^\circlearrowleft\} = \infty$ is also necessary for the validity of the thesis.

\subsection{Asymptotic properties}
\label{subsec:asymptotic_BM}
We discuss the asymptotic properties of the sequence of measures $\{\mm_k\}$. The first one is a consequence of Proposition \ref{prop:erg_opt_h}.
\begin{proposition}
\label{prop:asymptotic_integral_freq_mk}
    Let $(\Y,\sfd_\Y)$ be a compact, locally $\CATminus$ space and let $\Sigma \subseteq\Y$ be a closed, locally convex subset. Then
    $$\lim_{k\to +\infty} \int \freq\,\d\mm_k = \frac{1}{\inj}.$$
\end{proposition}
\begin{proof}
    If $\Sigma^\circlearrowleft = \emptyset$ then $\inj = \infty$ and $\int \freq \,\d\mm_k = 0$, by Proposition \ref{prop:erg_opt_h}. Otherwise, again by Proposition \ref{prop:erg_opt_h}, for every $\varepsilon > 0$ we can find measures $\mu_\varepsilon \in \mathcal{M}_1(\Br,\Phi_1)$ such that $\int \freq \,\d\mu_\varepsilon > \frac{1}{\inj}-\varepsilon$. 
    Corollary \ref{cor:m_BM_unique_measure_max_pressure} gives that
    $$h_{\mu_\varepsilon} + \log(k-1)\left(\frac{1}{\inj} - \varepsilon\right) \le h_{\mm_k} + \log(k)\int \freq\,\d\mm_k.$$
    By dividing by $\log(k)$ and taking the limit for $k$ going to infinity we obtain that $\int \freq \,\d\mm_k \ge \frac{1}{\inj}-\varepsilon$. By the arbitrariness of $\varepsilon>0$ we deduce that 
    $$\limi_{k\to +\infty}\int \freq \,\d\mm_k \ge \frac{1}{\inj}.$$
    On the other hand, $\lims_{k\to +\infty}\int \freq \,\d\mm_k \le \frac{1}{\inj}$ by Proposition \ref{prop:erg_opt_h}.
\end{proof}

The remainder of the paper is devoted to the description of the weak limit points of the sequence $\{\mm_k\}$. From now on, we denote by $\mm_\infty$ any measure on $\Br$ arising as weak limit of a subsequence of $\{\mm_k\}$. Notice that $\mm_\infty \in \mathcal{M}_1(\Br,\Phi_1)$. Actually, $\mm_\infty$ is $\Phi_t$-invariant for every $t\in \R$ because of Proposition \ref{prop:top_ent_geod_flow_BM_measure}. 

We introduce the following natural definitions.
\begin{definition}
Let $\Y$ be a compact, locally $\textup{CAT}(\kappa)$ space and let $\Sigma \subseteq \Y$ be closed and locally convex. We define
\begin{equation}
    \begin{aligned}
        &\next\colon \textup{Broken}(\Y,\Sigma) \to [0,+\infty], \quad \eta \mapsto \begin{cases}\sup \{t\ge 0\,:\, [0,t]\subseteq \eta^{-1}(\Y\setminus\Sigma)\}, &\textup{if } \eta(0)\notin \Sigma,\\
        \sup \{t\ge 0\,:\, [0,t]\subseteq \eta^{-1}(\Sigma)\} &\textup{if } \eta(0)\in \Sigma;
        \end{cases}\\
        &\prev\colon \textup{Broken}(\Y,\Sigma) \to [0,+\infty], \quad \eta \mapsto \next(-\eta),\\
        &\ell_{\internal}\colon \textup{Broken}(\Y,\Sigma) \to [0,+\infty],\quad \eta \mapsto \begin{cases}
            \next(\eta) + \prev(\eta) &\textup{if } \eta(0)\in \Sigma, \\
            0 &\textup{if } \eta(0)\notin \Sigma;
        \end{cases}\\
        &\ell_{\external}\colon \textup{Broken}(\Y,\Sigma) \to [0,+\infty],\quad \eta \mapsto \begin{cases}
            \next(\eta) + \prev(\eta) &\textup{if } \eta(0)\notin \Sigma, \\
            \inj &\textup{if } \eta(0)\in \Sigma.
        \end{cases}
    \end{aligned}
\end{equation}
\end{definition}

In other words, $\ell_\internal(\eta)$ is the length of the maximal internal local geodesic containing $\eta(0)$ when $\eta(0) \in \Sigma$, while 
$\ell_\external(\eta)$ is the length of the maximal external local geodesic containing $\eta(0)$ when $\eta(0) \notin \Sigma$.

\begin{lemma}
\label{lemma:next_prev_Borel}
    The functions $\next$ and $\prev$ are Borel.
    The function $\ell_\external$ is lower semicontinuous 
    and satisfies $\ell_\external \ge \inj$. The function $\ell_\internal$ is upper semicontinuous.
\end{lemma}
\begin{proof}
    The definitions give that $\next$ and $\prev$ are Borel and that $\ell_\external \ge \inj$. Let $\{\eta_j\}\subseteq \textup{Broken}(\Y,\Sigma)$ be a sequence converging to $\eta_\infty$. If $\eta_\infty(0) \in \Sigma$ then $\ell_\external(\eta_\infty) = \inj$ and there is nothing to prove. Otherwise, $\eta_\infty(0) \notin \Sigma$. By definition, $0\in (-\prev(\eta_\infty),\next(\eta_\infty)) \subseteq \eta_\infty^{-1}(\Y\setminus \Sigma)$. Then for every $\varepsilon > 0$, $[-\prev(\eta_\infty)+\varepsilon,\next(\eta_\infty)-\varepsilon] \subseteq \eta_j^{-1}(\Y\setminus\Sigma)$ for every $j$ big enough. This implies that $\ell_\external(\eta_j) \ge \ell_\external(\eta_\infty)-2\varepsilon$. By the arbitrariness of $\varepsilon$ we deduce that $\ell_\external(\eta_\infty) \le \limi_{j\to +\infty} \ell_\external(\eta_j)$. The upper semicontinuity of $\ell_\internal$ can be proved in a similar way.
\end{proof}

For every $\varepsilon > 0$ we define the sets
\begin{equation*}
    \begin{aligned}
        U_\varepsilon &:= \{\eta \in \Br\,:\, \ell_\external(\eta) > (1+\varepsilon)\inj\},\\
        V_\varepsilon &:= \{\eta \in \Br\,:\, \ell_\internal(\eta) > \varepsilon\},
    \end{aligned}
\end{equation*}
Notice that $U_\varepsilon$ is open in $\Br$ since $\ell_\external$ is lower semicontinuous, by Lemma \ref{lemma:next_prev_Borel}.
The next one is the main technical result of this section.

\begin{proposition}
\label{prop:limit_volume_U_epsilon}
    Let $(\Y,\sfd_\Y)$ be a compact, locally $\CATminus$ space and let $\Sigma \subseteq\Y$ be a closed, locally convex subset such that $\Sigma^\circlearrowleft \neq \emptyset$. Then $\lim_{k\to +\infty} \mm_k(U_\varepsilon) = \lim_{k\to +\infty} \mm_k(V_\varepsilon) = 0$ for every $\varepsilon > 0$.
\end{proposition}
\begin{proof}
    We fix $\varepsilon > 0$. We prove the following claim: for every $c >0$ there exists $\bar{k}\in \N$ such that $\mm_k(U_\varepsilon) < c$ for every $k\ge \bar{k}$. This will prove that $\lims_{k\to +\infty} \mm_k(U_\varepsilon) = 0$, hence that $\lim_{k\to +\infty} \mm_k(U_\varepsilon) = 0$. The proof for $V_\varepsilon$ is similar and will be omitted.
    
    Let us suppose by contradiction that $\mm_k(U_\varepsilon) \ge c$ for every $k$ big enough. By Theorem \ref{theo:Birkhoff} applied to $\mathbbm{1}_{U_\varepsilon}$ we get that
    $$\lim_{n\to +\infty} \frac{1}{n}\sum_{i=0}^{n-1} \mathbbm{1}_{U_\varepsilon}(\Phi_{-i}(\eta)) = \int \mathbbm{1}_{U_\varepsilon}\,\d\mm_k = \mm_k(U_\varepsilon) \ge c$$
    for $\mm_k$-a.e. $\eta \in \Br$ and every $k$ big enough. Let $\delta>0$ be very small, whose exact value will be found later. For $k$ big enough we have that
    $$\lim_{n\to +\infty} \frac{1}{n}\sum_{i=0}^{n-1} \freq(\Phi_{-i}(\eta)) = \lim_{n\to +\infty} \frac{\#\partial\eta^{-1}(\Sigma)\cap (-n,0]}{n}= \int\freq \,\d\mm_k >\frac{1}{(1+\delta)\inj},$$
    because of Proposition \ref{prop:asymptotic_integral_freq_mk}. Let $\eta \in \Br$ satisfying the two conditions above. Let also $n$ be so large that 
    \begin{equation}
        \label{eq:conditions_large_n}
        \frac{1}{n}\sum_{i=0}^{n-1} \mathbbm{1}_{U_\varepsilon}(\Phi_i(\eta)) \ge c/2,\qquad  \#\partial\eta^{-1}(\Sigma)\cap (-n,0] > \frac{n}{(1+2\delta)\inj}.
    \end{equation}
    We say that two times $t\le t' \in (-n,0]$ are in the same external component if $\eta\restr{[t,t']} \subseteq \Y\setminus\Sigma$. We partition the set of external components of $\eta$ into two sets: $S_{\mathrm{long}}$ are the external components with length $> (1+\varepsilon)\inj$ and $S_{\mathrm{short}}$ are the remaining ones. Notice that $\inj \le \ell(I) \le (1+\varepsilon)\inj$ for every $I\in S_\text{short}$.     
    The total time spent in the external components is at most $n$, giving:
    $$n \ge \sum_{I \in S_{\mathrm{long}}} \vert{}I\vert{} + \sum_{I \in S_{\mathrm{short}}} \vert{}I\vert{} \ge \sum_{I \in S_{\mathrm{long}}} \vert{}I\vert{} + |S_\textup{short}|\cdot\inj.$$
    Let us denote by $\{s_j\}$ the ordered elements of the set $\partial \eta^{-1}(\Sigma) \cap (-n,0]$, so that 
    $$\#\{s_j\} \le |S_\text{long}| + |S_\text{short}| + 2.$$
    Putting together these two facts we obtain that
    \begin{equation}
        \label{eq:estimate_on_s_j}
        \#\{s_j\} \le |S_\text{long}| + \frac{n - \sum_{I \in S_{\mathrm{long}}} \vert{}I\vert{}}{\inj} + 2 = \frac{n}{\inj} - \sum_{I \in S_{\mathrm{long}}} \left( \frac{\vert{}I\vert{}}{\mathrm{inj}^\perp(\Sigma)} - 1 \right) + 2.
    \end{equation}
    The integers $i\in (-n,0]$ for which $\Phi^i(\eta) \in U_\varepsilon$ are precisely those contained within one of the external components in $S_{\mathrm{long}}$, and by \eqref{eq:conditions_large_n} they are at least $cn/2$. Since any interval $I$ contains at most $\vert{}I\vert{} + 1$ integers we get
    $$\frac{cn}{2} \le \sum_{I \in S_{\mathrm{long}}} (\vert{}I\vert{} + 1).$$
    For every $I \in S_{\mathrm{long}}$, we have $\vert{}I\vert{} > (1+\varepsilon)\inj$, which implies $\vert{}I\vert{} + 1 < \alpha \vert{}I\vert{}$ for $\alpha = 1 + \frac{1}{(1+\varepsilon)\inj} > 1$. 
    Thus,
    $$\sum_{I \in S_{\mathrm{long}}} \vert{}I\vert{} \ge \frac{cn}{2\alpha}.$$
    Since $\vert{}I\vert{} > (1+\varepsilon)\inj$, we get $\frac{\vert{}I\vert{}}{\inj} - 1 > \frac{\varepsilon}{1+\varepsilon} \frac{\vert{}I\vert{}}{\inj}$. 
    Therefore,
    $$\sum_{I \in S_{\mathrm{long}}} \left( \frac{\vert{}I\vert{}}{\inj} - 1 \right) > \frac{\varepsilon}{(1+\varepsilon)\inj} \sum_{I \in S_{\mathrm{long}}} \vert{}I\vert{} \ge \beta n,$$
    where $\beta := \frac{\varepsilon}{(1+\varepsilon)\inj}\cdot\frac{c}{2\alpha} > 0$, which does not depend on $n$. Combining this with \eqref{eq:estimate_on_s_j} we obtain that
    $$\#\{s_j\} \le \frac{n}{\inj} - \beta n + 2.$$
    Dividing by $n$ and taking the limit as $n \to \infty$, we get
    $$\limsup_{n \to \infty} \frac{\#\{s_j\}}{n} \le \frac{1}{\mathrm{inj}^\perp(\Sigma)} - \beta.$$
    On the other hand, by \eqref{eq:conditions_large_n} we have that 
    \begin{equation}
    \label{eq:lower_bound_on_s_j}
        \frac{\#\{s_j\}}{n} > \frac{1}{(1+2\delta)\inj}.
    \end{equation}
     Choosing $\delta$ small enough so that $\frac{1}{(1+2\delta)\mathrm{inj}^\perp(\Sigma)} > \frac{1}{\mathrm{inj}^\perp(\Sigma)} - \beta$ we find the desired contradiction.
\end{proof}

We now show Theorem \ref{theo_intro:asymptotic_BM_measures}, i.e. that $\mm_\infty$ is concentrated on $$\textup{Broken}_\textup{bc}^\perp(\Y,\Sigma) := \{ \eta \in \Br\,:\, \ell_\external(\Phi_t(\eta)) = \inj,\, \ell_\internal(\Phi_t(\eta))=0\, \forall t \in \R\}.$$
This set corresponds to the set of elements of $\Br$ in which there are no internal segments and every external segment has the minimal possible length.

\begin{T4}
    Let $(\Y,\sfd_\Y)$ be a compact, locally $\textup{CAT}(-1)$ space and let $\Sigma \subseteq \Y$ be a closed, locally convex subset such that $\Sigma^\circlearrowleft \neq \emptyset$. Then $\mm_\infty$ is supported on $\textup{Broken}_{\textup{bc}}^\perp(\Y,\Sigma)$. 
\end{T4}

\begin{proof}
    We divide the proof in steps.
    
    \textbf{Step 1:} The set $C_\infty := \{\eta \in \Br\,:\, \ell_\external(\Phi_t(\eta)) = \inj\, \textup{ for all } t \in \R\}$
    satisfies $\mm_\infty(C_\infty) = 1$.
    
    Let $m\in \N$. The set $U_{1/m}$ is open. Since $\mm_\infty$ is the weak limit of a subsequence of $\{\mm_k\}$ we have that
    $$\mm_\infty(U_{1/m}) \le \limi_{k\to +\infty} \mm_k(U_{1/m}) = 0$$
    by Proposition \ref{prop:limit_volume_U_epsilon}. Therefore, $\mm_\infty(\bigcup_{m\in \N} U_{1/m}) = 0$. This means that $\ell_\external(\eta) = \inj$ for $\mm_\infty$-a.e. $\eta \in \Br$, since $\ell_\external \ge \inj$. Let us denote
    $$C := \{ \eta \in \Br\,:\, \ell_\external(\eta) = \inj\},$$
    that satisfies $\mm_\infty(C)=1$ as we have just shown. By $\Phi_t$-invariance of $\mm_\infty$ for every $t\in \R$, we get that 
    $$\mm_\infty\left(\bigcap_{m\in \Z}\Phi_{m\cdot\inj /2}(C)\right) = 1.$$
    The proof of the claim is concluded if we show that
    $$\bigcap_{m\in \Z}\Phi_{m\cdot\inj /2}(C) = C_\infty.$$
    Clearly $C_\infty \subseteq \bigcap_{m\in \Z}\Phi_{m\cdot\inj /2}(C)$. On the other hand let $\eta \in \bigcap_{m\in \Z}\Phi_{m\cdot\inj /2}(C)$ and fix $t\in \R$. Let $m\in \Z$ be such that $\vert t-m\cdot\inj/2\vert \le \inj/2$. If $\eta(t) \in \Sigma$ then $\ell_\external(\Phi_t(\eta)) = \inj$ by definition. If $\eta(t)\notin \Sigma$ then $\eta(t)$ belongs to an external segment of length at least $\inj$. This means that $\eta(m'\cdot\inj/2)$ belongs to the same external segment of $\eta(t)$ for some $m'\in \{m-1,m,m+1\}$. In particular, $\ell_\external(\Phi_t(\eta)) = \ell_\external(\Phi_{m'\cdot\inj/2}(\eta)) = \inj$ by assumption. This proves that $\bigcap_{m\in \Z}\Phi_{m\cdot\inj /2}(C) \subseteq C_\infty$.
    
    \textbf{Step 2:} $C_\infty$ is closed and $\ell_\internal$ is continuous on $C_\infty$.

    The set $C$ is closed by lower semicontinuity of $\ell_\external$, see Lemma \ref{lemma:next_prev_Borel}. The set $C_\infty$ is the intersection of images of $C$ via homeomorphisms, hence it is closed. Now, every external local geodesic segment of every $\eta\in C_\infty$ realizes the infimum defining $\inj$. By Proposition \ref{prop:normal_inj_radius_minimizing_curve}, each of the external segments of $\eta$ meets $\Sigma$ with an angle that is at least $\pi/2$. Let now $\{\eta_j\}\subseteq C_\infty$ be a sequence converging to $\eta_\infty$. Lemma \ref{lemma:next_prev_Borel} gives that $\ell_\internal(\eta_\infty) \ge \lims_{j\to+\infty} \ell_\internal(\eta_j)$.    
    If $\eta_\infty(0) \notin \Sigma$ then $\eta_j(0) \notin \Sigma$ for $j$ big enough and so $\ell_\internal(\eta_\infty) = 0 = \lim_{j\to +\infty} \ell_\internal(\eta_j)$ by definition. Otherwise, let $0\in [a_\infty,b_\infty] \subseteq \R$ be such that $\eta_\infty \restr{[a_\infty,b_\infty]}$ is the maximal internal geodesic containing $\eta_\infty(0)$. Let us suppose that $b_\infty - a_\infty = \ell_\internal(\eta_\infty) > \lims_{j\to+\infty} \ell_\internal(\eta_j)$. Then we can find $\varepsilon > 0$ and a sequence $t_j \in [a_\infty + \varepsilon, b_\infty - \varepsilon]$ such that $t_j \in \partial\eta_j^{-1}(\Sigma)$. We can choose $t_j$ to be the closest to $0$. We suppose that $t_j \ge 0$, the other case being analogous. By minimality of $t_j$, the segment $\eta_j\restr{[0,t_j]}$ is internal and therefore $\eta_j\restr{[t_j,t_j+\inj]}$ is external.  Proposition \ref{prop:normal_inj_radius_minimizing_curve} implies that $\sfd_\Y(\eta_j(t_j+s),\Sigma) = s$ for every $s<\rho_\textup{CAT}(\Y)$. Up to a subsequence we may suppose that $t_j$ converges to $t_\infty \in [a_\infty-\varepsilon,a_\infty+\varepsilon]$. This implies that $\sfd_\Y(\eta_\infty(t_\infty + s), \Sigma) = s$ for $s < \rho_\textup{CAT}(\Y)$. For $s<\varepsilon$ we get a contradiction because $\eta_\infty(t_\infty +s) \in \Sigma$. Therefore, $\ell_\internal(\eta_\infty) \le \lims_{j\to+\infty} \ell_\internal(\eta_j)$, concluding the proof of this step.

    \textbf{Step 3:} $\mm_\infty(\textup{Broken}_\textup{bc}^\perp(\Y,\Sigma)) = 1$.
    
    Now, for every $\varepsilon > 0$ we have that the set $V_\varepsilon \cap C_\infty$ is open in $C_\infty$ because $\ell_\internal$ is continuous on $C_\infty$.  Arguing as in Step 1 we deduce that $\mm_\infty(\textup{Broken}_\textup{bc}^\perp(\Y,\Sigma)) = 1$.
\end{proof}

The dynamical system $(\textup{Broken}_\textup{bc}^\perp(\Y,\Sigma), \Phi_1)$, up to reparametrization, has a simple description. Let $\mathcal{A}:=\{\sigma \in \Sigma^\circlearrowleft\,:\, \ell(\sigma) = \inj\}$ and for $\sigma,\sigma'\in \mathcal{A}$ define $\mathfrak{S}(\sigma,\sigma') = 1$ if and only if $\omega(\sigma) = \alpha(\sigma')$. The associated subshift of finite type is the dynamical system $(\textup{\texttt{SubShift}}(\mathcal{A},\mathfrak{S}), \textup{\texttt{shift}})$ where
$$\textup{\texttt{SubShift}}(\mathcal{A},\mathfrak{S}) := \{w = (\sigma_i)\in \mathcal{A}^\Z\,:\, \mathfrak{S}(\sigma_i,\sigma_{i+1}) = 1 \text{ for every } i \in \Z\}$$
and $\texttt{shift}((\sigma_i)_{i\in \Z}) = (\sigma_{i+1})_{i\in \Z}$.  It is equipped with the topology induced by the product topology on $\mathcal{A}^\Z$, which is metrizable via the metric $\sfd((\sigma_i),(\sigma_i')) = 2^{-\inf\{\vert i \vert \,:\, i\in \Z,\, \sigma_i\neq \sigma_i'\}}.$

\begin{proposition}
\label{prop:Broken_perp_subshift}
    Let $(\Y,\sfd_\Y)$ be a compact, locally $\CATminus$ space and let $\Sigma \subseteq\Y$ be a closed, locally convex subset such that $\Sigma^\circlearrowleft \neq \emptyset$. Then $\mathcal{A}$ is finite and the map
    $$\Psi\colon \textup{\texttt{SubShift}}(\mathcal{A},\mathfrak{S}) \times [0,\inj] \to \textup{Broken}_\textup{bc}^\perp(\Y,\Sigma), \quad (w=(\sigma_i)_{i\in \Z},s) \mapsto \Phi_s\left(\underset{i\in\Z}{\star}\sigma_i\right),$$
    is continuous, surjective and satisfies $\Psi\circ (\textup{\texttt{shift}}, \textup{id}) = \Phi_{\inj}\circ\Psi$.
    It induces  a homeomorphism $$\Psi\colon\textup{\texttt{SubShift}}(\mathcal{A},\mathfrak{S}) \times [0,\inj] /_\sim \to \textup{Broken}_\textup{bc}^\perp(\Y,\Sigma),$$ 
    where $(w,\inj) \sim (\textup{\texttt{shift}}(w),0)$ for every $w \in \textup{\texttt{SubShift}}(\mathcal{A},\mathfrak{S})$.
    Moreover,
    $$h_\textup{top}(\textup{Broken}_\textup{bc}^{\perp}(\Y,\Sigma), \Phi_1) \le \limi_{k\to +\infty} h_{\mm_k}(\Br,\Phi_1).$$
\end{proposition}

\begin{proof}
    Arguing as in the proof of Proposition \ref{prop:normal_inj_radius_minimizing_curve} we get that $\mathcal{A}$ is a non-empty, closed subset of the set of geodesic segments of $\Y$, hence compact. We claim that $\mathcal{A}$ is discrete, hence finite. Let $\tilde{\Y}$ be the universal cover of $\Y$ and let $\tilde{\Sigma}$ be a lift of $\Sigma$. Every element of $\mathcal{A}$ lifts to a geodesic segment on $\tilde{\Y}$ which realizes the distance between $\tilde{\Sigma}$ and another lift $\tilde{\Sigma}'$ of $\Sigma$. For every fixed $\tilde{\Sigma}'$ there is exactly one such curve because of the $\CATminus$ condition: if there were two then we could find a geodesic quadrilateral with all angles at least $\pi/2$, hence flat by \cite[II.2.11]{BridsonHaefliger}. But a $\CATminus$ space does not contain flat rectangles. Let $\sigma \in \mathcal{A}$ and $\tilde{\sigma}$ be a lift realizing the distance between $\tilde{\Sigma}$ and $\tilde{\Sigma}'$. Let $\{\sigma_j\}\subseteq \mathcal{A}$ be a sequence converging to $\sigma$. For every $j$ let $\tilde{\sigma}_j$ be a lift of $\sigma_j$ joining $\tilde{\Sigma}$ with another lift $\tilde{\Sigma}_j$. Notice that $\tilde{\sigma}_j(0)$ converges to $\tilde{\sigma}(0)$. This means that for $j$ big enough each $\tilde{\Sigma}_j$ intersects $B(\tilde{\sigma}(0),\inj + 1)$. Then the set of $\tilde{\Sigma}_j$'s is finite, hence constantly equal to $\tilde{\Sigma}'$. By the uniqueness of the minimal geodesic joining $\tilde{\Sigma}$ to $\tilde{\Sigma}'$ we have that $\tilde{\sigma}_j = \tilde{\sigma}$ for every $j$ big enough, hence $\sigma_j = \sigma$.    

    We now study the map $\Psi$. First of all, it is well-defined. Indeed, given $w=(\sigma_i)_{i\in \Z} \in \textup{\texttt{SubShift}}(\mathcal{A},\mathfrak{S})$, the curve $\eta = \underset{i\in\Z}{\star} \sigma_i$ is a $\Sigma$-broken local geodesic line, due to the fact that $\omega(\sigma_i) = \alpha(\sigma_{i+1})$ for every $i\in \Z$. Moreover it satisfies $\ell_\external(\Phi_t(\eta)) = \inj$ and $\ell_\internal(\Phi_t(\eta)) = 0$ for every $t\in \R$, by construction. In order to check that $\eta \in \Br$, the condition of Theorem \ref{theo:characterization_broken_geodesics} must be satisfied. Now, if $\sigma_i \star \sigma_{i+1}$ is a local geodesic in $\Y$ then condition (i) of Theorem \ref{theo:characterization_broken_geodesics} is satisfied. Otherwise we call $t$ the concatenation time and we consider the points $\eta(t-s), \eta(t+s)$ for $s < \rho_\textup{CAT}(\Y)$. By Proposition \ref{prop:normal_inj_radius_minimizing_curve}, $\eta(t)$ is the projection of both $\eta(t-s)$ and $\eta(t+s)$ on $\Sigma \cap B(\eta(t),r)$. Hence, for every other point $z\in \Sigma$, 
    $$\sfd(\eta(t-s),z) + \sfd(\eta(t+s),z) \ge \sfd(\eta(t-s),\eta(t)) + \sfd(\eta(t+s),\eta(t)).$$
    This means that $\eta(t) = \textup{\texttt{mid}}_\Sigma(\eta(t-s),\eta(t+s))$, which is condition (ii) of Theorem \ref{theo:characterization_broken_geodesics}. Hence $\eta$ is indeed an element of $\Br$, so of $\textup{Broken}_\textup{bc}^\perp(\Y,\Sigma)$. Moreover, $\Psi$ is trivially surjective and continuous. 
    For every $((\sigma_i)_{i\in \Z}, s) \in \textup{\texttt{SubShift}}(\mathcal{A},\mathfrak{S}) \times [0,\inj]$ we have that
    \begin{equation*}
        \begin{aligned}
            \Psi\circ (\textup{\texttt{shift}}, \textup{id})(((\sigma_i)_{i\in \Z}, s)) = \Phi_s\left(\underset{i\in \Z}{\star} \sigma_{i+1}\right) &= \Phi_s\left(\Phi_{\inj}\left(\underset{i\in \Z}{\star} \sigma_{i}\right)\right) \\
            &= \Phi_{\inj}\left(\Phi_{s}\left(\underset{i\in \Z}{\star} \sigma_{i}\right)\right)\\
            &= \Phi_{\inj}\circ \Psi(((\sigma_i)_{i\in \Z}, s)).
        \end{aligned}
    \end{equation*}
    This property implies that $\Psi$ induces a map on $\textup{\texttt{SubShift}}(\mathcal{A},\mathfrak{S}) \times [0,\inj] /_\sim$, which is clearly injective. Being a continuous, bijective map from a compact space to an Hausdorff space, it is a homeomorphism.

    Finally, let $\mu \in \mathcal{E}_1(\textup{Broken}_\textup{bc}^\perp(\Y,\Sigma), \Phi_1)$ and observe that
    $$\int\freq\,\d\mu = \lim_{n\to +\infty} \frac{1}{n} \sum_{i=0}^{n-1} \freq(\Phi_{-i}(\eta)) = \lim_{n\to +\infty} \frac{\#\partial\eta^{-1}(\Sigma)\cap (-n,0]}{n} = \frac{1}{\inj}$$
    for $\mu$-a.e. $\eta$, by Theorem \ref{theo:Birkhoff}. 
    Corollary \ref{cor:m_BM_unique_measure_max_pressure} implies that
    $$h_\mu + \log(k-1)\int \freq_{\textup{non-ext}}\,\d\mu + \log(k)\int \freq_{\external}\,\d\mu$$
    is less than or equal to
    $$h_{\mm_k} + \log(k-1)\int \freq_{\textup{non-ext}} \,\d\mm_k + \log(k)\int \freq_{\external} \,\d\mm_k.$$
    for every $k\ge 2$. Rearranging the terms and using that $\int \freq \,\d\mu = 1/\inj$ and that $\int \freq \,\d\mm_k \le 1/\inj$ we have that
    \begin{equation}
        \label{eq:h_mu_broken_perp}
        h_\mu \le h_{\mm_k} + \log\left(\frac{k-1}{k}\right)\left(\int \freq_{\textup{non-ext}}\,\d\mm_k - \int \freq_{\textup{non-ext}}\,\d\mu\right).
    \end{equation}
    Taking the limit of \eqref{eq:h_mu_broken_perp} for $k$ going to infinity and using that the absolute value of the difference of the integrals is at most $1/\inj$ for every $k$, we get that $h_\mu \le \limi_{k\to + \infty} h_{\mm_k}$. Clearly, $h_\mu(\textup{Broken}_\textup{bc}^\perp(\Y,\Sigma),\Phi_1) = h_\mu(\Br,\Phi_1)$. By Proposition \ref{prop:variational_principle}, we get that 
    $$h_\textup{top}(\textup{Broken}_\textup{bc}^\perp(\Y,\Sigma), \Phi_1) \le \limi_{k\to + \infty} h_{\mm_k}.$$
\end{proof}

The dynamical system $(\textup{Broken}_\textup{bc}^{\perp}(\Y,\Sigma), \Phi_1)$ may have or not a unique measure of maximal entropy, depending on the geometry of the couple $(\Y,\Sigma)$. In Example (1) of Section \ref{subsec:Examples} there is a unique measure of maximal entropy, while in Example (3) of Section \ref{subsec:Examples} there are more measures of maximal entropy.

\begin{remark}
    We do not know if
    $$h_{\mm_\infty}(\Br,\Phi_1) = \limi_{k\to +\infty} h_{\mm_k}(\Br,\Phi_1).$$
    If it were the case, then $\mm_\infty$ would be a measure of maximal entropy for the dynamical system $(\textup{Broken}_\textup{bc}^{\perp}(\Y,\Sigma), \Phi_1)$. Typical assumptions that guarantee such an (upper) semicontinuity of the entropy functional are expansiveness or at least entropy expansiveness. These properties do not hold in general, but suitable uniform versions associated to the sequence $\{\mm_k\}$ may be true. This would not give that the sequence $\{\mm_k\}$ has a unique limit point, a fact that we expect to be true anyway. 
\end{remark}

As we have seen, the dynamical system $(\Br,\Phi_1)$ is an important object attached to the couple $(\Y,\Sigma)$, and it allows us to define many invariants associated to the couple. However, its full dynamical properties are not completely understood, even in the locally $\CATminus$ case. It is straightforward to check that it satisfies the weak specification property and the weak periodic orbit closing properties as defined in \cite{ConstantineLafontThompson2020}. Indeed, these properties come directly from the corresponding properties holding on $(\LocGeod(\X), \Phi_1)$ for every branched covering $(\X,\Y,\Sigma,\pi)$, see \cite[Theorem A]{ConstantineLafontThompson2020}. On the other side, it is not known if it has a unique measure of maximal entropy. In order to apply \cite[Theorem 5.1]{ConstantineLafontThompson2020} one should prove that $(\Br,\Phi_1)$ is expansive, a property that does not seem to hold in general. A weaker notion of expansiveness is introduced in \cite{climenhagaThompson2016}, where it is shown that it is sufficient to obtain the uniqueness of the measure of maximal entropy, see \cite[Theorem A]{climenhagaThompson2016}. However, also this property seems difficult to be checked in general.

\begin{remark}
    All the statements of this section, except Proposition \ref{prop:Broken_perp_subshift}, are valid in the locally $\textup{CAT}(\kappa)$ setting, with arbitrary $\kappa$. It is enough to replace $\mm_k$ with the projection of a measure of maximal entropy of $(\LocGeod(\X),\Phi)$, that may be not unique, and to assume that $h_\textup{top}(\Br,\Phi_1) < \infty$.
\end{remark}

\bibliography{mybibmini}
\bibliographystyle{alpha}

\end{document}

%% file: Model_space.pdf_tex
\begingroup%
  \makeatletter%
  \providecommand\color[2][]{%
    \errmessage{(Inkscape) Color is used for the text in Inkscape, but the package 'color.sty' is not loaded}%
    \renewcommand\color[2][]{}%
  }%
  \providecommand\transparent[1]{%
    \errmessage{(Inkscape) Transparency is used (non-zero) for the text in Inkscape, but the package 'transparent.sty' is not loaded}%
    \renewcommand\transparent[1]{}%
  }%
  \providecommand\rotatebox[2]{#2}%
  \newcommand*\fsize{\dimexpr\f@size pt\relax}%
  \newcommand*\lineheight[1]{\fontsize{\fsize}{#1\fsize}\selectfont}%
  \ifx\svgwidth\undefined%
    \setlength{\unitlength}{444.89357606bp}%
    \ifx\svgscale\undefined%
      \relax%
    \else%
      \setlength{\unitlength}{\unitlength * \real{\svgscale}}%
    \fi%
  \else%
    \setlength{\unitlength}{\svgwidth}%
  \fi%
  \global\let\svgwidth\undefined%
  \global\let\svgscale\undefined%
  \makeatother%
  \begin{picture}(1,0.60700662)%
    \lineheight{1}%
    \setlength\tabcolsep{0pt}%
    \put(0,0){\includegraphics[width=\unitlength,page=1]{Model_space.pdf}}%
    \put(0.192286,0.15575734){\color[rgb]{0,0,0}\makebox(0,0)[lt]{\lineheight{1.25}\smash{\begin{tabular}[t]{l}$\Sigma$\end{tabular}}}}%
    \put(0,0){\includegraphics[width=\unitlength,page=2]{Model_space.pdf}}%
    \put(0.05039321,0.39489082){\color[rgb]{0,0,0}\makebox(0,0)[lt]{\lineheight{1.25}\smash{\begin{tabular}[t]{l}$\gamma$\end{tabular}}}}%
    \put(0.80513579,0.40896474){\color[rgb]{0,0,0}\makebox(0,0)[lt]{\lineheight{1.25}\smash{\begin{tabular}[t]{l}$\Sigma$\end{tabular}}}}%
    \put(0.68723478,0.44778413){\color[rgb]{0,0,0}\makebox(0,0)[lt]{\lineheight{1.25}\smash{\begin{tabular}[t]{l}$\gamma$\end{tabular}}}}%
  \end{picture}%
\endgroup%

%% file: mybibmini.bib
@article{climenhagaThompson2016,
  title={Unique equilibrium states for flows and homeomorphisms with non-uniform structure},
  author={Climenhaga, Vaughn and Thompson, Daniel J},
  journal={Advances in Mathematics},
  volume={303},
  pages={745--799},
  year={2016},
  publisher={Elsevier}
}

@article{ConstantineLafontThompson2020,
  title={The weak specification property for geodesic flows on {CAT}(-1) spaces},
  author={Constantine, David and Lafont, Jean-Fran{\c{c}}ois and Thompson, Daniel J},
  journal={Groups, Geometry, and Dynamics},
  volume={14},
  number={1},
  pages={297--336},
  year={2020}
}

@article{CapraceMono2009,
  title={Isometry groups of non-positively curved spaces: structure theory},
  author={Caprace, Pierre-Emmanuel and Monod, Nicolas},
  journal={Journal of topology},
  volume={2},
  number={4},
  pages={661--700},
  year={2009},
  publisher={Wiley Online Library}
}

@article{Cavallucci23,
 author = {Cavallucci, Nicola},
 title = {Continuity of critical exponent of quasiconvex-cocompact groups under {Gromov}-{Hausdorff} convergence},
 fjournal = {Ergodic Theory and Dynamical Systems},
 journal = {Ergodic Theory Dyn. Syst.},
 issn = {0143-3857},
 volume = {43},
 number = {4},
 pages = {1189--1221},
 year = {2023},
 language = {English},
 doi = {10.1017/etds.2022.9},
 zbMATH = {7680120},
 Zbl = {1525.53053}
}

@incollection {Gromov-HypGps,
    AUTHOR = {Gromov, Michail},
     TITLE = {Hyperbolic groups},
 BOOKTITLE = {Essays in group theory},
    SERIES = {Math. Sci. Res. Inst. Publ.},
    VOLUME = {8},
     PAGES = {75--263},
 PUBLISHER = {Springer, New York},
      YEAR = {1987},
   MRCLASS = {20F32 (20F06 20F10 22E40 53C20 57R75 58F17)},
  MRNUMBER = {919829},
MRREVIEWER = {Christopher W. Stark},
       DOI = {10.1007/978-1-4613-9586-7_3},
       URL = {https://doi.org/10.1007/978-1-4613-9586-7_3},
}

@article {GromovThurston87,
    AUTHOR = {Gromov, Michail and Thurston, William},
     TITLE = {Pinching constants for hyperbolic manifolds},
   JOURNAL = {Invent. Math.},
  FJOURNAL = {Inventiones Mathematicae},
    VOLUME = {89},
      YEAR = {1987},
    NUMBER = {1},
     PAGES = {1--12},
      ISSN = {0020-9910},
   MRCLASS = {53C20},
  MRNUMBER = {892185},
MRREVIEWER = {Karsten Grove},
       DOI = {10.1007/BF01404671},
}

@misc{ Stocker19,
    author = {Stocker, Arnaud},
    title = {G\'eom\'etrie de certains espaces de courbure n\'egative},
    howpublished = {PhD Thesis, Aix Marseilles Universit\'e, France},
    year = {2019}
}

@book {BridsonHaefliger,
    AUTHOR = {Bridson, Martin R. and Haefliger, Andr\'e},
     TITLE = {Metric spaces of non-positive curvature},
    SERIES = {Grundlehren der mathematischen Wissenschaften [Fundamental
              Principles of Mathematical Sciences]},
    VOLUME = {319},
 PUBLISHER = {Springer-Verlag, Berlin},
      YEAR = {1999},
     PAGES = {xxii+643},
      ISBN = {3-540-64324-9},
   MRCLASS = {53C23 (20F65 53C70 57M07)},
  MRNUMBER = {1744486},
MRREVIEWER = {Athanase\ Papadopoulos},
       DOI = {10.1007/978-3-662-12494-9},
       URL = {https://doi.org/10.1007/978-3-662-12494-9},
}

@article {Roblin2002,
    AUTHOR = {Roblin, Thomas},
     TITLE = {Sur la fonction orbitale des groupes discrets en courbure
              n\'egative},
   JOURNAL = {Ann. Inst. Fourier (Grenoble)},
  FJOURNAL = {Universit\'e{} de Grenoble. Annales de l'Institut Fourier},
    VOLUME = {52},
      YEAR = {2002},
    NUMBER = {1},
     PAGES = {145--151},
      ISSN = {0373-0956,1777-5310},
   MRCLASS = {37C85 (20H10 37C35 37F35)},
  MRNUMBER = {1881574},
MRREVIEWER = {Petra\ Bonfert-Taylor},
       DOI = {10.5802/aif.1880},
       URL = {https://doi.org/10.5802/aif.1880},
}

@article {Cavallucci:ErgodicLimitSet,
    AUTHOR = {Cavallucci, Nicola},
     TITLE = {Bishop-{J}ones' theorem and the ergodic limit set},
   JOURNAL = {Ergodic Theory Dynam. Systems},
  FJOURNAL = {Ergodic Theory and Dynamical Systems},
    VOLUME = {45},
      YEAR = {2025},
    NUMBER = {3},
     PAGES = {704--718},
      ISSN = {0143-3857,1469-4417},
   MRCLASS = {37D40 (37C10 53C23)},
  MRNUMBER = {4860664},
       DOI = {10.1017/etds.2024.49},
       URL = {https://doi.org/10.1017/etds.2024.49},
}

@article {RoblinThesis,
    AUTHOR = {Roblin, Thomas},
     TITLE = {Ergodicit\'e{} et \'equidistribution en courbure n\'egative},
   JOURNAL = {M\'em. Soc. Math. Fr. (N.S.)},
  FJOURNAL = {M\'emoires de la Soci\'et\'e{} Math\'ematique de France.
              Nouvelle S\'erie},
    NUMBER = {95},
      YEAR = {2003},
     PAGES = {vi+96},
      ISSN = {0249-633X,2275-3230},
   MRCLASS = {37D40 (28D10 30F40 37A25 37C40 53C22 53D25 57R30)},
  MRNUMBER = {2057305},
MRREVIEWER = {Marc\ Peign\'e},
       DOI = {10.24033/msmf.408},
       URL = {https://doi.org/10.24033/msmf.408},
}

@article{Cavallucci2024OtalPeigne,
  title={Otal-{P}eigné's {T}heorem for {G}romov-hyperbolic spaces},
  author={Cavallucci, Nicola},
  journal={arXiv:2412.10801},
  year={2024}
}

@article{Bow73,
	title={Topological entropy for noncompact sets},
	author={Bowen, Rufus},
	journal={Transactions of the American Mathematical Society},
	volume={184},
	pages={125--136},
	year={1973}
}

@article{LedrappierWalters1977,
  title={A relativised variational principle for continuous transformations},
  author={Ledrappier, Fran{\c{c}}ois and Walters, Peter},
  journal={Journal of the London Mathematical Society},
  volume={2},
  number={3},
  pages={568--576},
  year={1977},
  publisher={Oxford University Press}
}

@article{Ricks2021,
title = {The unique measure of maximal entropy for a compact rank one locally {CAT}(0) space},
journal = {Discrete and Continuous Dynamical Systems},
volume = {41},
number = {2},
pages = {507-523},
year = {2021},
issn = {1078-0947},
doi = {10.3934/dcds.2020266},
url = {https://www.aimsciences.org/article/id/b5ac9fcf-39b0-4078-a7c8-11cb68db2bd4},
author = {Russell Ricks}
}

@article{Allcock2000,
  title={Asphericity of moduli spaces via curvature},
  author={Allcock, Daniel},
  journal={Journal of Differential Geometry},
  volume={55},
  number={3},
  pages={441--451},
  year={2000},
  publisher={Lehigh University}
}

@book{Walters2000,
  title={An introduction to ergodic theory},
  author={Walters, Peter},
  volume={79},
  year={2000},
  publisher={Springer Science \& Business Media}
}

@inproceedings{Chernavskii63,
  title={Finitely multiple open mappings of manifolds},
  author={Chernavskii, Aleksei V.},
  booktitle={Doklady Akademii Nauk},
  volume={151},
  number={1},
  pages={69--72},
  year={1963},
  organization={Russian Academy of Sciences}
}

@article{Vaisala67,
  title={Discrete open mappings on manifolds},
  author={V{\"a}is{\"a}l{\"a}, Jussi},
  journal={Annales Fennici Mathematici},
  number={392},
  year={1967}
}

@article{Alexander1920,
author = {James W. Alexander},
title = {{Note on Riemann spaces}},
volume = {26},
journal = {Bulletin of the American Mathematical Society},
number = {8},
publisher = {American Mathematical Society},
pages = {370 -- 372},
year = {1920},
}

@inproceedings{Fox1957,
  title={Covering spaces with singularities},
  author={Fox, Ralph H.},
  booktitle={Algebraic Geometry and Topology: A Symposium in Honor of S. Lefschetz},
  pages={243--257},
  year={1957}
}

@article{HeinonenRickman2002,
author = {Juha Heinonen and Seppo Rickman},
title = {{Geometric branched covers between generalized manifolds}},
volume = {113},
journal = {Duke Mathematical Journal},
number = {3},
publisher = {Duke University Press},
pages = {465 -- 529},
year = {2002},
doi = {10.1215/S0012-7094-02-11333-7},
URL = {https://doi.org/10.1215/S0012-7094-02-11333-7}
}

@article{BersteinEdmonds1979,
  title={On the construction of branched coverings of low-dimensional manifolds},
  author={Berstein, Israel and Edmonds, Allan L.},
  journal={Transactions of the American Mathematical Society},
  volume={247},
  pages={87--124},
  year={1979}
}

@article{Piergallini1995,
  title={Four-manifolds as 4-fold branched covers of $\mathbb{S}^4$},
  author={Piergallini, Riccardo},
  journal={Topology},
  volume={34},
  number={3},
  pages={497--508},
  year={1995},
  publisher={Oxford; New York: Pergamon Press, 1962-}
}

@article{Hilden1974,
author = {Hilden, Hugh},
year = {1974},
month = {11},
pages = {},
title = {Every closed orientable 3-manifold is a 3-fold branched covering space of $\mathbb{S}^3$},
volume = {80},
journal = {Bulletin of The American Mathematical Society},
doi = {10.1090/S0002-9904-1974-13699-2}
}

@article{Montesinos1974,
author = {Jos{\'e} M. Montesinos},
title = {{A representation of closed, orientable 3-manifolds as 3-fold branched coverings of $\mathbb{S}^3$}},
volume = {80},
journal = {Bulletin of the American Mathematical Society},
number = {5},
publisher = {American Mathematical Society},
pages = {845 -- 846},
year = {1974},
}

@article{Hamenstadt2025,
  title={The geometry of branched covers of hyperbolic manifolds},
  author={Hamenst{\"a}dt, Ursula},
  year={2025}
}

@article{Deraux2005,
  title={A negatively curved {K}{\"a}hler threefold not covered by the ball},
  author={Deraux, Martin},
  journal={Inventiones mathematicae},
  volume={160},
  number={3},
  pages={501--525},
  year={2005},
  publisher={Springer}
}

@article{MostowSiu1980,
  title={A compact {K}{\"a}hler surface of negative curvature not covered by the ball},
  author={Mostow, George D. and Siu, Yum-Tong},
  journal={Annals of Mathematics},
  volume={112},
  number={2},
  pages={321--360},
  year={1980},
  publisher={JSTOR}
}

@article{CharneyDavis1993,
  title={Singular metrics of nonpositive curvature on branched covers of {R}iemannian manifolds},
  author={Charney, Ruth and Davis, Michael},
  journal={American Journal of Mathematics},
  volume={115},
  number={5},
  pages={929--1009},
  year={1993},
  publisher={JSTOR}
}

@article{Bowen1972,
  title={Entropy-expansive maps},
  author={Bowen, Rufus},
  journal={Transactions of the American Mathematical Society},
  volume={164},
  pages={323--331},
  year={1972}
}

@article{Jenkinson2006,
  title={Ergodic optimization},
  author={Jenkinson, Oliver},
  journal={Discrete and Continuous Dynamical Systems},
  volume={15},
  number={1},
  pages={197},
  year={2006},
  publisher={AIMS PRESS}
}

@article{LedrappierLim2010,
title = {Volume entropy of hyperbolic buildings},
journal = {Journal of Modern Dynamics},
volume = {4},
number = {1},
pages = {139-165},
year = {2010},
issn = {1930-5311},
doi = {10.3934/jmd.2010.4.139},
url = {https://www.aimsciences.org/article/id/ca505325-a20c-4d1e-9a07-538c2c975277},
author = {François Ledrappier and Seonhee Lim}
}

@article{Lim2008,
  title={Minimal volume entropy for graphs},
  author={Lim, Seonhee},
  journal={Transactions of the American Mathematical Society},
  volume={360},
  number={10},
  pages={5089--5100},
  year={2008}
}

@article{FinePremoselli2020,
  title={Examples of compact {E}instein four-manifolds with negative curvature},
  author={Fine, Joel and Premoselli, Bruno},
  journal={Journal of the American Mathematical Society},
  volume={33},
  number={4},
  pages={991--1038},
  year={2020}
}

@article{HamenstadtJackel2024,
  title={Negatively curved {E}instein metrics on {G}romov-{T}hurston manifolds},
  author={Hamenst{\"a}dt, Ursula and J{\"a}ckel, Frieder},
  journal={arXiv:2411.12956},
  year={2024}
}

@article{BCGS2017,
  title={Curvature-free Margulis lemma for Gromov-hyperbolic spaces},
  author={Besson, G{\'e}rard and Courtois, Gilles and Gallot, Sylvestre and Sambusetti, Andrea},
  journal={arXiv preprint arXiv:1712.08386},
  year={2017}
}

@article{CavallucciSambusetti2024,
  title={Discrete groups of packed, non-positively curved, {G}romov hyperbolic metric spaces},
  author={Cavallucci, Nicola and Sambusetti, Andrea},
  journal={Geometriae Dedicata},
  volume={218},
  number={2},
  pages={36},
  year={2024},
  publisher={Springer}
}

@article{RamosCuevas2013,
  title={Convexity is a local property in {CAT}$(\kappa)$ spaces},
  author={Ramos-Cuevas, Carlos},
  journal={arXiv preprint arXiv:1304.4147},
  year={2013}
}

@article{Engelking1995,
  title={Theory of dimensions, finite and infinite},
  author={Engelking, Ryszard},
  journal={Sigma Series in Pure Mathematics},
  volume={10},
  year={1995},
  publisher={Heldermann Verlag}
}
